\documentclass[reqno,11pt]{amsart}
\usepackage{amsmath,amssymb,amsthm,graphicx,mathrsfs,url,bbm,mathtools,cite}
\usepackage[usenames,dvipsnames]{color}
\usepackage[colorlinks=true,linkcolor=black,citecolor=blue]{hyperref}
\usepackage{amsxtra}
\usepackage{wasysym} 
\usepackage{graphicx}
\usepackage{scalerel}

\newtheorem{theorem}{Theorem}
\newtheorem{proposition}{Proposition}[section]

\newtheorem{lemma}[proposition]{Lemma}
\newtheorem{corollary}[proposition]{Corollary}
\newtheorem{remark}{Remark}

\numberwithin{equation}{section}

\DeclareMathOperator{\id}{id}
\DeclareMathOperator{\exterior}{ext}
\DeclareMathOperator{\interior}{int}

\DeclareMathOperator{\rad}{rad}
\DeclareMathOperator{\loc}{loc}
\def\Re{\operatorname{Re}}
\def\Im{\operatorname{Im}}

\title[]{On self-Similar blow up for energy supercritical semilinear wave equation}

\author{Jihoi Kim}
\address[Jihoi Kim]{University of Cambridge, United Kingdom.}
\email{rk614@cam.ac.uk}

\begin{document}
\begin{abstract}
We analyse the energy supercritical semilinear wave equation 
$$\Phi_{tt}-\Delta\Phi-|\Phi|^{p-1} \Phi=0$$
in $\mathbb R^d$ space. We first prove in a suitable regime of parameters the existence of a countable family of self-similar profiles which bifurcate from the soliton solution. We then prove the non-radial finite codimensional stability of these profiles by adapting the functional setting of \cite{MRRS}.
\end{abstract}

\maketitle

\textbf{Keywords:} Semi-linear wave equation, Self-similar solution, Blow up, Focusing, Energy super-critical, Finite codimensional stability

\section{Introduction}

\subsection{Setting of the problem} We consider the semi-linear focusing wave equation
\begin{equation}
\label{eq: Nonlinear Wave}
\begin{cases}\Phi_{tt}-\Delta\Phi-|\Phi|^{p-1}\Phi=0,\\
\Phi\big|_{t=0}=\Phi_0,\ \ \partial_t\Phi\big|_{t=0}=\Phi_1,\end{cases}
\quad (t,x) \in \mathbb R\times \mathbb R^d.
\end{equation}
This model admits a scaling invariance: if $\Phi(t,x)$ is a solution, then so is
$$
\Phi_\lambda(t,x)=\lambda^\alpha \Phi(\lambda t,\lambda x),\quad \lambda>0,\quad \alpha:=\frac{2}{p-1}.
$$
This transformation is an isometry on the homogeneous Sobolev space with critical exponent:
$$
\Vert \Phi_\lambda(t,\cdot)\Vert_{\dot H^{s_c}}=\Vert \Phi(t,\cdot)\Vert_{\dot H^{s_c}},\quad s_c:=\frac{d}{2}-\frac{2}{p-1}.
$$
In this paper, we focus on the energy super-critical case where space dimension $d\ge 3$ and $s_c>1$. The question we address is the existence and stability of self-similar blow up regimes.\\

The problem of singularity formation in semi-linear dispersive equations has attracted a considerable attention in the last fifty years both in the physics and mathematics communities, with a substantial acceleration in the last twenty years. The series of works by Merle and Zaag \cite{MZ1,MZ2, MZ3} give a detailed description of singularity formation mechanims in energy sub-critical ranges $s_c<1$ where the leading order expected behaviour is the self-similar ODE blow up. In the energy critical range, the situation is very different and new so called type II blow up scenario were discovered  in the setting of the energy-critical wave and Schr\"odinger map \cite{KST, RodSter, RaphRod, MRRsch} and semi-linear problems \cite{HR}. The soliton solution $$\left|\begin{array}{l}
\Delta Q+Q^p=0\\
\lim_{|x|\to +\infty} Q(x)=0
\end{array}\right.
$$ plays a distinguished role in the analysis as it serves as blow up profile for the main part of the singular bubble. The stability analysis of the obtained type II blow up bubbles then relies on delicate energy estimates built on repulsivity properties of the linearized self-similar flow near the soliton.\\

In the energy super-critical range, and in analogy with the pioneering results for the non-linear heat equation \cite{HV, MaMe1, MaMe2,CMRcylind}, the situation is quite different. Solitonic type II bubbles still exist but only for $p>p_{JL}$ large enough, \cite{MRRnls, C} where Joseph-Lundgren exponent $p_{JL}$ is defined in \eqref{eq: P_JL}, and a new type of self-similar blow up arises, different from the ODE blow up, as governed by explicit stationary self-similar solutions. More explicitely, the ansatz 
\begin{equation}
\label{eq: Ansatz}
\Phi(t,r)=(T-t)^{-\alpha}u(\rho),\quad \rho:=|y|,\quad y:=\frac{x}{T-t}
\end{equation}
maps \eqref{eq: Nonlinear Wave} onto the radially symmetric non-linear ODE
\begin{equation}
\label{eq: Self-similar Profile}
(1-\rho^2)u''+\bigg[\frac{d-1}{\rho}-2(1+\alpha)\rho\bigg]u'-\alpha(1+\alpha)u+|u|^{p-1}u=0.
\end{equation}
The program of existence of self-similar dynamics then becomes a two step analysis. First construct solutions to the non-linear ODE \eqref{eq: Self-similar Profile} with regularity at the origin and  good boundary condition at $+\infty$ $$u(\rho)\sim \frac{c}{\rho^{\frac 2{p-1}}}\ \ \mbox{as}\ \ \rho\to+\infty.$$ These solutions however never belong to the energy space in which \eqref{eq: Nonlinear Wave} is naturally well posed, hence a global in space stability analysis is required to ensure that a suitable truncation of these profiles can be stabilized, at least for a finite dimensional manifold of initial data. This second step relies on both a linear and non-linear analysis of the linearized flow around self-similar profiles.\\

Let us stress that the program of constructing self-similar solutions and showing their finite codimensional stability goes way beyond the scope of non-linear wave equations, and is in particular a very active field of research in fluid related problems, \cite{MRRS}, hence the need for robust analytic methods.


\subsection{Existence of self-similar profiles}


The esistence of self-similar profiles with suitable boundary conditions is in general a delicate problem, and here we take advantage of symmetry reductions to transform the problem into the non-linear ODE problem \eqref{eq: Self-similar Profile} which is of shooting type. However the understanding of solutions is non trivial, and relies on the derivation of explicit monotonicity formulas to follow the non-linear flow. The existence of a countable family of solutions to \eqref{eq: Self-similar Profile} is obtained in \cite{BMW,DD} in the expected range
\begin{equation}
\label{vneoneneoivne}
1<s_c<\frac{3}{2}\Leftrightarrow 1+\frac{4}{d-2}<p<1+\frac{4}{d-3}.
\end{equation} 
Our first result in this paper describes the asymptotic behaviour of the branch of solutions to \eqref{eq: Self-similar Profile} leading to an explicit sequence of solutions that concentrate at the origin to a soliton  profile. Our approach generalizes the analogous result for the semi-linear heat equation implemented in \cite{BB,CRS}. The advantage of this method is its robustness as it can be applied to more complicated problems, see e.g. \cite{BMR}, and also allows for a full description of the profile in space.

 \begin{theorem}[Existence and asymptotes of excited self-similar solutions]
\label{theo: Result 1}
Assume \eqref{vneoneneoivne}. There exists $N\in\mathbb N$ such that for all $n\ge N$, there exists a smooth radially symmetric self-similar solution to equation \eqref{eq: Nonlinear Wave} such that for
$$
\Lambda = \alpha+y\cdot \nabla,
$$
$\Lambda u_n$ vanishes exactly $n$ times on $(0,\infty)$. Moreover:\\\\
\emph{(i) Behaviour at infinity:} as $n\rightarrow \infty$ the solutions $u_n$ converge to the explicit singular solution 
$$
u_\infty(\rho):=b_\infty\rho^{-\alpha}, \quad b_\infty:=(\alpha(d-2-\alpha))^{\frac{2}{\alpha}}
$$
to \eqref{eq: Self-similar Profile} in the following sense: for all $\rho_0>0$,
$$
\lim_{n\rightarrow \infty }\sup_{\rho\ge\rho_0}(1+\rho^\alpha)|u_n(\rho)-u_\infty(\rho)|=0
$$
\emph{(ii) Behaviour at the origin:} There exists $0<\rho_0\ll1 $ and $\mu_n\rightarrow 0$ such that
$$
\lim_{n\rightarrow \infty }\sup_{\rho\le\rho_0}\bigg|u_n(\rho)-\mu_n^{-\alpha}Q\bigg(\frac{\rho}{\mu_n}\bigg)\bigg|=0
$$
where the soliton $Q$ is the unique non trivial radially symmetric solution to 
$$
\Delta Q+Q^p=0,\quad Q(\rho)=b_\infty\rho^{-\alpha}+\mathcal O_{\rho\rightarrow\infty } (\rho^{1-\frac{d}{2}}).
$$
\end{theorem}


\subsection{non-linear stability}


The non-linear stability of self-similar blow up is a classical problem. It has been addressed for the energy super-critical non-linear heat equation in \cite{CRS} and the stability proof relies on a two steps argument: linear exponential decay in time for local in space norms around the singularity which in the parabolic case rely on self-adjoint spectral methods, and then propagation of space time decay using energy estimates which provide strong enough control to close the non-linear terms.\\

In the setting of energy super-critical non-linear wave equations, a non-self adjoint spectral method is developped in the pioneering works by Donninger and Sch\"orkhuber for wave maps \cite{DS}, see also \cite{GS} and references therein, but decay is restricted to the light cone only $|x|<T-t$ and hence does not allow the full control of the solution. In \cite{MRRS}, a full linear and non-linear analysis is performed for the stability study of quasilinear self-similar blow up. Our claim in this paper is that this robust framework can be adapted to \eqref{eq: Nonlinear Wave} to show the stability of {\rm any} self-similar profile, modulo a finite number of unstable modes. We moreover claim that full non-radial perturbations can be considered as opposed to previous works which restrict to data with radial symmetries.

\begin{theorem}[Non-linear stability]
\label{theo: Result 2}
Let $d=3$ and $u_n$ be the self-similar profiles constructed in \emph{Theorem \ref{theo: Result 1}} with corresponding initial data $(\Phi(0),\Phi_t(0))=P_n$ for
\begin{equation}
\label{eq: Profile}
P_n:=\left(\frac{1}{T^\alpha} u_n\left(\frac{r}{T}\right),\frac{1}{T^{\alpha+1}}\Lambda u_n\left(\frac{r}{T}\right)\right).
\end{equation}
For $T\ll 1$, there exists a finite codimensional Lipschitz manifold of smooth initial data \footnote{see comments on the results below for the precise definition of the Lipschitz manifold} $(\Phi(0),\Phi_t(0))\in\cap_{m\ge0} H^m(\mathbb R^3,\mathbb R^2)$ such that in the neighbourhood of $P_n$, the corresponding solution $(\Phi,\Phi_t)$ to \eqref{eq: Nonlinear Wave} develops a Type I blow up at time $T$ at the origin i.e. as $t\rightarrow T$,
$$
\Vert \Phi(t)\Vert_{L^\infty}\sim(T-t)^{-\alpha}.
$$
More precisely, there holds the decomposition:
$$
(\Phi,\Phi_t) = \left(\frac{1}{(T-t)^\alpha} (u_n+\Psi)\left(t,\frac{r}{T-t}\right),\frac{1}{(T-t)^{\alpha+1}} (\Lambda u_n+\Omega)\left(t,\frac{r}{T-t}\right)\right).
$$
with the asymptotic behaviour in the limit $t\rightarrow T$:\\

\noindent\emph{1. Subcritical norms}
\begin{equation}
\label{eq: Subcritical Sobolev}
\limsup_{t\rightarrow T} \Vert \Phi\Vert_{\dot H^s}^2+\mathbbm 1_{s\ge1}\Vert \Phi_t\Vert_{\dot H^{s-1}}^2<\infty\quad \text{for}\quad 0\le s<s_c
\end{equation}
\emph{2. Critical norm}
\begin{equation}
\label{eq: Critical Sobolev}
(\Vert \Phi\Vert_{\dot H^{s_c}}^2, \Vert \Phi_t\Vert_{\dot H^{s_c-1}}^2)=(c_n,d_n)(1+o_{t\rightarrow T}(1))|\log(T-t)|,
\end{equation}
\emph{3. Supercritical norms} 
\begin{equation}
\label{eq: Supercritical Sobolev}
\lim_{t\rightarrow T}\Vert \Psi\Vert_{\dot H^s}^2+\Vert \Omega\Vert_{\dot H^{s-1}}^2=0 \quad\text{for}\quad s_c<s\le 2.
\end{equation}
\end{theorem}

{\bf Comments on the results}\\\\
\emph{1. Stability of the self-similar blow up.} As in \cite{MRRS}, a key step in the analysis is to realize the linearized operator close to a self-similar profile
as a compact perturbation of a maximal dissipative operator in a {\em global in space weighted Sobolev space with supercritical regularity}. Using sufficient regularity and propagating additional weighted energy estimates then allows to close bound for the nonlinear terms. Hence the counting of the exact number of instability is reduced to an explicit spectral problem.
\\\\
\emph{2. Restriction on the parameters.} Note that in Theorem \ref{theo: Result 2}, there is a further restriction on the parameters: 
$$
d=3 \Longleftrightarrow p>5
$$
This is due to the poor regularity of the nonlinearity. In particular, the nonlinearity $\Phi\mapsto |\Phi|^{p-1}\Phi$ has $C^{\lfloor p\rfloor}$ regularity for $p\notin 2\mathbb N+1$. The role of this constraints is to allow us to take $k\le \lfloor p\rfloor-1$ derivatives when closing the nonlinear estimates. We are only able to take one less derivative than the regularity of $|\Phi|^{p-1}\Phi$ since the Lipshcitz dependence of the nonlinear term on $\Phi$ in the weighted $H^k$ space means we lose one more power in the nonlinear term (see Lemma \ref{lem: Lipschitz}). Furthermore, we require $k\ge\frac{d}{2}$ by Sobolev embedding which is what we use to bound the nonlinear term. Since \eqref{vneoneneoivne} implies that $p-1\ll1$ for large values of $d$, the codimensional stability result cannot be generalised into higher dimensions. Also, note that the constraint $p+1>s_c$ which is implied by \eqref{vneoneneoivne} is essential in the development of the local theory (see \cite{GX} for the related well-posedness result).
\\\\
\emph{3. Manifold structure of the initial data.} Let 
$$
B_\varepsilon^{\mathbb H}=\{X\ | \ \Vert X\Vert_{\mathbb H}<\varepsilon\},\quad B_\delta^H=\{X\ | \ \Vert X\Vert_H<\delta\}
$$
with $\varepsilon,\ \delta\ll1$ where 
$$
\mathbb H=H_4\times H_3
$$
where the spaces $H_k$ are defined in Section \ref{sec: Notation} and $H$ is the weighted $W^{k,\infty}$-space defined in the Proof \ref{proof: Result 2} and consider the self-similar profile and the dampened profile in self-similar variables:
\begin{equation}
\label{eq: Profiles and Dampened}
P_n=(u_n(\rho),\Lambda u_n(\rho)),\quad P_n^D=(\eta(e^{-s_0}\rho)u_n(\rho), \eta(e^{-s_0}\rho)\Lambda u_n(\rho)).
\end{equation}
where $\eta$ is a smooth, rapidly decaying function defined in \eqref{eq: eta}. Profiles are dampened to acheive finite energy. We then, construct the finite codimensional manifold of initial data in Theorem \ref{theo: Result 2} as follows: consider
$$
\mathbb H=U\oplus V
$$
a direct sum decomposition into subspaces $U$ and $V$ stable and unstable under the semigroup action of the linearized operator with $\operatorname{dim}V<\infty$. Then consider the Lipschitz map $\Phi:B_\varepsilon^{\mathbb H}\cap (B_\delta^H+P_n^D-P_n)\cap U\rightarrow V$ obtained by solving a Brouwer type fixed point problem and a linear map $\Xi: V\rightarrow U$ on the finite dimensional space $V$ such that
$$
\operatorname{Id}+\,\Xi: V\rightarrow (B_\delta^H+P_n^D-P_n).
$$
Then, the finite codimensional manifold can be realized as
$$
\mathcal M=P_n+\Big(\operatorname{Id}+(\operatorname{Id}+\,\Xi) \circ \Phi\Big)\left(B_\varepsilon^{\mathbb H}\cap(B_\delta^H+P_n^D-P_n)\cap U\right)\subset H+P_n^D.
$$
Note that the modifier $\Xi$ is there to ensure that our initial data does not leave the neighbourhood $H+P_n^D$ which is essential in obtaining finite energy initial data. Also, in Lemma \ref{lem: Lipschitz}, it is proved that $\Phi$ is a Lipschitz map with respect to the topology of $\mathbb H$. Similar properties of the stable manifold is proved in \cite{GS}, \cite{CRS}, \cite{C}.

\subsection*{Aknowledgements} The author is endebted to his PhD supervisor P. Rapha\"el for stimulating discussions and guidance on this work. This work is supported by the UKRI, ERC advanced grant SWAT and Cambridge Commonwealth Trust.

\section{Notations}
\label{sec: Notation}
Let us introduce some notations before we start. We write for the generator of scaling operator $\Lambda$:
$$
\Lambda = \alpha+y\cdot \nabla,\quad \alpha:=\frac{2}{p-1}.
$$
We will denote by $(t,x)$ the original variables and $(s,y)$ for the self-similar variables:
$$
\quad s=-\log(T-t),\quad y= \frac{x}{T-t}
$$ 
and denote their modulus:
$$
r=|x|, \quad \rho=|y|.
$$
We also write
$$
\nabla^{j}=\begin{cases}
\Delta^i & j=2i,\\
\nabla \Delta^i & j=2i+1,
\end{cases}
$$
and for scalar (or vector) valued functions $f$, $g$ on $\mathbb R^d$,
$$
(f,g)= \int_{\mathbb R^d} f\cdot g\  dy.
$$
Now fix $d=3$. Let $\chi\in C_c^\infty(\mathbb R^3,[0,\infty))$ be a radial smooth cut-off function with 
$$
\chi(y)=\begin{cases} 1 & |y|\le 1,\\ 0 & |y|\ge 2.\end{cases}
$$
For $k\in \mathbb N$, denote by $H_k$ the completion of $C_c^\infty(\mathbb R^3)$ with respect to the norm induced by the inner product 
$$
(\Psi,\tilde \Psi)_{H_k}=(\nabla^k\Psi,\nabla^k\tilde \Psi)+ \int_{\mathbb R^3}\chi \Psi \tilde \Psi  dy.
$$

\section{Construction of exterior solutions}

Our aim in this section is to construct a family of outer solutions to the self-similar equation \eqref{eq: Self-similar Profile}. The key is that the outer spectral problem, including the singularity through the renormalized light cone $\rho=1$, is explicit.\\

We introduce relevant notations for this section. \\\\
\underline{\emph{Linearized operator}}. Recall the generator of scaling operator $\Lambda$:
$$
\Lambda = \alpha+y\cdot \nabla.
$$
Introduce the linearized operator
\begin{equation}
\label{eq: L_infty}
\mathcal L_\infty=(1-\rho^2)\frac{d^2}{d\rho^2}+\bigg[\frac{d-1}{\rho}-2(1+\alpha)\rho\bigg]\frac{d}{d\rho}-\alpha(1+\alpha)+p\alpha(d-2-\alpha)\rho^{-2}.
\end{equation}
for \eqref{eq: Self-similar Profile} near the singular solution $u=u_\infty$ where we recall
$$
u_\infty(\rho)=b_\infty\rho^{-\alpha}, \quad b_\infty=(\alpha(d-2-\alpha))^{\frac{2}{\alpha}}.
$$
Also, let 
$$
\quad\omega=\sqrt{pb_\infty^{p-1}-\frac{(d-2)^2}{4}}.
$$
Note that $\omega\in \mathbb R$ if
\begin{equation}
\label{eq: P_JL}
1+\frac{4}{d-2}<p<p_{JL}:=\begin{cases} \infty &\text{ for } d\le 10,\\ 1+\frac{4}{d-4-2 \sqrt{d-1}} & \text{ for } d\ge 11\end{cases}
\end{equation}
with sufficient condition being $1<s_c<\frac{3}{2}$. $p_{JL}$ is known as the Joseph-Lundgren exponent.\\\\
\underline{\emph{Hypergeometric functions}}. We denote by $_2F_1$ the Gauss hypergeometric functions:
\begin{equation}
\label{eq: Hypergeometric Functions}
_2F_1(a,b,c; z)=\sum_{n=0}^\infty \frac{(a)_n(b)_n}{(c)_n}\frac{z^n}{n!}
\end{equation}
where $(a)_n=a(a+1)\cdots (a+n-1)$.

\subsection{Fundamental solutions and exterior resolvent}

Recall the definition of linearized operator $\mathcal L_\infty$ above. In this section, we compute the fundamental solutions of the linearized operator $\mathcal L_\infty$ and use calculus of variation to invert $\mathcal L_\infty$ in a suitable space of functions.

\begin{lemma}[Fundamental solutions of $\mathcal L_\infty$]
\label{lem: Fundamental Solutions}
(i) Interior solution: In the region $\rho\in(0,1)$,  the homogeneous equation $\mathcal L_\infty (\psi)=0$ has a basis of solutions
\begin{equation}
\label{eq: Left Fundamental Solutions}
\begin{aligned}
\psi_1^L&=\Re\bigg[\rho^{1-\frac{d}{2}+i\omega}\, _2F_1\bigg(\frac{1-s_c+i\omega}{2},\frac{2-s_c+i\omega}{2},1+i\omega,\rho^2\bigg)\bigg]\\
\psi_2^L&=\Im\bigg[\rho^{1-\frac{d}{2}+i\omega}\, _2F_1\bigg(\frac{1-s_c+i\omega}{2},\frac{2-s_c+i\omega}{2},1+i\omega,\rho^2\bigg)\bigg].
\end{aligned}
\end{equation}
(ii) Exterior solution: In the region $\rho\in(1,\infty)$,  the homogeneous equation $\mathcal L_\infty (\psi)=0$ has a basis of solutions
\begin{equation}
\label{eq: Right Fundamental Solutions}
\begin{aligned}
\psi_1^R&=\rho^{-\alpha-1}\,_2F_1\bigg(\frac{2-s_c-i\omega}{2},\frac{2-s_c+i\omega}{2},\frac{3}{2},\rho^{-2}\bigg)\\
\psi_2^R&=\rho^{-\alpha}\, _2F_1\bigg(\frac{1-s_c-i\omega}{2},\frac{1-s_c+i\omega}{2},\frac{1}{2}, \rho^{-2}\bigg).
\end{aligned}
\end{equation}
\end{lemma}
\begin{proof}
For $\rho\in (0,1)$, consider solutions of the form $\psi=\rho^\gamma\sum_{n=0}^\infty a_n \rho^n$ for $(a_n)_{n=0}^\infty$ bounded sequence in $\mathbb R$ with $a_0\neq 0$ so the sum is absolutely convergent in $(0,1)$. Then 
$$
\begin{aligned}
&\mathcal L_\infty(\psi)=[\gamma(\gamma+d-2)+pb_\infty^{p-1}]a_0\rho^{\gamma-2}+[(\gamma+1)(\gamma+d-1)+pb_\infty^{p-1}]a_1\rho^{\gamma-1}\\
+&\sum_{n=0}^\infty\bigg\{\big[(\gamma+n+2)(\gamma+n+d)+pb_\infty^{p-1}\big]a_{n+2}-\big[(\gamma+n)(\gamma+n+1+2\alpha)+\alpha(1+\alpha)\big]a_n\bigg\}\rho^{\gamma+n}\\
\end{aligned}
$$
Equating first two terms to $0$, we infer $\gamma=1-\frac{d}{2}\pm i\omega$ and $a_1=0$. Equating higher order terms to $0$,
$$
a_{n+2}=\frac{(\gamma+n+\alpha)(\gamma+n+1+\alpha)}{(\gamma+n+\tfrac{d}{2}+1+i\omega)(\gamma+n+\tfrac{d}{2}+1-i\omega)} a_n.
$$
The cases $\gamma=1-\frac{d}{2}+i\omega$ and $1-\frac{d}{2}-i\omega$ give rise to complex conjugate solutions. Thus, real and imaginary parts of the complex solution satisfying the recursion relation relation above:
$$
\rho^{1-\frac{d}{2}+i\omega}\, _2F_1\bigg(\frac{1-s_c+i\omega}{2},\frac{2-s_c+i\omega}{2},1+i\omega,\rho^2\bigg)
$$
yields two linearly independent real solutions. In the region $(1,\infty)$, consider solutions of the form $\psi=\rho^{-\gamma}\sum_{n=0}^\infty a_n\rho^{-n}$ and proceed as in the region $(0,1)$.
\end{proof}

We now investigate the regularity of the fundamental solutions at the singular point $\rho=1$. First, we recall some results on the singular ODEs.\\

\begin{proposition}[Solutions to singular ODEs, \cite{V}]
\label{prop: Singular ODE}
Let $f\in C^m([0,T],\mathbb R^n)$, $A\in C^m([0,T],\mathbb R^{n\times n})$ for an $m\ge 1$, $m>\max_{\lambda_k\in \sigma(A(0))}\operatorname{Re}(\lambda_k)$ and $1\le l\le m$,
$$
\sigma(A(0))\cap\{l,l+1,\cdots\}=\emptyset.
$$
For $u_0^{0},\cdots,u_0^{(l-1)}\in\mathbb R^m$ such that
\begin{equation}
\label{eq: ODE Conditions}
(kI-A(0))u_0^{(k)}=f^{(k)}(0)+\sum_{j=0}^{k-1}\begin{pmatrix} k\\j\end{pmatrix}A^{(k-j)}(0)u_0^{(j)},\quad k=0,\cdots,l-1
\end{equation}
holds, there exists a unique solution $u\in C^m([0,T],\mathbb R^n)$ of the problem
$$
tu'(t)=A(t)u(t)+f(t),\quad 0<t\le T,\quad u^{(j)}(0)=u_0^{(j)},\quad j=0,\cdots,l-1.
$$
\end{proposition}

\begin{corollary}
There exists unique $\psi_1\in C^1((0,\infty))$ solution to $\mathcal L_\infty (\psi)=0$ with $\psi(1)=1$. Moreover, $\psi_1$ is smooth.
\end{corollary}
\begin{proof}
We write $\mathcal L_\infty(\psi)=0$ in the form required by Proposition \ref{prop: Singular ODE} so for $(\Psi_1,\Psi_2)=(\psi,\partial_\rho\psi)$,
$$
\begin{cases}
(\rho-1)\partial_\rho\Psi_1=(\rho-1)\Psi_2\\
(\rho-1)\partial_\rho\Psi_2=\frac{1}{1+\rho}\left[\frac{p\alpha(d-2-\alpha)}{\rho^2}-\alpha(1+\alpha)\right]\Psi_1+\frac{1}{1+\rho}\left[\frac{d-1}{\rho}-2(1+\alpha)\rho\right]\Psi_2.
\end{cases}
$$
Hence, we can write
$$
(\rho-1)\partial_\rho
\begin{pmatrix}
\Psi_1\\ \Psi_2
\end{pmatrix}
=A(\rho)
\begin{pmatrix}
\Psi_1\\ \Psi_2
\end{pmatrix},
\quad A(0)=
\begin{pmatrix}
0 & 0\\ 
c(\alpha) & s_c-\frac{3}{2}
\end{pmatrix}
$$
for $A$ smooth in $(0,\infty)$. Then since $\sigma(A(0))=\{s_c-\frac{3}{2},0\}$, by Proposition \ref{prop: Singular ODE}, we infer for $a\in\mathbb R$, there exists unique $\psi_a\in C^1((0,\infty))$ solving $\mathcal L_\infty(\psi_a)=0$ with
$$
(\psi_a(0),\psi_a'(0))=(a,0)
$$
and in fact, $\psi_a\in C^\infty((0,\infty))$ so done by setting $a=1$.
\end{proof}

For $0<\rho_0<1$, define the spaces of functions on which we invert our linearized operator $\mathcal L_\infty$:
\begin{equation}
\label{eq: X_rho_0}
\begin{aligned}
X_{\rho_0}&=\bigg\{ w:(\rho_0,\infty)\rightarrow \mathbb R\,\bigg|\, \Vert w\Vert_{X_{\rho_0}}:=\sup_{\rho_0\le\rho\le1}\rho^{\frac{d}{2}-1} |w|+\sup_{\rho\ge1}\rho^{\alpha+1}|w|<\infty\bigg\},\\
Y_{\rho_0}&=\bigg\{ w:(\rho_0,\infty)\rightarrow \mathbb R\,\bigg|\, \Vert w\Vert_{Y_{\rho_0}}:=\int_{\rho_0}^1 \rho^{\frac{d}{2}}|1-\rho|^{\frac{1}{2}-s_c}|w|\,d\rho+\int_{1}^\infty \rho^{\frac{d-1}{2}}|1-\rho|^{\frac{1}{2}-s_c}|w|\,d\rho<\infty\bigg\}.
\end{aligned}
\end{equation}

\begin{proposition}[Exterior resolvent]
\label{prop: Exterior Resolvent}\emph{(i) Basis of fundamental solutions:} There exists $\psi_2$ given by
\begin{equation}
\label{eq: psi_2}
\psi_2:=
\begin{cases}
c_1\psi_1^L & \text{ if }\rho\in(0,1)\\
c_2\psi_1^R & \text{ if }\rho\in(1,\infty).
\end{cases}
\end{equation}
for some $c_i\in \mathbb R$ which is linearly independent of the smooth homogeneous solution $\psi_1$ found in previous lemma and with the Wronskian given by
\begin{equation}
\label{eq: Wronskian}
W:=\psi_1'\psi_2-\psi_2'\psi_1=\rho^{1-d}|1-\rho^2|^{s_c-\frac{3}{2}}.
\end{equation}
The fundamental solutions have asymptotic behaviours:
\begin{equation}
\label{eq: Asymptote of psi_1 and psi_2}
\psi_i\propto \rho^{1-\frac{d}{2}}\sin(\omega \log \rho+\delta_i)\Big[1+\mathcal O_{\rho\to0} (\rho^2)\Big]
\end{equation}
and
\begin{equation}
\label{eq: Asymptote of Lambda psi_1}
\rho^{-1}\psi_1,\ \psi_2,\  \Lambda\psi_1\propto \rho^{-\alpha-1}\Big[1+\mathcal O_{\rho\to\infty} (\rho^{-1})\Big]
\end{equation}
for some $\delta_i\in \mathbb R$.\\\\
\emph{(ii) Continuity of the resolvent:}
There exists a bounded linear operator $\mathcal T:Y_{\rho_0}\rightarrow X_{\rho_0}$ such that $\mathcal L_\infty \circ \mathcal T=\id_{Y_{\rho_0}}$ given by
\begin{equation}
\label{eq: Exterior Resolvent}
\mathcal T(f)=\psi_1\int_{\rho}^\infty \frac{f\psi_2}{(1-r^2)W}\,dr-\psi_2\int_1^\rho \frac{f\psi_1}{(1-r^2)W}\,dr
\end{equation}
with $\Vert T\Vert_{\mathcal L(Y_{\rho_0},X_{\rho_0})}\lesssim 1$ for all $\rho_0>0$. 
\end{proposition}
\begin{proof}
(i): Since $\mathcal L_\infty (\psi_i^L)=0$ and $\mathcal L_\infty (\psi_i^R)=0$, we have from the definition of the Wronskian that
$$
(1-\rho^2)W'+\bigg[\frac{d-1}{\rho}-2(1+\alpha)\rho\bigg]W=0,\quad \rho\in (0,\infty)\setminus\{1\}.
$$
Then $W\propto \rho^{1-d}|1-\rho^2|^{s_c-\frac{3}{2}}$ in $(0,1)$. Also, in view of the asymptotic bahaviour of the hypergeometric functions at $\rho=1$ (see \cite{AS}), $\partial_\rho\psi_1^L$ is singular.  Then, $\psi_1^L$ and $\psi_1$ are linearly independent, so there exists $c_1\in\mathbb R$ such that \eqref{eq: Wronskian} holds. Similarly, $W\propto \rho^{1-d}|1-\rho^2|^{s_c-\frac{3}{2}}$ in $(1,\infty)$ and $\psi_1^R$ and $\psi_1$ are linearly independent, so we can choose $c_2$ with \eqref{eq: Wronskian}. The asymptotic behaviours then follow from the definitions \eqref{eq: Left Fundamental Solutions}.\\\\
(ii): Integrals in \eqref{eq: Exterior Resolvent} are well-defined since
$$\psi_1=
\begin{cases}
\mathcal O_{\rho\rightarrow 1}(1) \\
\mathcal O_{\rho\rightarrow \infty}(\rho^{-\alpha}) 
\end{cases}
,\quad
\psi_2=
\begin{cases}
\mathcal O_{\rho\rightarrow 1}(1)\\
\mathcal O_{\rho\rightarrow \infty}(\rho^{-\alpha-1})
\end{cases}
,\quad
\frac{1}{(1-\rho^2)W}=
\begin{cases}
\mathcal O_{\rho\rightarrow 1}((\rho-1)^{\frac{1}{2}-s_c})\\
\mathcal O_{\rho\rightarrow \infty}(\rho^{2\alpha})
\end{cases}
$$
(see  \cite{AS}). Using variation of constants, 
$$
w=\psi_1\bigg(a_1+\int_{\rho}^\infty \frac{f\psi_2}{(1-r^2)W}\,dr\bigg)-\psi_2\bigg(a_2+\int_1^\rho \frac{f\psi_1}{(1-r^2)W}\,dr\bigg).
$$
solves 
$$
\mathcal L_\infty(w)=f.
$$
Since we require $\mathcal T: Y_{\rho_0}\rightarrow X_{\rho_0}$, we choose $a_1=0$. Since $\psi_2'=\mathcal O(\rho-1)^{s_c-\frac{3}{2}}$ as $\rho\rightarrow 1$ (see \cite{AS}), by requiring $\mathcal T(f)$ to be differentiable at $\rho=1$ we take $a_2=0$. It suffices to prove that $\mathcal T$ is bounded. For all $\rho\ge1$,
$$
\begin{aligned}
&\rho^{1+\alpha}|\mathcal T(f)(\rho)|\lesssim \rho^{1+\alpha}\bigg(|\psi_1|\int_{\rho}^\infty \bigg|\frac{f\psi_2}{(1-r^2)W}\bigg|\,dr+|\psi_2|\int_1^\rho\bigg|\frac{f\psi_1}{(1-r^2)W}\bigg|\,dr\bigg)\\
\lesssim &\sup_{\rho\ge1}\bigg(\rho\int_\rho^\infty r^{\frac{d-3}{2}}(r-1)^{\frac{1}{2}-s_c}|f|\,dr +\int_1^\rho r^{\frac{d-1}{2}}(r-1)^{\frac{1}{2}-s_c}|f|\,dr\bigg) \lesssim \Vert f\Vert_{Y_{\rho_0}}.
\end{aligned}
$$
For all $\rho_0\le\rho\le 1$,
$$
\begin{aligned}
\rho^{\frac{d}{2}-1} |\mathcal T(f)(\rho)|&\lesssim \rho^{\frac{d}{2}-1}\bigg(|\mathcal T(f)(1)|+|\psi_1|\int_\rho^1\bigg|\frac{f\psi_2}{(1-r^2)W}\bigg|\, dr+|\psi_2|\int_{\rho}^1\bigg|\frac{f\psi_1}{(1-r^2)W}\bigg|\, dr\bigg)\\
&\lesssim \Vert f\Vert_{Y_{\rho_0}}+\sup_{\rho_0\le r\le 1}\int_{r}^1 s^{\frac{d}{2}}(s-1)^{\frac{1}{2}-s_c}|f|\,ds\lesssim \Vert f\Vert_{Y_{\rho_0}}
\end{aligned}
$$
where in the final inequality, we used $\psi_i=\mathcal O(\rho^{1-\frac{d}{2}})$ and $\frac{1}{(1-\rho^2)W}=\mathcal O(\rho^{d-1})$ as $\rho\rightarrow 0$. Thus, $\Vert \mathcal T(f) \Vert_{X_{\rho_0}}\lesssim \Vert f\Vert_{Y_{\rho_0}}$.
\end{proof}

\subsection{Exterior solutions}

We now solve \eqref{eq: Self-similar Profile} in the exterior region $\rho>\rho_0$ as a fixed point problem involving $\mathcal L_\infty$. We first prove a Lipschitz type bound on the nonlinear term.\\

\begin{lemma}[Non-linear bounds]
\label{lem: Non-linear Bounds}
For $w\in X_{\rho_0}$ and $\varepsilon>0$, define
\begin{equation}
\label{eq: G}
G[\psi_1,\varepsilon]w=\underbrace{(\psi_1+w)^2}_{:=A[\psi_1]w}\underbrace{\int_0^1(1-s)(u_\infty+s\varepsilon(\psi_1+w))^{p-2}\,ds}_{:=B[\psi_1,\varepsilon]w}\bigg].
\end{equation}
Then for all $\varepsilon\ll\rho_0^{s_c-1}$ and $w_1,\, w_2\in B_{X_{\rho_0}}=\{ w\in X_{\rho_0}\ | \ \Vert w\Vert_{X_{\rho_0}}<1\}$,
\begin{equation}
\label{eq: Non-linear Bounds}
\Vert G[\psi_1,\varepsilon]w_1\Vert_{Y_{\rho_0}}\lesssim \rho_0^{1-s_c},\quad \Vert G[\psi_1,\varepsilon]w_1-G[\psi_1,\varepsilon]w_2\Vert_{Y_{\rho_0}}\lesssim  \rho_0^{1-s_c}\Vert w_1-w_2\Vert _{X_{\rho_0}}.
\end{equation}
\end{lemma}
\begin{proof}
Note that for all $\rho\ge1$,
$$
|\psi_1(\rho)|+|w_1(\rho)|\lesssim |u_\infty(\rho)|.
$$
Since $\psi_1=\mathcal O( \rho^{-\alpha})$ as $\rho\rightarrow\infty$ and $\varepsilon\lesssim 1$,
$$
\begin{aligned}
|G[\psi_1,\varepsilon]w_1(\rho)|&\lesssim (|\psi_1|+|w_1|)^2\Big[|u_\infty|+\varepsilon (|\psi_1|+|w_1|)\Big]^{p-2}\\
&\lesssim \rho^{-2\alpha}\bigg(1+\sup_{r\ge 1} r^{\alpha+1}|w_1|\bigg)^2|u_\infty(\rho)|^{p-2}\\
&\lesssim \rho^{-\alpha-2}(1+\Vert w_1\Vert_{X_{\rho_0}})^2\lesssim \rho^{-\alpha-2}
\end{aligned}
$$
so
$$
\int_1^\infty \rho^{\frac{d-1}{2}}|1-\rho|^{\frac{1}{2}-s_c}|G[\psi_1,\varepsilon]w_1|\,d\rho\lesssim \int_1^\infty \rho^{s_c-\frac{5}{2}}|1-\rho|^{\frac{1}{2}-s_c}\,d\rho\lesssim1.
$$
Note that since $\psi_1=\mathcal O(\rho^{1-\frac{d}{2}})$ as $\rho\rightarrow0$, for all $\rho_0\le\rho\le 1$,
$$
|\psi_1(\rho)|+|w_1(\rho)|\lesssim \rho^{1-\frac{d}{2}}\lesssim \rho^{1-s_c}|u_\infty(\rho)|.
$$
Then since $\varepsilon\ll\rho_0^{s_c-1}$,
$$
\begin{aligned}
|G[\psi_1,\varepsilon]w_1|&\lesssim \rho^{2-d}\bigg(1+\sup_{\rho\le r\le 1} r^{\frac{d}{2}-1}| w_1|\bigg)^2|u_\infty(\rho)|^{p-2}\\
&\lesssim \rho^{\alpha-d}(1+\Vert w_1 \Vert_{X_{\rho_0}})^2\lesssim \rho^{\alpha-d}.
\end{aligned}
$$
Then
$$
\int_{\rho_0}^1 \rho^{\frac{d}{2}}(1-\rho)^{\frac{1}{2}-s_c}|G[\psi_1,\varepsilon]w_1|\, d\rho\lesssim \int_{\rho_0}^1 \rho^{-s_c}(1-\rho)^{\frac{1}{2}-s_c}\,d\rho\lesssim\rho_0^{1-s_c}.
$$
Hence, the first bound in \eqref{eq: Non-linear Bounds} holds. For the contraction estimate, note that
$$
\begin{aligned}
&|G[\psi_1,\varepsilon]w_1-G[\psi_1,\varepsilon]w_2| \le|Aw_1-Aw_2|\,|Bw_1|+|Aw_2|\, |Bw_1-Bw_2|\\
&\lesssim|2\psi_1+w_1+w_2|\, |w_1-w_2|\,\Big[|u_\infty|+\varepsilon (|\psi_1|+|w_1|)\Big]^{p-2}+\varepsilon|w_1-w_2|(\psi_1+w_2)^2 I_{w_1,w_2}
\end{aligned}
$$
where
$$
\begin{aligned}
I_{w_1,w_2}:&=\bigg|\int_0^1\varepsilon^{-1}\partial_w B[\psi_1,\varepsilon]w\Big|_{w_2+\sigma(w_1-w_2)}\,d\sigma\bigg|\\
&\lesssim\bigg|\int_0^1s(1-s)\int_0^1(u_\infty+s\varepsilon (\psi_1+w_2)+\sigma s\varepsilon (w_1-w_2))^{p-3}\,d\sigma ds\bigg|\\
&\lesssim \Big[|u_\infty|+\varepsilon (|\psi_1|+|w_1|+|w_2|)\Big]^{p-3} \lesssim u_\infty^{p-3}
\end{aligned}
$$
where the final inequality follows since $\varepsilon\ll\rho_0^{s_c-1}$. Then
$$
|G[\psi_1,\varepsilon]w_1-G[\psi_1,\varepsilon]w_2|\lesssim \Big[(|\psi_1|+|w_1|+|w_2|)|u_\infty|^{p-2}+\varepsilon(|\psi_1|+|w_2|)^2|u_\infty|^{p-3}\Big]|w_1-w_2|.
$$
Since $\psi_1=\mathcal O(\rho^{-\alpha})$ as $\rho\rightarrow \infty$,
$$
\int_1^\infty \rho^{\frac{d-1}{2}}|1-\rho|^{\frac{1}{2}-s_c}|G[\psi_1,\varepsilon]w_1-G[\psi_1,\varepsilon]w_2|\,d\rho\lesssim  \int_1^\infty \rho^{s_c-\frac{7}{2}}|1-\rho|^{\frac{1}{2}-s_c}\,d\rho\,\Vert w_1-w_2\Vert_{X_{\rho_0}}.
$$
Since $\psi_1=\mathcal O(\rho^{1-\frac{d}{2}})$ as $\rho\rightarrow 0$, for all $\rho_0\le \rho \le 1$,
$$
\begin{aligned}
|G[\psi_1,\varepsilon]w_1-G[\psi_1,\varepsilon]w_2|&\lesssim \bigg(\rho^{2(1-\frac{d}{2})-\alpha(p-2)}+\varepsilon\rho^{3(1-\frac{d}{2})-\alpha(p-3)} \bigg)\sup_{\rho_0\le r\le 1}r^{\frac{d}{2}-1}|w_1-w_2|\\
&\lesssim \rho^{\alpha-d}\Vert w_1-w_2\Vert _{X_{\rho_0}}
\end{aligned}
$$
where the final inequality holds by our choice of $\varepsilon$. Thus,
$$
\int_{\rho_0}^1 \rho^{\frac{d}{2}}|1-\rho|^{\frac{1}{2}-s_c}|G[\psi_1,\varepsilon]w_1-G[\psi_1,\varepsilon]w_2|\, d\rho\lesssim \rho_0^{1-s_c}\Vert w_1-w_2\Vert _{X_{\rho_0}}.
$$
Hence, the second bound in \eqref{eq: Non-linear Bounds} holds.
\end{proof}

We are now in position to solve \eqref{eq: Self-similar Profile}. We in particular, prove the existence of a one-parameter family of smooth solutions in the region $\rho>\rho_0$.\\
\begin{proposition}[Exterior solutions]
\label{prop: Exterior Solutions}
For all $0<\varepsilon\ll\rho_0^{s_c-1}$, there exists a smooth solution to \eqref{eq: Self-similar Profile} of the form
$$
u=u_\infty+\varepsilon(\psi_1+w)
$$
with
\begin{equation}
\label{eq: Bounds on w}
 \Vert w\Vert _{X_{\rho_0}}\lesssim\varepsilon \rho_0^{1-s_c},\quad \Vert \Lambda w\Vert _{X_{\rho_0}}\lesssim \varepsilon \rho_0^{1-s_c}.
\end{equation}
Furthermore,
$$
w|_{\varepsilon=0}=0,\quad \Vert \partial_\varepsilon w|_{\varepsilon=0}\Vert _{X_{\rho_0}}\lesssim \rho_0^{1-s_c}.
$$
\end{proposition}
\begin{proof}
$u=u_\infty+\varepsilon v>0$ solves \eqref{eq: Self-similar Profile} if and only if 
$$
\begin{aligned}
\mathcal L_\infty(v)&=\varepsilon^{-1}[u_\infty^p+pu_\infty^{p-1}\varepsilon v-(u_\infty+\varepsilon v)^p]\\
&=-p(p-1)\varepsilon v^2\int_0^1(1-s)(u_\infty+s\varepsilon v)^{p-2}\,ds.
\end{aligned}
$$
We further decompose $v=\psi_1+ w$. Since $\mathcal L_\infty (\psi_1) =0$,
\begin{equation}
\label{eq: Fixed Point}
w=-p(p-1)\varepsilon\mathcal T\circ G[\psi_1,\varepsilon] w.
\end{equation}
Lemma \ref{lem: Non-linear Bounds} together with Proposition \ref{prop: Exterior Resolvent} states precisely that for $\varepsilon\ll\rho_0^{s_c-1}$,
$$
-p(p-1)\varepsilon \mathcal T\circ G[\psi_1,\varepsilon]:B_{X_{\rho_0}}\rightarrow B_{X_{\rho_0}}
$$
is a contraction map. From the Banach fixed point theorem, there exists a unique solution $w$ to \eqref{eq: Fixed Point} with $\Vert w\Vert_{X_{\rho_0}}\lesssim \varepsilon \rho_0^{1-s_c}$. Clearly, $w$ is smooth in $(0,\infty)\setminus \{1\}$. In view of \eqref{eq: Fixed Point}, $w\in C^1((0,\infty))$ so $u\in C^1((0,\infty))$. Writing \eqref{eq: Self-similar Profile} in the form required by Proposition \ref{prop: Singular ODE}, for $(\Psi_1,\Psi_2)=(u,u')$,
$$
\begin{cases}
(\rho-1)\partial_\rho \Psi_1=(\rho-1)\Psi_2\\
(\rho-1)\partial_\rho \Psi_2=-\frac{\alpha(\alpha+1)}{1+\rho}\Psi_1+\frac{1}{1+\rho}\left[\frac{d-1}{\rho}-2(\alpha+1)\rho\right]\Psi_2+\frac{u^p}{1+\rho}.
\end{cases}
$$
Hence,
$$
(\rho-1)\partial_\rho\begin{pmatrix} \Psi_1\\ \Psi_2\end{pmatrix} = A(\rho)\begin{pmatrix} \Psi_1\\ \Psi_2\end{pmatrix} +\frac{1}{\rho+1}\begin{pmatrix} 1\\u^p\end{pmatrix} 
$$
where $A$ is smooth in $(0,\infty)$ and 
$$
A(1)=\frac{1}{2}
\begin{pmatrix}
0 & 0\\
-\alpha(\alpha+1) & 2s_c-3
\end{pmatrix}
$$
with $\sigma(A(1))=\{s_c-\frac{3}{2},0\}$. By Proposition \ref{prop: Singular ODE}, since $u\in C^1((0,\infty))$, $(u,u')\in C^1(0,\infty)$ so $u\in C^2((0,\infty))$. Iterating this, we conclude that $u$ is smooth. \\\\
Applying $\Lambda$ to \eqref{eq: Fixed Point}, we infer
$$
\Lambda w=-p(p-1)\varepsilon \bigg[(\Lambda \psi_1)\int_{\rho}^\infty \frac{G[\psi_1,\varepsilon](w)\psi_2}{(1-r^2)W}\,dr-(\Lambda\psi_2)\int_1^\rho \frac{G[\psi_1,\varepsilon](w)\psi_1}{(1-r^2)W}\,dr\bigg].
$$
Hence, by considering the asymptotes of $\Lambda \psi_i$ and proceeding as in the proof of Proposition \ref{prop: Exterior Resolvent}, we infer
$$
\Vert\Lambda w\Vert_{X_{\rho_0}}\lesssim \varepsilon \Vert G[\psi_1,\varepsilon]w\Vert_{Y_{\rho_0}}\lesssim \varepsilon \rho_0^{1-s_c}.
$$
In view of \eqref{eq: Fixed Point}, $w|_{\varepsilon=0}=0$. Differentiating \eqref{eq: Fixed Point} in $\varepsilon$,
$$
\begin{aligned}
\partial_{\varepsilon}w|_{\varepsilon=0}&=-p(p-1)\bigg(\mathcal T\circ G[\psi_1,0]w|_{\varepsilon=0}+\varepsilon\mathcal T(\partial_{\varepsilon}G[\psi_1,\varepsilon]w)|_{\varepsilon=0}\bigg)\\
&=-p(p-1) \mathcal T\circ G[\psi_1,0]w|_{\varepsilon=0}= -\frac{p(p-1)}{2}\mathcal T(u_\infty^{p-2}\psi_1^2)
\end{aligned}
$$
so by continuity of the resolvent and the asymptotic behaviour of $\psi_1$ as $\rho\rightarrow 0$ and $\rho\rightarrow\infty$,
$$
\Vert \partial_{\varepsilon}w|_{\varepsilon=0}\Vert_{X_{\rho_0}}\lesssim \Vert u_\infty^{p-2}\psi_1^2\Vert_{Y_{\rho_0}}\lesssim \int_{\rho_0}^1\rho^{-s_c}|1-\rho|^{\frac{1}{2}-s_c}\,d\rho+ \int_1^\infty\rho^{s_c-\frac{5}{2}}|1-\rho|^{\frac{1}{2}-s_c}\,d\rho\lesssim \rho_0^{1-s_c}.
$$
\end{proof}

\section{Construction of interior solutions}

In this section, we construct inner solutions to the self-similar equation \eqref{eq: Self-similar Profile} which are perturbations of a rescaled soliton. The steps are similar to that of the previous section.\\

Let us first introduce some notations for this section.\\\\
\underline{\emph{Linearized Operator}}. Recall the definition of soliton solution
\begin{equation}
\label{eq: Soliton}
\Delta Q+Q^p=0,\quad Q(\rho)=b_\infty\rho^{-\alpha}+\mathcal O_{\rho\rightarrow\infty } (\rho^{1-\frac{d}{2}}).
\end{equation}
We let the linearized operator $\mathcal H_\infty$ near 
$$
Q_\lambda(\rho):=\lambda^{-\alpha}Q\left(\frac{\rho}{\lambda}\right),\,\,\,\, \lambda>0.
$$ 
for the profile equation \eqref{eq: Self-similar Profile} be
\begin{equation}
\label{eq: H_infty}
\mathcal H_\infty =-\Delta -pQ^{p-1}=-\frac{d^2}{d\rho^2}-\frac{d-1}{\rho}\frac{d}{d\rho}-pQ^{p-1}.
\end{equation}

\begin{lemma}[Fundamental solutions of $\mathcal H_\infty$]
\label{lemma:homogeneoussolutionsofH}
Recall from above the definition of the soliton $Q$. We then have a basis of fundamental solutions
$$
\mathcal H_\infty(\Lambda Q)=0, \quad \mathcal H_\infty\varphi=0
$$
with the following asymptotic behavior as $\rho\to\infty$
\begin{equation}
\label{eq: Asymptotes of Lambda Q and phi}
\Lambda Q,\, \varphi \propto \rho^{1-\frac{d}{2}}\sin(\omega\log\rho+\delta_{\bullet})+\mathcal O(\rho^{2-d+\alpha})
\end{equation}
for some $\delta_{\Lambda Q}$, $\delta_{\varphi}\in \mathbb R$. By scaling $\varphi$ if necessary, we assume that the Wronskian is given by
$$
W:=(\Lambda Q)'\varphi-\varphi'\Lambda Q=-\rho^{1-d}.
$$
\end{lemma}

\begin{proof}
Recall the definition of $Q_\lambda$ above. Then, for all $\lambda>0$,
$$
\Delta Q_\lambda +Q_\lambda^p=0
$$
and differentiating with respect to $\lambda$ and evaluating at $\lambda=1$ yields $\mathcal H_\infty(\Lambda Q)=0$. Let $\varphi$ be another solution to $\mathcal H_\infty(\varphi)=0$ which does not depend linearly on $\Lambda Q$, we aim at deriving the asymptotic of both $\Lambda Q$ and $\varphi$ as $\rho\to \infty$. We first solve
\begin{equation}
\label{eq:equationasymptotictoH}
-\tilde\varphi ''- \frac{d-1}{\rho}\tilde\varphi' -\frac{pb_\infty^{p-1}}{\rho^2}\tilde\varphi=f.
\end{equation}
The homogeneous problem admits the explicit basis of solutions
\begin{equation}\label{eq:explictiformulaforvarphi1andvarphi2}
\varphi_1=\rho^{1-\frac{d}{2}}\sin(\omega\log\rho),\,\,\,\,\varphi_2=\rho^{1-\frac{d}{2}}\cos(\omega\log\rho),
\end{equation}
and the corresponding Wronskian is given by
$$W=\varphi_1'\varphi_2-\varphi_2'\varphi_1=\omega \rho^{1-d}.$$
Using the variation of constants, the solutions to \eqref{eq:equationasymptotictoH} are given by
$$\tilde\varphi(\rho)=\varphi_1\left(a_1+\int_\rho^{\infty}f\varphi_2\frac{{r}^{d-1}}{\omega}dr\right)+\varphi_2\left(a_2-\int_\rho^{\infty}f\varphi_1\frac{{r}^{d-1}}{\omega}dr\right).$$
Then, we rewrite the equation $\mathcal H_\infty(\varphi)=0$:
$$-\varphi'' - \frac{2}{\rho}\varphi'-\frac{pb_\infty^{p-1}}{\rho^2}\varphi = p\bigg(Q^{p-1}-\frac{b_\infty^{p-1}}{\rho^2}\bigg)\varphi,$$
and hence
\begin{equation}\label{eq:easylinearfixedpoint}
\varphi=a_1\varphi_1+a_2\varphi_2+\widetilde{\phi},\,\,\,\,\widetilde{\phi}=\mathcal G(\widetilde{\phi})
\end{equation}
where
$$
\begin{aligned}
\mathcal G(\widetilde{\phi})(\rho) &= \varphi_1\int_\rho^\infty p\bigg(Q^{p-1}-\frac{b_\infty^{p-1}}{{r}^2}\bigg)\left(a_1\varphi_1+a_2\varphi_2+\widetilde{\phi}\right)\varphi_2\frac{{r}^{d-1}}{\omega}dr\\
&+\varphi_2\int_\rho^\infty p\bigg(Q^{p-1}-\frac{b_\infty^{p-1}}{{r}^2}\bigg)\left(a_1\varphi_1+a_2\varphi_2+\widetilde{\phi}\right)\varphi_1\frac{{r}^{d-1}}{\omega}dr.
\end{aligned}
$$
In view of the asymptotic behaviour \eqref{eq: Soliton} for $Q$, we infer for all $\rho\ge1$, 
$$
\bigg|p\bigg(Q^{p-1}-\frac{b_\infty^{p-1}}{{\rho}^2}\bigg)\bigg|\lesssim \rho^{-1-s_c}
$$
We infer for $\rho\geq 1$
$$
|\mathcal G(\widetilde{\phi})(\rho)| \lesssim \rho^{1-\frac{d}{2}}\int_\rho^\infty\left(r^{-s_c}+r^{\alpha-1}|\widetilde{\phi}|\right)dr\lesssim \rho^{2-d+\alpha}+\rho^{1-\frac{d}{2}}\int_\rho^\infty r^{\alpha-1}|\widetilde{\phi}| dr
$$
and similarly,
$$
|\mathcal G(\widetilde{\phi}_1)(\rho)-\mathcal G(\widetilde{\phi}_2)(\rho)| \lesssim \rho^{1-\frac{d}{2}}\int_\rho^\infty r^{\alpha-1}|\widetilde{\phi}_1-\widetilde{\phi}_2| dr.
 $$ 
 Thus, for $R\geq 1$ large enough, the Banach fixed point theorem applies and yields a unique solution $\widetilde{\phi}$ to \eqref{eq:easylinearfixedpoint} in the space corresponding to the norm
 $$\sup_{\rho\geq R}\,\rho^{d-\alpha-2}|\widetilde{\phi}|.$$
 In particular, in view of the explicit formula \eqref{eq:explictiformulaforvarphi1andvarphi2} for $\varphi_1$ and $\varphi_2$, and in view of the fact that $\mathcal H_\infty(\Lambda Q)=0$ and $\mathcal H_\infty(\varphi)=0$, we infer \eqref{eq: Asymptotes of Lambda Q and phi}
\end{proof}

For $\rho_1\ge 1$, we define the space of functions on which we invert our linearlized operator $\mathcal H_\infty$:
\begin{equation}
\begin{aligned}
\tilde X_{\rho_1}&=\Big\{w:(0,\rho_1)\rightarrow \mathbb R\,\Big|\,\Vert w\Vert_{\tilde X_{\rho_1}}:=\sup_{0\le \rho\le\rho_1}(1+\rho)^{\frac{d}{2}-3}(|w|+\rho|w'|+\rho^2|w''|)<\infty\Big\}\\
\tilde Y_{\rho_1}&=\Big\{w:(0,\rho_1)\rightarrow \mathbb R\,\Big|\,\Vert w\Vert_{\tilde Y_{\rho_1}}:=\sup_{0\le \rho\le\rho_1}(1+\rho)^{\frac{d}{2}-1}|w|<\infty\Big\}.
\end{aligned}
\end{equation}

\begin{proposition}[Interior resolvent]
\label{prop: Interior Resolvent}
There exists a bounded linear operator $\mathcal S: \tilde Y_{\rho_1}\rightarrow \tilde X_{\rho_1}$ such that $\mathcal H_\infty \circ \mathcal S=\id_{\tilde Y_{\rho_1}}$ given by
\begin{equation}
\mathcal S(f)=\Lambda Q\int_0^\rho f\varphi r^{d-1}\,dr-\varphi\int_0^\rho f\Lambda Q r^{d-1}\,dr
\end{equation}
with $\Vert \mathcal S\Vert_{\mathcal L(\tilde Y_{\rho_1},\tilde X_{\rho_1})}\lesssim 1$ for all $\rho_1\ge1$.
\end{proposition}
\begin{proof}
We recall from the previous lemma that $W=-\rho^{1-d}$. Let $R_0>0$ be sufficiently small so that $\Lambda Q>0$ in $[0,R_0]$. Then solving the Wronskian equation, we assume without loss of generality that for $\varphi$,
$$
\varphi=-\Lambda Q\int_\rho^{R_0}\frac{dr}{(\Lambda Q)^2r^{d-1}}.
$$
on $(0,R_0]$ which ensures that as $\rho\to0$,
\begin{equation}
\label{estrhorgigine}
|\varphi|\lesssim \rho^{2-d},\quad   |\varphi'|\lesssim \rho^{1-d}\quad |\varphi''|\lesssim \rho^{-d} .
\end{equation} 
where we have used that $Q$ and hence, $\Lambda Q$ is a smooth radial function. Using the variation of constants
$$
w = \Lambda Q \left(a_1+\int_0^\rho f\varphi r^{d-1}dr\right)+ \varphi\left(a_2-\int_0^\rho f\Lambda Q r^{d-1}dr\right)
$$
solves 
$$
\mathcal H_\infty(w) = f.
$$
In particular, $\mathcal S(f)$ corresponds to the choice $a_1=a_2=0$. Finally, using the estimates \eqref{eq: Asymptotes of Lambda Q and phi}, \eqref{estrhorgigine}, we estimate for $0\leq \rho\leq 1$:
$$
\begin{aligned}
|\mathcal S(f)|=& \left|\Lambda Q \int_0^\rho f\varphi r^{d-1}dr- \varphi\int_0^\rho f\Lambda Q r^{d-1}dr\right|\\
\lesssim &\left(\int_0^\rho rdr+\rho^{2-d}\int_0^\rho r^{d-1}dr \right)\sup_{0\leq \rho\leq 1}|f| \lesssim \Vert f\Vert_{\tilde Y_{\rho_1}}.
\end{aligned}
$$
Similarly, taking derivatives,
$$
\begin{aligned}
|\rho\mathcal S(f)'|=& \rho\left|(\Lambda Q)' \int_0^\rho f\varphi r^{d-1}dr- \varphi'\int_0^\rho f\Lambda Q r^{d-1}dr\right|\\
\lesssim& \left(\rho^2\int_0^\rho rdr+\rho^{2-d}\int_0^\rho r^{d-1}dr \right)\sup_{0\leq r\leq 1}|f| \lesssim \Vert f\Vert_{\tilde Y_{\rho_1}},
\end{aligned}
$$ 
and
$$
\begin{aligned}
|\rho^2\mathcal S(f)''|&= \rho^2\bigg|(\Lambda Q)''\int_0^\rho f\varphi r^{d-1} \, dr -\varphi''\int_0^\rho f\Lambda Q r^{d-1}\,dr-f\bigg|\\
&\lesssim\left(\rho^2\int_0^\rho r\,dr +\rho^{2-d}\int_0^\rho r^{d-1}\,dr +\rho^2\right)\sup_{0\le\rho\le 1}|f|\lesssim \Vert f\Vert_{\tilde Y_{\rho_1}}.
\end{aligned}
$$
For $1\leq \rho \leq \rho_1$,
$$
\begin{aligned}
&(1+\rho)^{\frac{d}{2}-3}|\mathcal S(f)|= (1+\rho)^{\frac{d}{2}-3}\left|\Lambda Q \int_0^\rho f\varphi r^{d-1}dr- \varphi\int_0^\rho f\Lambda Q r^{d-1}dr\right|\\
\lesssim& (1+\rho)^{-2}\int_0^\rho (1+r)^{\frac{d}{2}} |f|dr\lesssim  (1+\rho)^{-2}\int_0^\rho (1+r)\, dr \sup_{0\leq \rho\leq \rho_1}(1+\rho)^{\frac{d}{2}-1}|f|\lesssim\Vert f\Vert_{\tilde Y_{\rho_1}}.
\end{aligned}
$$
Similarly, taking derivatives,
$$
\begin{aligned}
(1+\rho)^{\frac{d}{2}-3}|\rho \mathcal S(f)'|&= (1+\rho)^{\frac{d}{2}-3}\rho\left|(\Lambda Q)' \int_0^\rho f\varphi r^{d-1}dr- \varphi'\int_0^\rho f\Lambda Q r^{d-1}dr\right|\\
&\lesssim  (1+\rho)^{-2}\int_0^\rho (1+r)\, dr \sup_{0\leq \rho\leq \rho_1}(1+\rho)^{\frac{d}{2}-1}|f|\lesssim\Vert f\Vert_{\tilde Y_{\rho_1}}
\end{aligned}
$$ 
and
$$
\begin{aligned}
(1+\rho)^{\frac{d}{2}-3}|\rho^2\mathcal S(f)''|&=(1+\rho)^{\frac{d}{2}-3}\rho^2\bigg|(\Lambda Q)''\int_0^\rho f\varphi r^{d-1} \, dr -\varphi''\int_0^\rho f\Lambda Q r^{d-1}\,dr-f\bigg|\\
&\lesssim (1+\rho)^{-2}\int_0^\rho (1+r)^{\frac{d}{2}}|f|\,dr+(1+\rho)^{\frac{d}{2}-1}|f|\lesssim \Vert f\Vert_{\tilde Y_{\rho_1}}.
\end{aligned}
$$
Thus, $\Vert \mathcal S(f)\Vert_{\tilde X_{\rho_1}}\lesssim \Vert f\Vert _{\tilde Y_{\rho_1}}$.
\end{proof}

\begin{lemma}[Non-linear bounds]
\label{lem: Non-linear Bounds 2}
For $w\in \tilde X_{\rho_1}$ and $\lambda>0$, define
\begin{equation}
\label{eq: F}
F[Q,\lambda]w=\underbrace{p(p-1)\lambda^2 w^2}_{:=\tilde A[\lambda]w}\underbrace{\int_0^1(1-s)(Q+\lambda^2 sw)^{p-2}\,ds}_{:=\tilde B[Q,\lambda]w}- \mathcal F(Q+\lambda^2 w).
\end{equation}
where
$$ 
\mathcal F=\rho^2\frac{d^2}{d\rho^2}+2(1+\alpha)\rho \frac{d}{d\rho}+\alpha(1+\alpha).
$$
Then there exists $C>0$ such that for all $\rho_1\lambda\ll1$ and $\Vert w_1\Vert_{\tilde X_{\rho_1}},\,\Vert w_1\Vert_{\tilde X_{\rho_1}}\le C$,
\begin{equation}
\label{eq: Non-linear Bounds 2}
\Vert F[Q,\lambda] w_1\Vert _{\tilde Y_{\rho_1}}\le C\Vert \mathcal S\Vert_{\mathcal L(\tilde Y_{\rho_1},\tilde X_{\rho_1})}^{-1},\quad \Vert F[Q,\lambda]w_1-F[Q,\lambda]w_2\Vert _{\tilde Y_{\rho_1}}\lesssim \rho_1^2\lambda^2\Vert w_1-w_2\Vert_{\tilde X_{\rho_1}}
\end{equation}
\end{lemma}
\begin{proof}
We first bound $\mathcal F(Q)$. In view of \eqref{eq: Soliton},
$$
\rho^2 Q^{p-1}=b_\infty^{p-1}+\mathcal O_{\rho\rightarrow \infty}(\rho^{1-s_c}).
$$
Then in view of \eqref{eq: Asymptotes of Lambda Q and phi}, since $Q''+\frac{d-1}{\rho}Q'+Q^p=0$, we infer
$$
\begin{aligned}
\mathcal F(Q)&=-\rho^2 Q^p+(3-2s_c)\rho \,Q'+\alpha(1+\alpha)Q\\
&=(b_\infty^{p-1}-\rho^2Q^{p-1})Q+(3-2s_c)\Lambda Q=\mathcal O_{\rho\rightarrow \infty}(\rho^{1-\frac{d}{2}}).
\end{aligned}
$$
Note also that since $s_c>1$, we have that for all $0\le \rho\le \rho_1$,
$$
|w_1(\rho)|\lesssim (1+\rho_1)^{3-\frac{d}{2}}\Vert w_1\Vert_{\tilde X_{\rho_1}}\lesssim (1+\rho_1)^2|Q(\rho)|\,\Vert w_1\Vert_{\tilde X_{\rho_1}}
$$
so by our choice of $\lambda$, 
$$
\lambda^2|w_1(\rho)|\lesssim |Q(\rho)|\,\Vert w_1\Vert_{\tilde X_{\rho_1}}.
$$
With these estimates, for all $0\le \rho\le \rho_1$,
$$
\begin{aligned}
&|F[Q,\lambda]w_1|\lesssim \lambda^2|w_1|^2\Big(|Q|+\lambda^2|w_1|\Big)^{p-2}+|\mathcal F(Q)|+\lambda^2|\mathcal F(w_1)|\\
\lesssim &\,\lambda^2(1+\rho)^{6-d-\alpha(p-2)}\Big(\Vert w_1\Vert_{\tilde X_{\rho_1}}^2+\Vert w_1\Vert_{\tilde X_{\rho_1}}^p\Big)+(1+\rho)^{1-\frac{d}{2}}+\lambda^2(1+\rho)^{3-\frac{d}{2}}\Vert w_1\Vert_{\tilde X_{\rho_1}}\\
\lesssim &\,\Big[\rho_1^{3-s_c}\lambda^2\Big(\Vert w_1\Vert _{\tilde X_{\rho_1}}^2+\Vert w_1\Vert _{\tilde X_{\rho_1}}^p\Big)+1+\rho_1^2\lambda^2\Big] (1+\rho)^{1-\frac{d}{2}}\\
\lesssim&\,\Big[1+ \rho_1^2\lambda^2\Big(1+\Vert w_1 \Vert_{\tilde X_{\rho_1}}^p\Big)\Big](1+\rho)^{1-\frac{d}{2}}
\end{aligned}
$$
where we have used that $s_c>1$ in the last inequality. Choose $C>0$ such that 
$$
|F[Q,\lambda]w_1|\le \frac{C}{2}\Vert \mathcal S\Vert_{\mathcal L(\tilde Y_{\rho_1},\tilde X_{\rho_1})}^{-1}\Big[1+\rho_1^2\lambda^2\Big(\Vert w_1\Vert _{\tilde X_{\rho_1}}+\Vert w_1\Vert _{\tilde X_{\rho_1}}^p\Big)\Big] (1+\rho)^{1-\frac{d}{2}}.
$$
Then for $\rho_1\lambda\ll1$ and $\Vert w_1\Vert_{\tilde X_{\rho_1}}\le C$,
$$
| F[Q,\lambda]w_1|\le C\Vert \mathcal S\Vert_{\mathcal L(\tilde Y_{\rho_1},\tilde X_{\rho_1})}^{-1}.
$$
Hence, the first bound in \eqref{eq: Non-linear Bounds 2} holds.
$$
\begin{aligned}
&|F[Q,\lambda]w_1-F[Q,\lambda]w_2|\le |\tilde Aw_1-\tilde Aw_2|\,|\tilde Bw_1|+|\tilde Aw_2|\, |\tilde Bw_1-\tilde Bw_2|+\lambda^2|\mathcal F(w_1-w_2)|\\
&\lesssim \lambda^2|w_1+w_2|\,|w_1-w_2|(|Q|+\lambda^2|w|)^{p-2}+\lambda^4|w_1-w_2|\, |w_2|^2\tilde I_{w_1,w_2}+\lambda^2(1+\rho)^{3-\frac{d}{2}}\Vert w_1-w_2\Vert_{\tilde X_{\rho_1}}
\end{aligned}
$$
where
$$
\begin{aligned}
\tilde I_{w_1,w_2}&:=\bigg|\int_0^1\lambda^{-2}\partial_w \tilde B[Q,\lambda]w|_{w_2+\sigma(w_1-w_2)}\,d\sigma\bigg|\\
&\lesssim \bigg|\int_0^1s(1-s)\int_0^1(Q+s\lambda^2w_2+\sigma s \lambda^2(w_1-w_2))^{p-3}\,d\sigma ds\bigg|\\
&\lesssim \Big[|Q|+\lambda^2(|w_1|+|w_2|)\Big]^{p-3}\lesssim (1+\rho)^{-\alpha(p-3)}.
\end{aligned}
$$
Thus, 
$$
\begin{aligned}
&|F[Q,\lambda]w_1-F[Q,\lambda]w_2|\\
\lesssim&\Big[\lambda^2 (1+\rho)^{6-d-(p-2)\alpha}+\lambda^4(1+\rho)^{9-\frac{3d}{2}-(p-3)\alpha}+\lambda^2(1+\rho)^{3-\frac{d}{2}}\Big]\Vert w_1-w_2\Vert_{\tilde X_{\rho_1}}\\
\lesssim&\Big(\rho_1^{3-s_c}\lambda^2+\rho_1^{6-2s_c}\lambda^4+\rho_1^2\lambda^2\Big)(1+\rho)^{1-\frac{d}{2}}\Vert w_1-w_2\Vert_{\tilde X_{\rho_1}}\lesssim\rho_1^2 \lambda^2(1+\rho)^{1-\frac{d}{2}}\Vert w_1-w_2\Vert_{\tilde X_{\rho_1}}
\end{aligned}
$$
where again, we have used that $s_c>1$. Hence the second bound in \eqref{eq: Non-linear Bounds 2} holds.
\end{proof}

We prove the existence of a one-parameter family of smooth solutions to \eqref{eq: Self-similar Profile} in the region $\rho<\rho_0$.\\

\begin{proposition}[Interior solutions]
\label{prop: Interior Solutions}
For all $0\le\rho_0\ll 1$, $0<\lambda\le \rho_0$, there exists a solution to \eqref{eq: Self-similar Profile} on $0\le \rho\le \rho_0$ of the form 
$$
u=\lambda^{-\alpha}(Q+\lambda^2w)\bigg(\frac{\rho}{\lambda}\bigg)
$$
with $\Vert w\Vert_{\tilde X_{\rho_1}}\lesssim 1$ where $\rho_1=\frac{\rho_0}{\lambda}\ge1$.
\end{proposition}
\begin{proof}
$u=\lambda^{-\alpha}(Q+\lambda^2w)(\frac{\rho}{\lambda})$ solves \eqref{eq: Self-similar Profile} if and only if
\begin{equation}
\label{eq: Fixed Point 2}
\mathcal H_\infty(w)=\lambda^{-2}\Big[(Q+\lambda^2w)^p-Q^p-pQ^{p-1}\lambda^2w\Big]-\mathcal F(Q+\lambda^2w)=F[Q,\lambda]w.
\end{equation}
Lemma \ref{lem: Non-linear Bounds 2} together with Proposition \ref{prop: Interior Resolvent} states precisely that for $\rho_1\lambda=\rho_0\ll1$, 
$$
\mathcal S\circ F[Q,\lambda]:B_{\tilde X_{\rho_1}}(C):=\{w\in \tilde X_{\rho_1}\,|\,\Vert w\Vert _{\tilde X_{\rho_1}}\le C\}\rightarrow B_{\tilde X_{\rho_1}}(C)
$$
is a contraction map. Thus, Banach fixed point theorem applies and yields a unique solution $w$ to \eqref{eq: Fixed Point 2} with $\Vert w \Vert _{\tilde X_{\rho_1}}\le C$.
\end{proof}

\section{The matching}

We are now in position to ``glue" inner and outer solutions to produce exact solutions to \eqref{eq: Nonlinear Wave}.

\begin{proposition}[Existence of a countable number of smooth self-similar profiles]
\label{prop: Countable Solutions}
There exists $N\in\mathbb N$ such that for all $n\ge N$, there exists a smooth solution $u _n$ to \eqref{eq: Nonlinear Wave} such that $\Lambda u_n $ vanishes exactly $n$ times.
\end{proposition}
\begin{proof}
\noindent{\bf step 1} (Matching): Recall that
\begin{equation}
\label{eq: Asymptote of psi_1}
\begin{aligned}
\psi_1&= c_1\rho^{1-\frac{d}{2}}\sin(\omega\log\rho+\delta_1)+\mathcal O_{\rho\rightarrow 0}(\rho^{3-\frac{d}{2}})\\
\Lambda \psi_1&=c_1 \rho^{1-\frac{d}{2}}\Big[(1-s_c)\sin(\omega\log\rho+\delta_1)+\omega\cos(\omega\log\rho+\delta_1)\Big]+\mathcal O_{\rho\rightarrow 0}(\rho^{3-\frac{d}{2}}),
\end{aligned}
\end{equation}
for some $c_1\in\mathbb R$. Then, we can choose $0<\rho_0\ll 1$ such that
\begin{equation}
\label{intiialirnot}
\psi_1(\rho_0)=c_1\rho_0^{1-\frac{d}{2}}+\mathcal O_{\rho\rightarrow 0}(\rho^{3-\frac{d}{2}}),\quad \Lambda\psi_1(\rho_0)=c_1(1-s_c)\rho_0^{1-\frac{d}{2}}+\mathcal O_{\rho\rightarrow 0}(\rho^{3-\frac{d}{2}}),
\end{equation}
and Proposition \ref{prop: Exterior Solutions} and Proposition \ref{prop: Interior Solutions} apply. In particular, let 
$$
\begin{aligned}
u_{\exterior}[\varepsilon] &=u_\infty +\varepsilon \psi_1+\varepsilon w_{\exterior}\\
u_{\interior}[\lambda] &=\lambda^{-\alpha}(Q+\lambda^2 w_{\interior})\bigg(\frac{\rho}{\lambda}\bigg)
\end{aligned}
$$ 
be solutions to \eqref{eq: Self-similar Profile} in the regions $[\rho_0,\infty)$ and $[0,\rho_0]$ respectively. Define
$$
\mathcal I[\rho_0](\varepsilon,\lambda)=u_{\exterior}[\varepsilon](\rho_0)-u_{\interior}[\lambda](\rho_0).
$$
Then
$$
\partial_\varepsilon \mathcal I[\rho_0](\varepsilon,\lambda)=\partial_\varepsilon u_{\exterior}[\varepsilon](\rho_0)=\psi_1(\rho_0)+w_{\exterior}(\rho_0)+\varepsilon\partial_\varepsilon w(\rho_0).
$$
In view of Proposition \ref{prop: Exterior Solutions}, since $\psi_1(\rho_0)\neq 0$,
$$
\partial_\varepsilon \mathcal I[\rho_0](0,0)=\psi_1(\rho_0)\neq 0.
$$
From the asymptotic behaviour of $Q$ as $\rho\rightarrow\infty$, as $\lambda\rightarrow 0$,
$$
\begin{aligned}
\bigg|\lambda^{-\alpha}(Q-u_\infty+\lambda^2 w_{\interior})\bigg(\frac{\rho_0}{\lambda}\bigg)\bigg|&\lesssim\lambda^{-\alpha}\bigg[\bigg(\frac{\rho_0}{\lambda}\bigg)^{1-\frac{d}{2}}+\lambda^2\bigg(\frac{\rho_0}{\lambda}\bigg)^{3-\frac{d}{2}}\bigg]
&\lesssim \lambda^{s_c-1}\rho_0^{1-\frac{d}{2}}(1+\rho_0^2)\rightarrow 0
\end{aligned}
$$
Since $u_{\exterior}[0]=u_\infty$ is self-similar, this implies
$$
\mathcal I[\rho_0](0,0)=u_\infty (\rho_0)-\lim_{\lambda\rightarrow 0} \lambda^{-\alpha}u_\infty\bigg(\frac{\rho_0}{\lambda}\bigg)=0.
$$
Applying the implicit function theorem to 
$$
\tilde{\mathcal I}(\varepsilon, \mu):=\mathcal I[\rho_0](\varepsilon,\mu^{\frac{1}{s_c-1}})
$$
which is $C^1$, there exists $\lambda_0>0$ and $\tilde \varepsilon \in C^1([0,\lambda_0^{s_c-1}))$ such that $\tilde{\mathcal I} (\tilde\varepsilon(\mu),\mu)=0$. Then, for $\varepsilon(\lambda):=\tilde \varepsilon (\lambda^{s_c-1})$, we have $\mathcal I[\rho_0](\varepsilon(\lambda),\lambda)=0$ and $\varepsilon\in C^{s_c-1}([0,\lambda_0))$. Hence, 
$$
u_{\exterior}[\varepsilon(\lambda)](\rho_0)=u_{\interior}[\lambda](\rho_0)$$
on $[0,\lambda_0)$ i.e.
\begin{equation}
\label{eq: Matching}
\varepsilon(\lambda)(\psi_1(\rho_0)+w_{\exterior}(\rho_0))=\lambda^{-\alpha}(Q-u_\infty+\lambda^2w_{\interior})\bigg(\frac{\rho_0}{\lambda}\bigg).
\end{equation}
By the definition of $\rho_0$ and from the bounds on $w_{\exterior}$ and $w_{\interior}$ in Propositions \ref{prop: Exterior Solutions} and \ref{prop: Interior Solutions}, we infer for some $c\in \mathbb R$,
$$
\begin{aligned}
&\varepsilon(\lambda)\rho_0^{1-\frac{d}{2}}\Big[c+\mathcal O(\rho_0^2+\varepsilon(\lambda)\rho_0^{s_c-1})\Big]= \varepsilon(\lambda)(\psi_1(\rho_0)+w_{\exterior}(\rho_0))\\
=&\lambda^{-\alpha}(Q-u_\infty+\lambda^2w_{\interior})\bigg(\frac{\rho_0}{\lambda}\bigg) \lesssim \lambda^{s_c-1}\rho_0^{1-\frac{d}{2}}\Big[1+\mathcal O(\rho_0^2)\Big]
\end{aligned}
$$
as $\rho_0\rightarrow 0$, so as $\lambda\rightarrow 0$,
$$
|\varepsilon (\lambda)|\lesssim \lambda^{s_c-1}.
$$
It then follows from \eqref{eq: Matching} and \eqref{eq: Bounds on w} that 
\begin{equation}
\label{eq: epsilon}
\varepsilon (\lambda) = \psi_1^{-1}(\rho_0)\lambda^{-\alpha}(Q-u_\infty )\bigg(\frac{\rho_0}{\lambda}\bigg)+\mathcal O\Big(\lambda^{s_c-1}(\rho_0^2+\lambda^{s_c-1}\rho_0^{1-s_c})\Big).
\end{equation}
Consider now the spatial derivative
$$
\mathcal I'[\rho_0](\varepsilon(\lambda),\lambda)= \varepsilon(\lambda)(\psi_1'(\rho_0)+w_{\exterior}'(\rho_0))-\lambda^{-1-\alpha}(Q'-u_\infty'+\lambda^2w_{\interior}')\bigg(\frac{\rho_0}{\lambda}\bigg).
$$
From the bound on $\varepsilon(\lambda)$ above and the bound on $w_{\exterior}'$ and $w_{\interior}'$ in Propositions \ref{prop: Exterior Solutions} and \ref{prop: Interior Solutions}, we infer
$$
\begin{aligned}
&\mathcal I'[\rho_0](\varepsilon(\lambda),\lambda)=\varepsilon(\lambda)\psi_1'(\rho_0)-\lambda^{-1-\alpha}(Q'-u_\infty')\bigg(\frac{\rho_0}{\lambda}\bigg)+\mathcal O\Big(\lambda^{s_c-1}(\rho_0^{2-\frac{d}{2}}+\lambda^{s_c-1}\rho_0^{1-d+\alpha})\Big)\\
=&\frac{\lambda^{s_c-1}}{\rho_0^{\frac{d}{2}-1}\psi_1(\rho_0)}\bigg[\bigg(\frac{\rho_0}{\lambda}\bigg)^{\frac{d}{2}-1}(Q-u_\infty)\bigg(\frac{\rho_0}{\lambda}\bigg)\psi_1'(\rho_0)-\bigg(\frac{\rho_0}{\lambda}\bigg)^{\frac{d}{2}}(Q'-u_\infty')\bigg(\frac{\rho_0}{\lambda}\bigg)\frac{\psi_1(\rho_0)}{\rho_0}\bigg]\\
+&\ \mathcal O\Big(\lambda^{s_c-1}(\rho_0^{2-\frac{d}{2}}+\lambda^{s_c-1}\rho_0^{1-d+\alpha})\Big)
\end{aligned}
$$
where in the final inequality we inject \eqref{eq: epsilon} for $\varepsilon(\lambda)$. From the asymptotic behaviours \eqref{eq: Asymptote of psi_1} for $\psi_1$ and knowing that
\begin{equation}
\label{eq: Asymptote of Q}
\begin{aligned}
(Q-u_\infty )(\rho)&=c_2 \rho^{1-\frac{d}{2}}\sin (\omega\log \rho+\delta_2)+\mathcal O_{\rho\rightarrow\infty }(\rho^{2-d+\alpha}),\\
(Q'-u_\infty ')(\rho)&=c_2 \rho^{-\frac{d}{2}}\Big[(1-\tfrac{d}{2})\sin(\omega\log\rho+\delta_2)+\omega\cos (\omega\log\rho+\delta_2)\Big]+\mathcal O_{\rho\rightarrow\infty}(\rho^{1-d+\alpha}),
\end{aligned}
\end{equation}
for some $c_2\in\mathbb R$, it follows that 
$$
\begin{aligned}
&\frac{\rho_0^{\frac{d}{2}-1}\psi_1(\rho_0)}{\lambda^{s_c-1}}\mathcal I'[\rho_0](\varepsilon(\lambda),\lambda)=c_1c_2\,\omega \rho_0^{-\frac{d}{2}}\Big[\sin(\omega\log \rho_0-\omega\log\lambda +\delta_2)\cos(\omega\log\rho_0+\delta_1)\\
&-\cos(\omega\log\rho_0-\omega\log\lambda+\delta_2)\sin (\omega\log\rho_0+\delta_1)\Big]+\mathcal O\Big(\rho_0^{2-\frac{d}{2}}+\lambda^{s_c-1}\rho_0^{1-d+\alpha}\Big)\\
&=c_1c_2\,\omega \rho_0^{-\frac{d}{2}}\sin(-\omega\log\lambda+\delta_2-\delta_1)+\mathcal O\Big(\rho_0^{2-\frac{d}{2}}+\lambda^{s_c-1}\rho_0^{1-d+\alpha}\Big).
\end{aligned}
$$
Thus,
\begin{equation}
\mathcal I'[\rho_0](\varepsilon(\lambda),\lambda)=c_1c_2\,\omega\lambda^{s_c-1}\bigg[\frac{\sin(-\omega\log\lambda+\delta_2-\delta_1)}{\rho_0^{d-1}\psi_1(\rho_0)}+\mathcal O\Big(\rho_0^{2-\frac{d}{2}}+\lambda^{s_c-1}\rho_0^{1-d+\alpha}\Big)\bigg].
\end{equation}
Let
\begin{equation}
\label{eq:definitionoflambdakpm}
\lambda_{n,+}=\exp\bigg[\frac{-n\pi+\delta_2-\delta_1+\delta_0}{\omega}\bigg],\quad \lambda_{n,-}=\exp\bigg[\frac{-n\pi+\delta_2-\delta_1-\delta_0}{\omega}\bigg].
\end{equation}
Then, $\lambda_{n,\pm}\rightarrow 0$ as $n\rightarrow\infty$ and
$$
0<\cdots <\lambda_{n,+}<\lambda_{n,-}<\lambda_{n-1,+}<\lambda_{n-1,-}<\cdots.
$$
Then,
$$
\mathcal I'[\rho_0](\varepsilon(\lambda_{n,\pm}),\lambda_{n,\pm})= \pm(-1)^n \lambda_{n,\pm}^{s_c-1}\bigg[\frac{c_1c_2\,\omega}{\rho_0^{d-1}\psi_1(\rho_0)}\sin \delta_0+\mathcal O\Big(\rho_0^{2-\frac{d}{2}}+\lambda_{n,\pm}^{s_c-1}\rho_0^{1-d+\alpha}\Big)\bigg]
$$
For $\rho_0\ll 1$, and $n\gg 1$,
$$
\mathcal I'[\rho_0](\varepsilon(\lambda_{n,\pm}),\lambda_{n,-})\mathcal I'[\rho_0](\varepsilon(\lambda_{n,\pm}),\lambda_{n,+})<0.
$$
Since $\lambda\mapsto \mathcal I'[\rho_0](\varepsilon(\lambda),\lambda)$ is continuous, it follows from intermediate value theorem that for all $n\ge N\gg1$, there exists $\lambda_{n,+}<\mu_n<\lambda_{n,-}$ such that $\mathcal I'[\rho_0](\varepsilon(\mu_n),\mu_n)=0$ i.e.
$$
u_{\exterior}[\varepsilon (\mu_n)](\rho_0)=u_{\interior}[\mu_n](\rho_0),\quad u_{\exterior}'[\varepsilon (\mu_n)](\rho_0)=u_{\interior}'[\mu_n](\rho_0).
$$
Hence, the function
$$
u_n(\rho):=\begin{cases} 
u_{\interior}[\mu_n](\rho) & 0\le \rho\le \rho_0,\\
u_{\exterior}[\varepsilon(\mu_n)](\rho) & \rho_0\le \rho
\end{cases}
$$
is a smooth solution to \eqref{eq: Self-similar Profile} in $[0,\infty)$ for all $n\ge N$. \\

\noindent{\bf step 2} (Counting the zeroes): The remaining part of the proof is devoted to counting the number of zeroes of $\Lambda u_n$. We first claim that for $\rho_0\ll1$,
\begin{equation}
\label{cenjnenoeo}
\Lambda u_{\exterior}[\varepsilon]\ \ \mbox{has as many zeros as}\ \ \Lambda\psi_1 \ \ \mbox{on}\ \ \rho\geq \rho_0.
\end{equation}
Indeed, $\Lambda \psi_1+\Lambda w_{\exterior}$ does not vanish on $[R_0,\infty)$ for $R_0$ large enough from \eqref{eq: Asymptote of Lambda psi_1} and the uniform bound \eqref{eq: Bounds on w}. Moreover, $\Lambda\psi_1(\rho_0)\neq 0$ from the normalization \eqref{intiialirnot}, and the absolute value of the derivative of $\Lambda\psi_1$ at any of its zeroes is uniformly lower bounded using \eqref{eq: Asymptote of psi_1 and psi_2} and hence the uniform smallness  \eqref{eq: Bounds on w} yields the claim.\\\\
We now claim that for $\rho_0\ll1$,
\begin{equation}
\label{innerzeroes}
\Lambda u_{\interior}[\mu_n]\ \ \mbox{has as many zeros as}\ \  \Lambda Q \ \ \mbox{on}\ \  0\leq r\leq \frac{\rho_0}{\mu_n}.
\end{equation}
Indeed, recall that
$$
\Lambda u_{\interior}[\mu_n](\rho) =\mu_n^{-\alpha}(\Lambda Q+\mu_n^2 \Lambda w_{\interior})\left(\frac{\rho}{\mu_n}\right).
$$
We now claim 
\begin{equation}
\label{tobeprovedlambdaq}
\left(\frac{\rho_0}{\mu_n}\right)^{\frac{d}{2}-1} \left|\Lambda Q\left(\frac{\rho_0}{\mu_n}\right)\right|\gtrsim 1.
\end{equation}
Assume \eqref{tobeprovedlambdaq}, then since the zeros of $\Lambda Q$ are simple and since
$$\|\Lambda w_{\interior}\|_{\tilde X_{\frac{\rho_0}{\mu_n}}}=\sup_{0\leq \rho\leq \frac{\rho_0}{\mu_n}}(1+\rho)^{\frac{d}{2}-3}|\Lambda w_{\interior}|\lesssim 1$$
so that
$$\sup_{0\leq \rho\leq \frac{\rho_0}{\mu_n}}(1+\rho)^{\frac{d}{2}-1}|\mu_n^2\Lambda w_{\interior}|\lesssim \rho_0^2,$$
and similarily for $\Lambda^2 w_{\interior}$, and since
$$
\Lambda Q(0) = \frac{2}{p-1}\neq 0,
$$
we conclude for $\rho_0\ll1$ that $\Lambda Q+\mu_n^2\Lambda w_{\interior}$ has as many zeros as $\Lambda Q$ on $0\leq \rho\leq \frac{\rho_0}{\mu_n}$. We deduce that on $0\leq \rho\leq \rho_0$, $\Lambda u_{\interior}[\mu_n]$ has as many zeros as $\Lambda Q$ on $0\leq \rho\leq \frac{\rho_0}{\mu_n}$.\\

\noindent{\it Proof of \eqref{tobeprovedlambdaq}}: Recall that
$$
u_{\exterior}[\varepsilon(\mu_n)](\rho_0) = u_{\interior}[\mu_n](\rho_0),\quad u_{\exterior}[\varepsilon(\mu_n)]'(\rho_0) = u_{\interior}[\mu_n]'(\rho_0),
$$
which implies
$$\Lambda u_{\exterior}[\varepsilon(\mu_n)](\rho_0) = \Lambda u_{\interior}[\mu_n](\rho_0).$$
This yields using \eqref{eq: epsilon}:
$$
\frac{\varepsilon(\mu_n)}{\mu_n^{s_c-1}} = \frac{1}{\psi_1(\rho_0)\mu_n^{\frac{d}{2}-1}}(Q-u_\infty)\left(\frac{\rho_0}{\mu_n}\right) + O\Big(\mu_n^{s_c-1}\rho_0^{s_c-1}+\rho_0^2\Big)
$$
and taking $\Lambda$ of \eqref{eq: Matching}:
$$
\frac{\varepsilon(\mu_n)}{\mu_n^{s_c-1}} = \frac{1}{\Lambda\psi_1(\rho_0)\mu_n^{\frac{d}{2}-1}}\Lambda Q\left(\frac{\rho_0}{\mu_n}\right) + O\Big(\mu_n^{s_c-1} \rho_0^{s_c-1}+\rho_0^2\Big).
$$
We infer
$$
 \frac{1}{\psi_1(\rho_0)\mu_n^{\frac{d}{2}-1}}(Q-u_\infty)\left(\frac{\rho_0}{\mu_n}\right) = \frac{1}{\Lambda\psi_1(\rho_0)\mu_n^{\frac{d}{2}-1}}\Lambda Q\left(\frac{\rho_0}{\mu_n}\right)+ O\Big(\mu_n^{s_c-1} \rho_0^{s_c-1}+\rho_0^2\Big).
$$
In view of the asymptote \eqref{intiialirnot} of $\psi_1$, we infer
\begin{equation}
\label{eq:usefulboundtocountnumberofzeros}
 \left|\left(\frac{\rho_0}{\mu_n}\right)^{\frac{d}{2}-1}(Q-u_\infty)\left(\frac{\rho_0}{\mu_n}\right)\right| \leq \frac{2}{s_c-1}\left|\left(\frac{\rho_0}{\mu_n}\right)^{\frac{d}{2}-1}\Lambda Q\left(\frac{\rho_0}{\mu_n}\right)\right|+ O\Big(\mu_n^{s_c-1}+\rho_0^2\Big).
\end{equation}
On the other hand, from \eqref{eq: Asymptote of Q},
\begin{equation}
\label{eq:equationwithtwostars}
\begin{aligned}
\Lambda Q(\rho)&=\frac{c_2}{ \rho^{\frac{d}{2}-1}}\Big[(1-s_c)\sin(\omega\log\rho+\delta_2)+\omega\cos (\omega\log\rho+\delta_2)\Big]+\mathcal O_{\rho\rightarrow\infty}(\rho^{2-d+\alpha})\\
&=\frac{c_2\sqrt{(s_c-1)^2+\omega^2}}{ \rho^{\frac{d}{2}-1}}\sin(\omega\log\rho+\delta_2+\alpha_0)+\mathcal O_{\rho\rightarrow\infty}(\rho^{2-d+\alpha})
\end{aligned}
\end{equation}
where
$$
\cos(\alpha_0)=\frac{1-s_c}{\sqrt{(s_c-1)^2+\omega^2}},\,\,\,\, \sin(\alpha_0)=\frac{\omega}{\sqrt{(s_c-1)^2+\omega^2}},\,\,\,\,\alpha_0\in \left(\frac{\pi}{2}, \pi\right).
$$
Thus, in view of \eqref{eq: Asymptote of Q} and \eqref{eq:equationwithtwostars}, there exists $\rho_2>0$ sufficiently small and a constant $\delta>0$ sufficiently small only depending on $\omega$ and $s_c-1$ such that for $0<\rho<\rho_2$, we have
$$\textrm{dist}\Big(\omega\log\rho+\delta_2+\alpha_0, \pi\mathbb{Z}\Big)<\delta\,\Rightarrow\, \rho^{\frac{d}{2}-1}|Q(\rho) - u_\infty(\rho)|\geq \frac{4}{s_c-1}\rho^{\frac{d}{2}-1}|\Lambda Q(\rho)|+\frac{c_1\sin(\alpha_0)}{2}.$$
In view of \eqref{eq:usefulboundtocountnumberofzeros}, we infer for $n\ge n_0$ large enough
\begin{equation}\label{eq:equationwithonestar}
\textrm{dist}\left(\omega\log\left(\frac{\rho_0}{\mu_n}\right)+\delta_2+\alpha_0, \pi\mathbb{Z}\right)\geq \delta
\end{equation}
and \eqref{tobeprovedlambdaq} is proved.\\\\
Combining the two claims proved above, we infer
$$
\begin{aligned}
&\#\Big\{\rho\geq 0\ \Big| \ \Lambda u_n(\rho)=0\Big\}\\
 =& \#\left\{0\leq \rho\leq \frac{\rho_0}{\mu_n}\ \bigg| \ \Lambda Q(\rho)=0\right\}+ \#\Big\{\rho>\rho_0\ \Big| \ \Lambda\psi_1(\rho)=0\Big\}
\end{aligned}
$$
which implies
$$
\#\{\rho\geq 0 \ | \ \Lambda u_{n+1}(\rho)=0\} = \#\{\rho\geq 0 \ | \ \Lambda u_n(\rho)=0\} + \# A_n,
$$
with
$$
A_n := \left\{\frac{\rho_0}{\mu_n}< \rho\leq \frac{\rho_0}{\mu_{n+1}}\ \bigg| \ \Lambda Q(r)=0\right\}.
$$
We claim for $n\geq n_1$ large enough:
\begin{equation}
\label{estcarajo}
\# A_n=1
\end{equation}
which by possibly shifting the numeration by a fixed amount ensures that $\Lambda u_n$ vanishes exactly $k$ times.\\

\noindent{\em Upper bound}. We first claim 
\begin{equation}
\label{upperbound}
\# A_n\le 1
\end{equation}
Recall \eqref{eq:equationwithtwostars} so that there exists $R\geq 1$ large enough such that 
$$
\Big\{\rho\geq R\,\Big|\,\Lambda Q(\rho)=0\Big\}= \Big\{r_q\ \Big|\ q\geq q_1\Big\},\,\,\omega\log(r_q)+\delta_2+\alpha_0 =q\pi +\mathcal O_{r_q\to\infty}(r_q^{1-s_c})
$$
and hence, together with \eqref{eq:equationwithonestar}, we infer 
\begin{equation}
\label{distanceinninimisee}
\inf_{q\geq q_1, n\geq n_1}\left|\log\left(\frac{\rho_0}{\mu_n}\right)-\log(r_q)\right|\geq \frac{\delta}{2\omega}.
\end{equation}
This implies for $n\geq n_1$
\begin{equation}
\label{eq:inclusioninalargerintervalofAk}
A_n \subset \left\{q\geq q_1\ \bigg|\ \log\left(\frac{\rho_0}{\mu_n}\right)+\frac{\delta}{2\omega}\leq \log(r_q) \leq \log\left(\frac{\rho_0}{\mu_{n+1}}\right)-\frac{\delta}{2\omega}\right\}.
\end{equation}
Since $\lambda_{n,+}<\mu_n<\lambda_{n,-}$ with $\lambda_{n,\pm}$ given by \eqref{eq:definitionoflambdakpm}, we have for $k\geq k_1$
$$
\begin{aligned}
&\log\left(\frac{\rho_0}{\mu_{n+1}}\right)-\frac{\delta}{2\omega} - \left(\log\left(\frac{\rho_0}{\mu_n}\right)+\frac{\delta}{2\omega}\right) = \log(\mu_n)-\log(\mu_{n+1})-\frac{\delta}{\omega}\\
&\leq \log(\lambda_{n,-})-\log(\lambda_{n+1,+}) -\frac{\delta}{\omega}\leq \frac{\pi+2\delta_0-\delta}{\omega}.
\end{aligned}
$$
Also, we have for $q\geq q_1$
$$
\log(r_{q+1})-\log(r_q) = \frac{\pi}{\omega}+\mathcal O_{r_q\to\infty}(r_q^{1-s_c}).
$$
We now choose $\delta_0$ such that
\begin{equation}
\label{choicedeltao}
0<\delta_0<\frac{\delta}{4}.
\end{equation}
Then, we infer
 for $n\geq n_1$ and $q\geq q_1$,
$$
 \log\left(\frac{\rho_0}{\mu_{n+1}}\right)-\frac{\delta}{2\omega} - \left(\log\left(\frac{\rho_0}{\mu_n}\right)+\frac{\delta}{2\omega}\right)\leq \frac{\pi}{\omega} -\frac{\delta}{2\omega}<\log(r_{q+1})-\log(r_q)
$$
which in view of \eqref{eq:inclusioninalargerintervalofAk} implies \eqref{upperbound}.\\

\noindent{\em Lower bound}. We now prove \eqref{estcarajo} and assume for a contradiction: $\# A_{n_2}=0.$ Then, let $q_2\geq q_1$ such that
$$r_{q_2} < \frac{\rho_0}{\mu_{n_2}} <  \frac{\rho_0}{\mu_{n_2+1}} < r_{q_2+1}.$$
We infer from \eqref{distanceinninimisee}:
\begin{equation}
\label{eq:equationwithfivestars}
\log(r_{q_2}) \leq  \log\left(\frac{\rho_0}{\mu_{n_2}}\right)-\frac{\delta}{2\omega} <  \log\left(\frac{\rho_0}{\mu_{n_2+1}}\right)+\frac{\delta}{2\omega} \leq \log(r_{q_2+1}).
\end{equation}
However, we have for $n_2\geq n_1$ and $q_2\geq q_1$,
$$
\begin{aligned}
&\log\left(\frac{\rho_0}{\mu_{n_2+1}}\right)+\frac{\delta}{2\omega} - \left(\log\left(\frac{\rho_0}{\mu_{n_2}}\right)-\frac{\delta}{2\omega}\right)= \log(\mu_{n_2})-\log(\mu_{n_2+1})+\frac{\delta}{\omega}\\
&\geq \log(\lambda_{n_2, -})-\log(\lambda_{n_2+1,+}) +\frac{\delta}{\omega}\geq \frac{\pi-2\delta_0+\delta}{\omega}\geq  \frac{\pi}{\omega} + \frac{\delta}{2\omega} > \log(r_{q_2+1}) -\log(r_{q_2})
\end{aligned}
$$
which contradicts \eqref{eq:equationwithfivestars}. This concludes the proof of Proposition \ref{prop: Countable Solutions}.
\end{proof}

\begin{corollary}
\label{cor: Behaviour of u_n}
Let $u_n$ be the solution to \eqref{eq: Self-similar Profile} constructed in \emph{Proposition \ref{prop: Countable Solutions}}. For $\rho_0\ll1$,\\\\
\emph{(i) Convergence to $u_\infty$ as $n\rightarrow \infty$:}
\begin{equation}
\label{eq: Convergence to u_infty}
\lim_{n\rightarrow \infty }\sup_{\rho\ge\rho_0}(1+\rho^\alpha)|u_n(\rho)-u_\infty(\rho)|=0.
\end{equation}
\emph{(ii) Convergence to $Q$ at the origin:} There exists $\mu_n\rightarrow 0$ such that
\begin{equation}
\label{eq: Convergence to Q}
\lim_{n\rightarrow \infty }\sup_{\rho\le\rho_0}\bigg|u_n(\rho)-\mu_n^{-\alpha}Q\bigg(\frac{\rho}{\mu_n}\bigg)\bigg|=0.
\end{equation}
\emph{(iii) Last zeros:} Let 
$$
\rho_{0,n}:=\max\Big\{\rho\,\Big|\,\Lambda u_n(\rho)=0,\,\rho<\rho_0\Big\}, \quad \rho_{\Lambda Q,n}:=\max\Big\{\rho\,\Big|\,\Lambda Q(\rho)=0,\,\rho<\tfrac{\rho_0}{\mu_n}\Big\}.
$$ 
Then 
$$
\rho_{0,n}=\mu_n\rho_{\Lambda Q, n}\Big[1+\mathcal O_{\rho_0\rightarrow 0}(\rho_0^2)\Big].
$$
Furthermore, for $n\ge N$,
$$
e^{-\frac{2\pi}{\omega}}\rho_0<\rho_{0,n}<\rho_0.
$$
\end{corollary}
\begin{proof}
Choose $\rho_0\ll 1$  as in the proof of Proposition \ref{prop: Countable Solutions}. \\\\
(i) In view of \eqref{eq: Asymptote of psi_1} and \eqref{eq: Bounds on w}, we infer
$$
\begin{aligned}
&\sup_{\rho\ge\rho_0}(1+\rho^\alpha)|u_n(\rho)-u_\infty(\rho)|= \sup_{\rho\ge\rho_0}(1+\rho^\alpha)|\varepsilon(\mu_n)(\psi_1(\rho)+w_{\exterior}(\rho))|\\
\lesssim& \varepsilon (\mu_n)\bigg[\sup_{\rho_0\le\rho\le 1}(|\psi_1(\rho)|+|w_{\exterior}(\rho)|)+\sup_{\rho\ge1}\rho^\alpha(|\psi_1(\rho)|+|w_{\exterior}(\rho)|)\bigg]\\
\lesssim & \varepsilon (\mu_n)\rho_0^{1-\frac{d}{2}}.
\end{aligned}
$$
Since $\varepsilon(\mu_n)\rightarrow 0$ as $n\rightarrow \infty$, result follows.\\\\
(ii) In view of Proposition \ref{prop: Interior Solutions}, we infer
$$
\sup_{\rho\le\rho_0}\bigg| u_n(\rho)-\mu_n^{-\alpha} Q\bigg(\frac{\rho}{\mu_n}\bigg)\bigg|\le \mu_n^{2-\alpha}\sup_{\rho\le\rho_0} \bigg|w_{\interior}\bigg(\frac{\rho}{\mu_n}\bigg)\bigg|\lesssim \mu_n^{s_c-1}.
$$
Since $\mu_n\rightarrow 0$ as $n\rightarrow \infty$, result follows.\\\\
(iii) In view of \eqref{eq: Asymptotes of Lambda Q and phi}, 
$$
\Lambda Q\bigg(e^{-\frac{3\pi}{2\omega}}\frac{\rho_0}{\mu_n}\bigg)\Lambda Q\bigg(\frac{\rho_0}{\mu_n}\bigg)<0
$$
so by intermediate value theorem, there exists a zero of $\Lambda Q$ in the interval $[e^{-\frac{3\pi}{2\omega}}\tfrac{\rho_0}{\mu_n},\tfrac{\rho_0}{\mu_n})$. In particular,
\begin{equation}
\label{eq: Interval of rho_Lamba Q,n}
e^{-\frac{3\pi}{2\omega}}\frac{\rho_0}{\mu_n}\le \rho_{\Lambda Q, n}\le \frac{\rho_0}{\mu_n}.
\end{equation}
Also, if 
$$
e^{-\frac{2\pi}{\omega}}\rho_0\le \rho\le\rho_0,
$$
then $\frac{\rho}{\mu_n}\gg1$ for $n\ge N\gg 1$. Thus, from \eqref{eq: Asymptotes of Lambda Q and phi} and Proposition \ref{prop: Interior Solutions} since
$$
\sup_{0\le \rho\le \frac{\rho_0}{\mu_n}}(1+\rho)^{\frac{d}{2}-3}|\Lambda w_{\interior}|\lesssim 1,
$$
it follows that
$$
\begin{aligned}
\Lambda u_n(\rho)&=\mu_n^{-\alpha}(\Lambda Q+\mu_n^2\Lambda w_{\interior})\bigg(\frac{\rho}{\mu_n}\bigg)\\
&\propto \mu_n^{s_c-1}\rho^{1-\frac{d}{2}}\Big[\sin(\omega \log\rho-\omega\log\mu_n+\delta_2)+\mathcal O_{\rho\rightarrow 0}(\rho_0^2)\Big].
\end{aligned}
$$
Thus,
$$
\Big|\omega\log\rho_{0,n}-\omega\log\mu_n-\omega\log \rho_{\Lambda Q, n}\Big|\lesssim \rho_0^2.
$$
Hence,
$$
\rho_{0,n}=\mu_n\rho_{\Lambda Q, n} e^{\mathcal O(\rho_0^2)}=\mu_n\rho_{\Lambda Q, n}\Big[1+\mathcal O_{\rho_0\rightarrow 0}(\rho_0^2)\Big].
$$
Furthermore, since \eqref{eq: Interval of rho_Lamba Q,n} holds, we deduce
$$
e^{-\frac{2\pi}{\omega}}\rho_0<\rho_{0,n}<\rho_0.
$$
\end{proof}

\begin{remark} Statements of \emph{Proposition \ref{prop: Countable Solutions}} and \emph{Corollary \ref{cor: Behaviour of u_n}} yields \emph{Theorem \ref{theo: Result 1}}.
\end{remark}

\section{Dissipativity of linearized operator}

We now start the study of the dynamical stability of self-similar profiles. Our aim in this section is to realize the linearized operator as a compact perturbation of a maximal accretive operator in a {\em global in space} Sobolev norm. From now on, we assume $d=3$.\\

\noindent\underline{\emph{Linearized wave equation}}. Recall from Section \ref{sec: Notation} the definition of similarity transformation variables:
$$
\tilde\Psi(s,y)=(T-t)^\alpha\Phi(t,x),\quad s=-\log(T-t).
$$
which maps the wave equation \eqref{eq: Nonlinear Wave} onto
\begin{equation}
\label{eq: Transformed Nonlinear Wave}
\partial_s^2\tilde\Psi=-2y\cdot\nabla\partial_s\tilde\Psi-(1+2\alpha)\partial_s\tilde\Psi+\sum_{i,j}(\delta_{ij}-y_iy_j)\partial_{y_i}\partial_{y_j}\tilde\Psi-2(1+\alpha)y\cdot\nabla\tilde\Psi-\alpha(1+\alpha)\tilde\Psi+|\tilde\Psi|^{p-1}\tilde \Psi.
\end{equation}
We write the above as a system of linearized equations near $u_n$. For the perturbations:
$$
\Psi=\tilde \Psi-u_n,\quad  \Omega=-\partial_s \Psi-\Lambda \Psi,
$$
we have
\begin{equation}
\label{eq: Linearized System}
\partial_s X=\mathcal M X+G,\quad X=\begin{pmatrix} \Psi\\ \Omega\end{pmatrix},\quad G=\begin{pmatrix} 0\\ -|\tilde\Psi|^{p-1}\tilde\Psi+u_n^p+pu_n^{p-1}\Psi\end{pmatrix}
\end{equation}
where
\begin{equation}
\label{eq: M}
\mathcal M=-
\begin{pmatrix}
\Lambda & 1\\
\Delta +p u_n^{p-1} & \Lambda+1
\end{pmatrix}.
\end{equation}
From now on, we write
\begin{equation}
\label{eq: Derivative Notation}
\Psi_j=\nabla^j\Psi,\quad \Omega_j=\nabla^j\Omega
\end{equation}
where
$$
\nabla^{j}=\begin{cases}
\Delta^i & j=2i,\\
\nabla \Delta^i & j=2i+1.
\end{cases}
$$

\begin{lemma}[Commuting with derivatives]
\label{lem: Commuting with derivatives}
For $k\in \mathbb N$, there holds
$$
\nabla^k\mathcal M X= \mathcal M_k\nabla^k X+\tilde{\mathcal M}_kX
$$
where
\begin{equation}
\label{eq: Definition of M_k}
\mathcal M_k=-
\begin{pmatrix}
\Lambda+k & 1\\
\Delta & \Lambda+k+1
\end{pmatrix},
\end{equation}
and $\tilde{\mathcal M_k}$ satisfies the pointwise bound
\begin{equation}
\label{eq: Pointwise Bound on M_k}
|\tilde{\mathcal M_k}X|\lesssim_k
\begin{pmatrix} 
0\\ \sum_{j=0}^{k}\langle \rho\rangle ^{-2+j-k}|\nabla^j\Psi|
\end{pmatrix}.
\end{equation}
\end{lemma}
\begin{proof}
Direct computation yields the following formulae
$$
[\nabla^k, V]=\sum_{j\le k-1} c_j\nabla^{k-j} V\nabla^j,\quad [\nabla^k,\Lambda]=k\nabla^k.
$$
Hence, by Lemma \ref{lem: Bounds on u_n}, since $\partial_\rho^k (u_n^{p-1})=\mathcal O(\rho^{-2-k})$ as $\rho\rightarrow \infty$ for all $k$,
$$
\nabla^k (\Delta+pu_n^{p-1})\Psi=\Delta\Psi_k+\mathcal O\bigg(\sum_{j=0}^{k}\langle \rho\rangle ^{-2+j-k}|\nabla^j\Psi|\bigg)
$$
and
$$
\nabla^k\Lambda\Omega=(\Lambda+k) \Omega_k,\quad \nabla^k(\Lambda+1)\Omega=(\Lambda+k+1) \Omega_k.
$$
\end{proof}

\subsection{Subcoercivity}
Let us introduce some notations. First, recall the definition of $H_k$ from Section \ref{sec: Notation}. \\\\
\underline{\emph{Weighted $L^2$-space}}. We also define for $\gamma>0$, the weighted $L^2$-space $L_\gamma^2$ as the completion of $C_c^\infty(\mathbb R^3)$ with respect to the norm induced by the inner product
$$
(\Psi,\tilde\Psi)_{L_\gamma^2}=\int_{\mathbb R^3} \Psi\tilde\Psi\langle\rho\rangle^{-2\gamma}\,dy
$$
where $\langle \cdot\rangle$ denotes the Japanese bracket. We write $\Vert \Psi\Vert_{L_\gamma^2}^2=(\Psi,\Psi)_{L_\gamma^2}$.

\begin{lemma}
\label{lem: Compact Embedding}
Recall the notations for the spaces $H_k$ and $L_{k+2}^2$ above. Then the embedding $\iota:H_{k+1}\xhookrightarrow{} L_{k+2}^2$ is compact.
\end{lemma}
\begin{proof}
An improved Hardy's inequality (see \cite{CCF}) states that for all $\alpha\in2\mathbb Z$ and $f\in C_c^\infty(\mathbb R^3\setminus B_1(0))$,
$$
\int_{\mathbb R^3}\frac{|f|^2}{|y|^{2+\alpha}}\,dy\lesssim\int_{\mathbb R^3}\frac{|\nabla f|^2}{|y|^\alpha}\,dy.
$$
Also an improved Hardy-Rellich inequality (see \cite{CCF}) states that for all $\beta\in 2\mathbb Z$ and $f\in C_c^\infty(\mathbb R^3\setminus B_1(0))$
$$
\int_{\mathbb R^3} \frac{|f|^2}{|y|^{4+\beta}}\,dy\lesssim\int_{\mathbb R^3} \frac{|\Delta f|^2}{|y|^\beta}\,dy.
$$
By repeatedly applying these inequalities, starting with $f=(1-\chi)\Psi$ for the cut-off function $\chi$ defined in Section \ref{sec: Notation}, we infer for all $\Psi\in C_c^\infty(\mathbb R^3)$,
$$
\begin{aligned}
&\Vert \Psi\Vert_{L_{k+1}^2(\mathbb R^3\setminus B_1(0))}\lesssim \int_{\mathbb R^3} \frac{|(1-\chi)\Psi|^2}{|y|^{2(k+1)}}\,dy\lesssim \int_{\mathbb R^3} \frac{|\Delta ((1-\chi)\Psi)|^2}{|y|^{2(k-1)}}\,dy\\
\lesssim &\cdots \lesssim \int_{\mathbb R^3} \frac{|\nabla^k((1-\chi)\Psi)|^2}{|y|^2}\,dy\lesssim \int_{\mathbb R^3} |\nabla^{k+1}((1-\chi)\Psi)|^2\,dy\lesssim \Vert \Psi\Vert_{H_{k+1}}^2.
\end{aligned}
$$
By density, above inequality holds also for all $\Psi\in H_{k+1}$. On the other hand, by Rellich-Kondrachov theorem, the embedding
$$
\iota: H_{k+1}\xhookrightarrow{}L_{\loc}^2(\mathbb R^3):=\{\Psi|\, \chi\Psi\in L^2(\mathbb R^3)\text{ for all } \chi\in C_c^\infty(\mathbb R^3)\}
$$
is compact. Combining the two and using smallness of $\langle \rho\rangle^{-2}$ for large $\rho$, result follows.
\end{proof}

\begin{lemma}[Subcoercivity estimate]
\label{lem: Subcoercivity Estimate}
There exist $0<\mu_n$ with $\lim_{n\rightarrow\infty }\mu_n=\infty$ and $(\Pi_i)_{i=1}^{n}\in  H_{k+1}$, $c_n>0$ such that for all $n\ge0$, $\Psi\in H_{k+1}$,
\begin{equation}
\label{eq: Coercivity}
\Vert \Psi\Vert_{H_{k+1}}^2 \ge \mu_n\sum_{j=0}^{k} \int_{\mathbb R^3}|\nabla^j \Psi|^2\langle \rho\rangle ^{-2(k+2-j)}\,dy-c_n\sum_{i=1}^{n}(\Psi,\Pi_i)_{L_{k+2}^2}^2.
\end{equation}
\end{lemma}
\begin{proof}
Given $T\in L_{k+2}^2$, the antilinear map $h\mapsto (T,h)_{k+2}$ is continuous on $H_{k+1}$ since
$$
(h,h)_{L_{k+2}^2}\le ( h,h)_{H_{k+1}}
$$
by Lemma \ref{lem: Compact Embedding}. By Riesz, there exists a unique $L(T)\in H_{k+1}$ such that
\begin{equation}
\label{eq: Definition of L}
\forall h\in H_{k+1},\quad ( L(T),h)_{H_{k+1}}=(T,h)_{L_{k+2}^2}
\end{equation}
and by setting $h=L(T)$, we infer that $L: L_{k+2}^2\rightarrow H_{k+1}$ is a bounded linear map. By Lemma \ref{lem: Compact Embedding}, the map $\iota\circ L:L_{k+2}^2\rightarrow L_{k+2}^2$ is compact. If $\Psi_i=L(T_i)$, $i=1,2$, then
$$
(L(T_1),T_2)_{L_{k+2}^2}=(\Psi_1,T_2)_{L_{k+2}^2}=(\Psi_1,L(T_2))_{H_{k+1}}= ( \Psi_1,\Psi_2 )_{H_{k+1}}.
$$
Similarly,
$$
(T_1,L(T_2))_{L_{k+2}^2}=( \Psi_1,\Psi_2 )_{H_{k+1}}=(L(T_1),T_2)_{L_{k+2}^2}
$$
i.e. $L$ is self-adjoint on $L_{k+2}^2$. Since $L\succ 0$ from \eqref{eq: Definition of L}, there exists an $L_{k+2}^2$-orthonormal eigenbasis $(\Pi_{n,i})_{1\le i\le I(n)}$ of $L$ with positive eigenvalues $\lambda_n\rightarrow0$. The eigenvalue equation implies $\Pi_{n,i}\in H_{k+1}$. Let
$$
\mathcal A_n=\Big\{\Psi\in H_{k+1} \,\Big|\, (\Psi,\Psi)_{L_{k+2}^2}=1,\, (\Psi,\Pi_{j,i})_{L_{k+2}^2}=0,\,1\le i\le I(j),\, 1\le j\le n\Big\}
$$
and consider the minimization problem
$$
I_n=\inf_{\Psi\in \mathcal A_n}(\Psi,\Psi)_{H_{k+1}},
$$
whose infimum is attained at some $\Psi\in \mathcal A_n$ since the embedding $\iota:H_{k+1}\xhookrightarrow{} L_{k+2}^2$ is compact. Also, by a standard Lagrange multiplier argument,
$$
\forall  h\in H_{k+1},\quad( \Psi,h)_{H_{k+1}}=\sum_{j=1}^n\sum_{i=1}^{I(j)}\beta_{i,j}(\Pi_{j,i},h)_{L_{k+2}^2}+\beta(\Psi,h)_{L_{k+2}^2}.
$$
Set $h=\Pi_{j,i}$ and since $\Pi_{j,i}$ is an eigenvector of $L$, we infer $\beta_{i,j}=0$ and in view of \eqref{eq: Definition of L}, $L(\Psi)=\beta^{-1}\Psi$. Together with the orthogonality conditions, $\beta^{-1}\le \lambda_{n+1}$. Hence
\begin{equation}
\label{eq: Lower Bound on I_n}
I_n=( \Psi,\Psi)_{H_{k+1}}=\beta(\Psi,\Psi)_{L_{k+2}^2}\ge\frac{1}{\lambda_{n+1}}.
\end{equation}
For all $\varepsilon >0$, $k\ge 1$, from Gagliardo-Nirenberg interpolation inequality with weight (see \cite{DDS}) together with Young's inequality, we infer
$$
\sum_{j=0}^{k}\int_{\mathbb R^3}|\nabla^j\Psi|^2\langle \rho\rangle ^{-2(k+2)\frac{k+1-j}{k+1}}\,dy\le\varepsilon\int_{\mathbb R^3}|\nabla^{k+1}\Psi|^2\,dy+c_{\varepsilon,k}\int_{\mathbb R^3}|\Psi|^2\langle \rho\rangle ^{-2(k+2)}\,dy.
$$
Together with \eqref{eq: Lower Bound on I_n}, we have that for all $\Psi$ satisfying orthogonality condition of $\mathcal A_n$ and $\delta>0$,
$$
\sum_{j=0}^{k}\int_{\mathbb R^3}|\nabla^j\Psi|^2\langle \rho\rangle ^{-2(k+2-j)}\,dy\le (\varepsilon + c_{\varepsilon, k}\lambda_{n+1})\Vert\Psi\Vert_{H_{k+1}}^2.
$$
Choosing $\varepsilon_n\rightarrow 0$ such that $c_{\varepsilon_n,k}\lambda_{n+1}\le \varepsilon_n$ yields \eqref{eq: Coercivity}.
\end{proof}

\subsection{Dissipativity}
\label{Sec: Dissipativity}
We now turn to the fundamental dissipativity property. Let us introduce some notations.\\

\noindent\underline{\emph{Sobolev space}}. Recall \eqref{eq: Derivative Notation} and the definition of $H_k$ from Section \ref{sec: Notation}. Let 
\begin{equation}
\label{eq: Sobolev Space}
\mathbb H_{k}:= H_{k+1}\times H_k
\end{equation}
with the inner product: 
\begin{equation}
\label{eq: Principal Part}
\langle X,\tilde X\rangle=\underbrace{(\Psi_{k+1},\tilde \Psi_{k+1})+(\Omega_k,\tilde \Omega_k)}_{:=\langle X,\tilde X\rangle_1}+\underbrace{\int_{\mathbb R^3} \chi(\Psi \tilde \Psi +\Omega\tilde \Omega)\ dy}_{:=\langle X,\tilde X\rangle_2},
\end{equation}
for
$$
X=\begin{pmatrix} \Psi\\\Omega\end{pmatrix},\quad \tilde X=\begin{pmatrix} \tilde\Psi\\\tilde\Omega\end{pmatrix}.
$$
Further, we define the domain of $\mathcal M$
$$
D(\mathcal M)=\{X\in\mathbb H_k\,|\, \mathcal M X\in\mathbb H_k\}
$$
which is a Banach space equipped with the graph norm
$$
\Vert X\Vert_{D(\mathcal M)}=\Vert X\Vert_{\mathbb H_k}+\Vert \mathcal MX\Vert_{\mathbb H_k}.
$$

\noindent\underline{\emph{Spherical harmonics}}. Denote by $\Delta_{\mathbb S^{d-1}}$ the Laplace-Beltrami operator defined on a unit sphere $\mathbb S^{d-1}$. Then we can write
$$
\Delta=\frac{\partial^2}{\partial\rho^2}+\frac{d-1}{\rho}\frac{\partial}{\partial\rho}+\frac{1}{\rho^2} \Delta_{\mathbb S^{d-1}}=: \mathcal L+\rho^{-2}\Delta_{\mathbb S^{d-1}}.
$$
Denote by $Y^{(l,m)}$ the orthonormal $\Delta_{\mathbb S^{d-1}}$-eigenbasis (e.g. spherical harmonics if $d=3$) of $L^2(\mathbb S^{d-1})$ with discrete eigenvalues $-\lambda_m=-m(m+d-2)$ for $m\ge0$. We fix $d=3$ and define the space of test functions
$$
\mathcal D = \bigg\{X=\sum_{l,m} X_{l,m}(\rho)Y^{(l,m)} \in C_c^\infty(\mathbb R^3,\mathbb R^2)\emph{ is a finite sum}\bigg\}.
$$
Note then, that $\mathcal D$ is dense in $\mathbb H_k$.\\

\noindent\underline{\emph{Dissipativity}}. We will first prove dissipativity in the space of test functions
\begin{equation}
\label{eq: D_R}
\mathcal D_R=\bigg\{ X\in C^\infty(\mathbb R^3,\mathbb R^2)\bigg|\sum_{m=0}^{k+1}\sup_{\mathbb R^3} \rho^{\alpha+R+m}\left(|\nabla^m \Psi|+\mathbbm 1_{m\ge1}|\nabla^{m-1} \Omega|\right)< \infty\bigg\}
\end{equation}
and argue that the result extends to $\mathbb H_k$.

\begin{proposition}[Maximal dissipativity]
\label{prop: Maximal Dissipativity}
For all $k\ge 3$, there exists $c_k>0$ and $(X_i)_{1\le i\le N}\in\mathbb H_{k}$ such that for the finite rank projection operator
$$
\mathcal P=\sum_{i=1}^N\langle \cdot,X_i\rangle X_i,
$$
the modified operator
$$
\tilde {\mathcal M}=\mathcal M-\mathcal P
$$
is dissipative:
\begin{equation}
\label{eq: Dissipativity}
\forall X\in D(\mathcal M),\quad  \langle -\tilde {\mathcal M}X,X\rangle \ge c_k\langle X,X\rangle
\end{equation}
and is maximal:
\begin{equation}
\label{eq: Maximality}
\forall R>0,\,\, F\in \mathbb H_k,\quad \exists X\in D(\mathcal M)\quad \text{such that}\quad (-\tilde {\mathcal M}+ R)X=F.
\end{equation}
\end{proposition}
\begin{proof}
\textbf{Step 1} (Dissipativity on dense subset): We claim the bound \eqref{eq: Dissipativity} for $X\in \mathcal D_R$ for $R$ sufficiently large so integrating by parts is justified. Integrate by parts the principal part of the inner product defined in \eqref{eq: Principal Part}:
$$
\begin{aligned}
&\langle -\mathcal M X, X\rangle_1=(\nabla^{k+2} (\mathcal MX)_\Psi,\Psi_k)-(\nabla^k (\mathcal MX)_\Omega,\Omega_k)\\
=& \int_{\mathbb R^3} \Big[\nabla((\Lambda +k)\Psi_k+\Omega_k)\cdot\nabla\Psi_k+(\Delta\Psi_k+(1+k+\Lambda)\Omega_k-(\tilde{\mathcal M}_kX)_\Omega)\cdot\Omega_k\Big]\,dy\\
=& \int_{\mathbb R^3} \Big[\nabla((\Lambda+k)\Psi_k)\cdot\nabla\Psi_k+(1+k+\Lambda)\Omega_k \cdot\Omega_k-(\tilde{\mathcal M}_kX)_\Omega\cdot \Omega_k\Big]\,dy\\
=& (-s_c+k+1)\Big[(\nabla\Psi_k,\nabla\Psi_k)+(\Omega_k,\Omega_k)\Big]- \int_{\mathbb R^3} (\tilde{\mathcal M}_k X)_\Omega\cdot\Omega_k \,dy
\end{aligned}
$$
where in the last equality, we have used the Pohozaev identity. In view of \eqref{eq: Pointwise Bound on M_k} and by Young's inequality, we infer 
$$
\bigg|\int_{\mathbb R^3} (\tilde{\mathcal M}_k X)_\Omega\cdot\Omega_k\,dy\bigg|\le\varepsilon\int_{\mathbb R^3} | \Omega_k|^2\,dy+C_{\varepsilon,k}
 \sum_{j=0}^k\int_{\mathbb R^3} |\nabla^j \Psi|^2\langle \rho\rangle ^{-4+2j-2k}\,dy.
$$
Taking $\varepsilon>0$ small, it follows that 
$$
\langle -\mathcal M X, X\rangle_1\ge 2c_k\Big[(\Psi_{k+1},\Psi_{k+1})+(\Omega_k,\Omega_k)\Big]-C_k \sum_{j=0}^k\int_{\mathbb R^3} |\nabla^j \Psi|^2\langle \rho\rangle ^{-4+2j-2k}\,dy.
$$
We also lower bound the non-principal part: 
$$
\begin{aligned}
&\langle -\mathcal M X, X\rangle_2=-\int_{\mathbb R^3} \chi\Big[(\mathcal MX)_\Psi\Psi+(\mathcal MX)_\Omega\, \Omega\Big] \,dy\\
=& \int_{\mathbb R^3}\chi \Big[(\Lambda\Psi+\Omega)\Psi+((\Delta+pu_n^{p-1})\Psi+(1+\Lambda)\Omega)\,\Omega\Big] \,dy\\
\ge&- C\int_{|y|\le2}  \Big[|\Psi|^2+|\Delta \Psi|^2+|\Omega|^2+|\nabla\Omega|^2\Big]\,dy
\end{aligned}
$$
where the last inequality follows since $\chi=0$ for $|y|\ge2$. Thus, by adding the principal and non-principal parts, we infer
$$
\langle -\mathcal M X, X\rangle \ge 2c_k\langle X,X\rangle- C_k \sum_{j=0}^k\int_{\mathbb R^3} |\nabla^j \Psi|^2\langle \rho\rangle ^{-4+2j-2k}\,dy-C\Vert X\Vert_{H^2(|y|\le2)}^2.
$$
We conclude using \eqref{eq: Coercivity} and an analogous result for $\Omega$ that
$$
\langle -\mathcal M X, X\rangle \ge c_k\langle X, X\rangle-C\left[\sum_{i=1}^N (\Psi,\Pi_i)_{L_{k+2}^2}^2+\sum_{i=1}^N (\Omega,\Xi_i)_{L_{k+1}^2}^2\right]
$$
for $(\Pi_i)$ as in Lemma \ref{lem: Subcoercivity Estimate} and for some $\Xi_i\in L_{k+1}^2$. Since the linear form
$$
X=(\Psi,\Omega)\mapsto \sqrt{C}(\Psi,\Pi_i)_{L_{k+2}^2}
$$
is continuous on $\mathbb H_k$, by Riesz theorem, there exists $X_i\in\mathbb H_k$ such that 
$$
\forall X\in\mathbb H_k,\quad \langle X,X_i\rangle =(\Psi,\Pi_i)_{L_{k+2}^2}
$$
and similarly for $(\Xi_i)$. Hence, the claim \eqref{eq: Dissipativity} follows for all $X\in \mathcal D_R$.\\\\
\textbf{Step 2} (ODE formulation of maximality): Next, we claim that for all $R$ sufficiently large, 
\begin{equation}
\label{eq: Maximality on Dense Subset 1}
\forall F\in \mathcal D,\quad \exists! X\in \mathbb H_k\quad\text{such that}\quad (-\mathcal M+R)X=F. 
\end{equation}
Furthermore, we claim that $X\in\mathcal D_R$. Note that this is equivalent to
\begin{equation}
\label{eq: Maximality on Dense Subset 2}
\begin{cases}
(\Lambda+R)\Psi+\Omega=F_\Psi \\
(\Delta+ pu_n^{p-1})\Psi+(\Lambda+R+1)\Omega=F_\Omega.
\end{cases}
\end{equation}
Let $F\in \mathcal D$. Then, solving for $\Psi$, we have
\begin{equation}
\label{eq: Maximality for Psi 1}
[\Delta-(\Lambda+R+1)(\Lambda+R)+pu_n^{p-1}]\Psi=\underbrace{F_\Omega-(\Lambda+R+1)F_\Psi}_{:=H}.
\end{equation}
Since $\Lambda$ commutes with $\Delta_{\mathbb S^{d-1}}$, we can write
$$
F=\sum_{l,m} F_{l,m}Y^{(l,m)},\quad H=\sum_{l,m} H_{l,m}Y^{(l,m)}
$$
as a finite sum where $H_{l,m}(\rho)Y^{(l,m)}\in C_{c}^\infty(\mathbb R^3)$. Then the solution is of the form 
\begin{equation}
\label{eq: Decomposition of Psi}
\Psi=\sum_{l,m} Y^{(l,m)} \Psi_{l,m},\quad \Big[\mathcal L-\rho^{-2}\lambda_m-(\Lambda+R+1)(\Lambda+R)+pu_n^{p-1}\Big]\Psi_{l,m}(\rho)=H_{l,m}(\rho)
\end{equation}\\
By Lemma \ref{lem: Global Existence}, it follows that for all $R$ sufficiently large and $F_{l,m}Y^{(l,m)}\in C_c^\infty(\mathbb R^3,\mathbb R^2)$, there exists unique $\Psi_{l,m}(\rho)Y^{(l,m)}\in H^{k+1}(\mathbb R^3)$ solution to \eqref{eq: Decomposition of Psi}. Hence, there exists a unique $\Omega_{l,m}(\rho)Y^{(l,m)}\in H^k(\mathbb R^3)$ given by first equation of \eqref{eq: Maximality on Dense Subset 2} so that $X_{l,m}(\rho)Y^{(l,m)}=(\Psi_{l,m},\Omega_{l,m})Y^{(l,m)}\in \mathbb H_k$ smooth. Thus, we have \eqref{eq: Maximality on Dense Subset 1}. Also, from the decay properties of each $X_{l,m}$ proved in Lemma \ref{lem: Global Existence}, we infer $X\in \mathcal D_R$.\\\\
{\bf Step 3} (Density of $\mathcal D_R$): Now, we extend these results from $\mathcal D_R$ to $D(\mathcal M)$. Claim that for $R$ large, $\mathcal D_R\subset D(\mathcal M)$ is dense. For $X\in D(\mathcal M)$, we have $X$, $\mathcal MX\in\mathbb H_k$ so there exists a sequence $(Y_n)\in \mathcal D$ such that
$$
Y_n\rightarrow (-\mathcal M+R)X\quad\text{in }\mathbb H_k.
$$
By \eqref{eq: Maximality on Dense Subset 1} and Lemma \ref{lem: Global Existence}, there exists unique $X_n\in\mathbb H_k$ smooth solution to 
$$
(-\mathcal M+R)X_n=Y_n\rightarrow (-\mathcal M+R)X ,\quad X_n\in \mathbb H_k.
$$
It suffices to prove the $X_n\rightarrow X$ in $\mathbb H_k$. Recall that for $R$ sufficiently large all integration by parts used to prove \eqref{eq: Dissipativity} is justified. Then since $X_n\in \mathcal D_R$, \eqref{eq: Dissipativity} holds for $X_n-X_m$ i.e.
$$
\begin{aligned}
& \langle Y_n-Y_m, X_n-X_m\rangle =  \langle (-\mathcal M+R)(X_n-X_m), X_n-X_m\rangle \\
 =& \langle (-\mathcal M+\mathcal P)(X_n-X_m), X_n-X_m\rangle - \langle \mathcal P(X_n-X_m), X_n-X_m\rangle +R\Vert X_n-X_m\Vert _{\mathbb H_k}^2\\
\ge& R\Vert X_n-X_m\Vert _{\mathbb H_k}^2 -  \langle \mathcal P(X_n-X_m), X_n-X_m\rangle.
\end{aligned} 
$$ 
Since $\mathcal P$ is a bounded operator, we infer for $R$ large,
$$
\frac{R}{2}\Vert X_n-X_m\Vert_{\mathbb H_k}\le \Vert Y_n-Y_m\Vert_{\mathbb H_k}.
$$
In view of the convergence of $(Y_n)$ in $\mathbb H_k$, we deduce that $(X_n)$ is a Cauchy sequence hence, convergent in $\mathbb H_k$ to say, $\tilde X$. Then $\tilde X-X\in\mathbb H_k$ and
$$
(-\mathcal M+R)(\tilde X-X)=0
$$
as distributions. By the uniqueness statement in \eqref{eq: Maximality on Dense Subset 1}, it follows that $\tilde X=X$ i.e.
$$
X_n\rightarrow X,\quad \mathcal MX_n\rightarrow\mathcal MX \quad\text{in }\mathbb H_k\quad \Longleftrightarrow \quad X_n\rightarrow X\quad\text{in }\,\,D(\mathcal M).
$$
Hence, $\mathcal D_R$ is dense in $D(\mathcal M)$ as claimed.\\\\
\textbf{Step 4} (Conclusion): Since \eqref{eq: Dissipativity} holds for all $X\in \mathcal D_R$, by density of $\mathcal D_R$, we have dissipativity i.e. \eqref{eq: Dissipativity} holds for all $X\in D(\mathcal M)$. It remains to prove \eqref{eq: Maximality}. Let $F\in \mathbb H_k$. There exists $(F_n)\in \mathcal D$ such that 
$$F_n\rightarrow F\quad\text{in }\mathbb H_k.$$
By \eqref{eq: Maximality on Dense Subset 1}, there exists $X_n\in\mathbb H_k$ solution to
$$
(-\mathcal M+R)X_n=F_n.
$$
Using \eqref{eq: Dissipativity} and arguing as in the proof of density, we infer for $R$  large,
$$
\frac{R}{2}\Vert X_n-X_m\Vert_{\mathbb H_k}\le \Vert F_n-F_m\Vert_{\mathbb H_k}
$$
so $X_n$ has a limit say, $X\in \mathbb H_k$. Since $F_n$ converges to $F$ in $\mathbb H_k$ and $D(\mathcal M)$ is a Banach space, we infer
$$
(-\mathcal M+R)X=F,\quad X\in D(\mathcal M).
$$
Thus we have shown that for $R$ large,
\begin{equation}
\label{eq: Maximality of M}
\forall F\in \mathbb H_k,\quad \exists X\in D(\mathcal M)\quad \text{such that}\quad (-\mathcal M+R)X=F.
\end{equation}
Now we prove this for $\tilde{\mathcal M}$. Let $F\in\mathbb H_k$. Since $\mathcal P$ is bounded, for $R$ large, by \eqref{eq: Dissipativity}, for $X$ as in \eqref{eq: Maximality of M},
$$
\langle F,X\rangle =\langle(-\mathcal M+R)X,X\rangle = \langle(-\tilde{\mathcal M}-\mathcal P+R)X,X\rangle \ge \frac{R}{2}\Vert X\Vert_{\mathbb H_k}^2.
$$
Thus, for all $F\in \mathbb H_k$, solution $X$ to \eqref{eq: Maximality of M} is unique i.e. $(-\mathcal M+R)^{-1}$ is well-defined on $\mathbb H_k$ with
$$
\Vert (-\mathcal M+R)^{-1}\Vert\lesssim \frac{1}{R}.
$$
Hence,
$$
-\tilde{\mathcal M}+R=-\mathcal M+\mathcal P+R=(-\mathcal M+R)[\id+(-\mathcal M+R)^{-1}\mathcal P]
$$
is invertible on $\mathbb H_k$ for $R$ large which yields \eqref{eq: Maximality}. An elementary induction argument ensures that \eqref{eq: Maximality} holds for all $R>0$ (see Proposition 3.14 from \cite{EN}).
\end{proof}

\section{Growth bounds for the dissipative operators}
\label{sec: 6}

In this section, we recall some classical facts on growth bounds for compact perturbations of maximal accretive operators. We realize the linearized operator defined on the real vector space from previous sections as real operator on the corresponding complex space. This is essential in the spectral theory of the linearized operator.\\

In this section, $(H,\langle \cdot,\cdot\rangle)$ is a Hilbert space and $A$ a closed operator defined on a dense domain $D(A)$. Define the adjoint operator $A^*$ on the domain
$$
D(A^*)=\{X\in H\,|\, Y\in D(A)\mapsto \langle X,AY\rangle \emph{ extends to an element of }H^*\}
$$
to be $X\mapsto A^*X$ the unique element of $H$ given by Riesz theorem such that
$$
\forall Y\in D(A),\quad \langle A^* X,Y\rangle =\langle X, AY\rangle.
$$
Denote by 
$$
\Lambda_\nu(A)=\{\lambda\in \sigma(A)\,|\,\Re(\lambda)\ge\nu\},\quad V_\nu(A)=\bigoplus_{\lambda\in \Lambda_\nu(A)}\ker(A-\lambda).
$$

\begin{lemma}[Perturbative exponential decay]
\label{lem: Exponential Decay}\
Let $T_0$ and $T$ be the strongly continuous semigroup generated by a maximal dissipative operator $A_0$ and $A=A_0+K$ where $K$ is a compact operator on $H$. Then for all $\nu>0$, the following holds:\\\\
\emph{(i)} The set $\Lambda_\nu (A)$ is finite and each eigenvalue $\lambda\in\Lambda_\nu(A)$ has finite algebraic multiplicity $k_\lambda$.\\\\
We have $\Lambda_\nu(A)=\overline{\Lambda_\nu(A^*)}$ and $\dim V_\nu(A^*)=\dim V_\nu(A)$. The direct sum decomposition
$$
H=V_\nu(A)\bigoplus V_\nu^\bot(A^*)
$$
is preserved by $T(s)$ and there holds
$$
\forall X\in V_\nu^\bot(A^*),\quad \Vert T(s) X\Vert\le M_\nu e^{\nu s}\Vert X\Vert.
$$
\emph{(iii)} The restriction of $A$ to $V_\nu(A)$ is given by a direct sum of Jordan blocks. Each block corresponds to an invariant subspace $J_\lambda$ and the semigroup $T$ restricted to $J_\lambda$ is given by 
$$
T(s)|_{J_\lambda}=
\begin{pmatrix}
e^{\lambda s} & s e^{\lambda s} &\cdots & \tfrac{s^{m_\lambda-1}e^{\lambda s}}{(m_\lambda -1)!}\\
0 & e^{\lambda s}& \cdots &\tfrac{s^{m_\lambda-2} e^{\lambda s}}{ (m_\lambda -2)!}\\
\vdots & \vdots & \ddots & \vdots\\
0 & 0 & \cdots & e^{\lambda s}
\end{pmatrix}
$$
where $m_\lambda$ is the geometric multiplicity of the eigenvalue $\lambda$.
\end{lemma}
\begin{proof}
See Lemma 3.9 of \cite{MRRS}.
\end{proof}

\begin{corollary}[Exponential decay modulo finitely many instabilities]
\label{cor: Exponential Decay 1}
Let $\nu>0$, $T_0$, $T$ be the strongly continuous semigroup generated by a maximal dissipative operator $A_0$ and $A=A_0-\nu+K$ respectively where $K$ is a compact operator on Hilbert space $H$. Then $\Lambda_0(A)$ is finite and let
$$
H=U\bigoplus V
$$
where $U$ and $V$ are invariant subspaces for $A$ and $V$ is the image of the spectral projection of $A$ for the set $\Lambda_\nu(A)$. Then there exists $C, \delta>0$ such that
$$
\forall X\in U,\quad \Vert T(s)X\Vert \le Ce^{-\frac{\delta}{2}s}\Vert X\Vert.
$$
\end{corollary}
\begin{proof}
We apply Lemma \ref{lem: Exponential Decay} to $\tilde A=A_0+K$ which generates the semigroup $\tilde T$. Note that $\Lambda_{\frac{\nu}{4}}(\tilde A)$ is finite and $\Lambda_0(A)\subset \Lambda_{\frac{\nu}{4}}(\tilde A)$. Let 
$$ 
H=U_\nu\bigoplus V_\nu
$$
be the invariant decomposition of $\tilde A$ associated to the set $\Lambda_{\frac{\nu}{4}}$ with $V_\nu$ being the image of the spectral projection of the set $\Lambda_{\frac{\nu}{4}}$. Then $U_\nu\subset U$ and 
$$
U=U_\nu\bigoplus O_\nu
$$
where $O_\nu$ is the image of the spectral projection of $A$ associated with the set $\Lambda_{\frac{\nu}{4}}(\tilde A)\setminus \Lambda_0(A)$. Then by Lemma \ref{lem: Exponential Decay},
$$
\forall X\in U_\nu,\quad \Vert T(s) X\Vert = e^{-\nu s}\Vert \tilde T(s)X\Vert\le M_\nu e^{-\frac{3\nu}{4}s}\Vert X\Vert.
$$
Now for $X\in U$, since $U_\nu$ is invariant under $T$ and we have exponential decay on $U_\nu$, so without loss of generality, assume $X\in O_\nu$. $O_\nu$ is an invariant subspace of $A$ generated by the eigenvalues $\lambda$ such that $-\frac{3\nu}{4}\le \Re(\lambda)<0$. Then for 
$$
\delta = \inf\Big\{\Re(\lambda)\,|\,0<-\Re(\lambda)\le\frac{3\nu}{4}\Big\}
$$
Lemma \ref{lem: Exponential Decay} implies that
$$
\Vert T(s)X\Vert_{O_\nu}\lesssim \sup_{\Re(\lambda)<0} e^{\lambda s} s^{m_\lambda-1}\Vert X \Vert \le e^{-\frac{\delta}{2}s}\Vert X\Vert.
$$
\end{proof}

\begin{corollary}
\label{cor: Exponential Decay 2}
Let $A$, $\delta$, $U$ and $V$ as in \emph{Corollary \ref{cor: Exponential Decay 1}}. For $c$, $s_0>0$, let $G(s)\in V$ such that
$$
\Vert G\Vert\le e^{-\frac{\delta}{2}(1+c)s}.
$$
If $X(s)$ solves
$$
\frac{dX(s)}{ds}=AX(s)+G(s),\quad X(s_0)=x\in V
$$
for some $\Vert x\Vert\le e^{-\frac{\delta}{2}(1+\frac{c}{2})s_0}$, then
\begin{equation}
\label{eq: Exponential Decay}
\Vert X(s)\Vert\le e^{-\frac{\delta}{2}s},\quad s_0\le s\le s_0+\Gamma_{A,s_0}
\end{equation}
where $\Gamma_{A,s_0}$ can be made arbitrarily large by a choice of $s_0$. Moreover, there exists $x\in V$, $\Vert x\Vert\le e^{-\frac{\delta}{2}(1+\frac{c}{2})s_0}$ such that for all $s\ge s_0$,
$$
\Vert X(s)\Vert \le e^{-\frac{\delta}{2}(1+\frac{c}{2})s}.
$$
\end{corollary}
\begin{proof}
By Lemma \ref{lem: Exponential Decay}, the subspace $V$ can be further decomposed into invariant subspaces on which $A$ is represented by Jordan blocks. Therefore, without loss of generality, assume that $V$ is irreducible and for $\Re(\lambda)\ge0$,
\begin{equation}
\label{eq: Jordan Block}
A=\lambda+N,\quad e^{sN}=
\begin{pmatrix}
1 & s& \cdots & \tfrac{s^{m_\lambda-1}}{(m_\lambda -1)!}\\
0 & 1 & \cdots & \tfrac{s^{m_\lambda-2}}{(m_\lambda -2)!}\\
\vdots & \vdots &\ddots& \vdots\\
0 & 0 & \cdots & 1
\end{pmatrix}.
\end{equation}
Then from the growth bound on the Jordan block, we infer, for all $s_0\le s\le s_0+\Gamma$ that
$$
\begin{aligned}
\Vert X(s)\Vert&= \bigg\Vert e^{(s-s_0)A}x+\int_{s_0}^se^{(s-\tau)A}G(\tau)\,d\tau\bigg\Vert \\
&\lesssim  \Gamma ^{m_\lambda-1}e^{\Re(\lambda)\Gamma}e^{-\frac{\delta}{2}(1+\frac{c}{2})s_0}+\int_{s_0}^s |\tau-s_0|^{m_\lambda-1}e^{\Re(\lambda)(s-\tau)}e^{-\frac{\delta}{2}(1+c)\tau}\,d\tau\\
&\lesssim  \Gamma^{m_\lambda-1}e^{\Re(\lambda)\Gamma}e^{-\frac{\delta}{2}(1+\frac{c}{2})s_0}.
\end{aligned}
$$
Hence \eqref{eq: Exponential Decay} follows by choosing $\Gamma$ such that
$$
\Gamma^{m_\lambda-1}e^{\Re(\lambda)\Gamma}e^{-\frac{\delta}{2}(1+\frac{c}{2})s_0}\le e^{-\frac{\delta}{2}(s_0+\Gamma)},
$$
a sufficient condition being
$$
\Gamma\le \frac{s_0}{2}\bigg[\frac{c\,\delta}{2\Re(\lambda)+\delta}\bigg].
$$
Now consider
$$
Y(s)= e^{-sN}e^{\frac{\delta}{2}(1+\frac{3c}{4})s}X(s),\quad \tilde G(s)=e^{-sN}e^{\frac{\delta}{2}(1+\frac{3c}{4})s}G(s).
$$
Then since $N$ and $A$ commute,
$$
\frac{dY(s)}{ds}=\bigg[\lambda+\frac{\delta}{2}\bigg(1+\frac{3c}{4}\bigg)\bigg]Y(s)+\tilde G(s),\quad Y(s_0)=y.
$$
For $s_0$ sufficiently large, for all $s\ge s_0$,
$$
\Vert\tilde G(s)\Vert \le e^{-\frac{c\,\delta}{16}s}.
$$
We now run a standard Brouwer type argument for $Y$. For $\Vert y \Vert \le 1$, define the exit time
$$
s^*=\inf\{s\ge s_0\,|\, \Vert Y(s)\Vert\ge1\}.
$$
If $s^*=\infty$ for some $\Vert y\Vert \le 1$, then we're done. Otherwise, the map $\Phi:B=\{\Vert y\Vert \le1\}\rightarrow S=\{\Vert y\Vert =1\}$ given by $\Phi(y)=Y(s^*)$ is well-defined. Note that $\Phi|_{S}=\id_S$ and $\Phi$ is continuous since
$$
\frac{d\Vert Y\Vert^2}{ds}(s^*)=2\Re(\lambda)+\delta\bigg(1+\frac{3c}{4}\bigg)+2\Re\langle \tilde G(s^*),Y(s^*)\rangle\ge \frac{\delta}{2}\bigg(1+\frac{3c}{4}\bigg)>0
$$
i.e. the outgoing condition is met. This is a contradiction by Brouwer fixed point theorem. Thus, there exists $x$ such that for all $s\ge s_0$, 
$$
\Vert e^{-sN}X(s)\Vert \le e^{-\frac{\delta}{2}(1+\frac{3c}{4})s}.
$$
Since $e^{-sN}$ is invertible with inverse $e^{sN}$ bounded by $s^{m_\lambda-1}$, result follows immediately. 
\end{proof}

\section{Finite codimensional stability} 
\label{sec: Bootstrap}

We are now in position to prove non linear finite codimensional stability of the self-similar profiles for the full problem.\\

\noindent\underline{\emph{Choice of parameters}}. In this section, we set $d=3$ and $k=3$ so that $H^{k+1}(\mathbb R^3)$ is an algebra which we shall later use in the proof of \emph{Theorem \ref{theo: Result 2}}. For convenience, we write 
$$
\mathbb H=\mathbb H_3:= H_4\times H_3.
$$
where we recall from Section \ref{sec: Notation} the definition of $H_k$.\\\\
 \underline{\emph{Stable and Unstable supspaces}}. Recall from \emph{Proposition \ref{prop: Maximal Dissipativity}} that $\mathcal M-\mathcal P+\frac{c_k}{2}$ is maximal dissipative so \emph{Corollary \ref{cor: Exponential Decay 1}} applies:
$$
\Lambda_0(\mathcal M)=\{\lambda\in \sigma(\mathcal M)\,|\, \Re(\lambda)\ge 0\}
$$
is a finite set with an associated finite dimensional invariant subspace $V$. Consider the invariant decomposition
$$
\mathbb H=U\bigoplus V
$$
and let $P$ be the associated projection on $V$. We denote by $\mathcal N$ the nilpotent part of the matrix representing $\mathcal M$ on $V$. Let $\delta>0$ such that the conclusions of \emph{Corollary \ref{cor: Exponential Decay 1}} and \emph{\ref{cor: Exponential Decay 2}} hold.\\\\
 \underline{\emph{Dampened profile}}. We produce a finite energy initial value by dampening the tail of the self-similar profiles on $|x|\ge 1$: for some large constant $n_p$, let $\eta:\mathbb R_+\rightarrow\mathbb R$ be a smooth function
\begin{equation}
\label{eq: eta}
\eta(r)=\begin{cases} 
1 & r\le1 \\
r^{-n_p} & r\ge 2\end{cases}
\end{equation}
and define the dampened profile
$$
u_n^D(s,\rho)=\eta(e^{-s}\rho)u_n(\rho).
$$
We introduce the perturbation variables $(\Psi^D, \Omega^D)$:
$$
\tilde \Psi=\Psi+u_n=\Psi^D+\underbrace{\eta(e^{-s}\rho)u_n}_{=u_n^D},\quad \Omega-\Lambda u_n=\Omega^D-\eta(e^{-s}\rho)\Lambda u_n.
$$
Then the wave equation \eqref{eq: Transformed Nonlinear Wave} yields 
\begin{equation}
\label{eq: Wave Equation for Perturbation}
\begin{cases}
\partial_s\Psi^D=-\Lambda \Psi^D-\Omega^D\\
\partial_s\Omega^D=-\Delta\Psi^D-(\Lambda+1)\Omega^D-|\tilde \Psi|^{p-1}\tilde \Psi+\mathcal E(s,\rho)
\end{cases}
\end{equation}
where
\begin{equation}
\label{eq: E}
\mathcal E(s,\rho)= \eta(e^{-s}\rho) u_n^p-(\Delta\eta(e^{-s}\rho))u_n-2\nabla\eta(e^{-s}\rho) \cdot\nabla u_n.
\end{equation}

\subsection{Bootstrap bound and proof of Theorem \ref{theo: Result 2} }

The heart of the proof of Theorem \ref{theo: Result 2} is the following bootstrap proposition.

\begin{proposition}[Bootstrap]
\label{prop: Bootstrap}
Recall the definition of $H_k$ from Section \ref{sec: Notation}. Assume $d=3$, $k=3$ and write $$
\mathbb H=\mathbb H_3= H_4\times H_3.
$$
Given $c\ll 1$ and $s_0\gg1$ to be chosen in the proof, consider $X(s_0)\in \mathbb H$ such that 
\begin{equation}
\label{eq: Choice of Initial Value 1}
\Vert (I-P)X(s_0)\Vert_{\mathbb H}\le e^{-\frac{\delta}{2}s_0}, \quad \Vert PX(s_0)\Vert_{\mathbb H}\le e^{-\frac{\delta}{2}(1+\frac{c}{2})s_0}
\end{equation}
and for all $0\le j\le 4$,
\begin{equation}
\label{eq: Choice of Initial Value 2}
\bigg\Vert \frac{\langle\rho\rangle^{j+1}\nabla ^j\Psi^D(s_0)}{u_n^D}\bigg\Vert_{L^\infty(\mathbb R^3)}+\bigg\Vert \frac{\langle\rho\rangle^{j+1}\mathbbm 1_{j\ge1}\nabla ^{j-1}\Omega^D(s_0)}{u_n^D}\bigg\Vert_{L^\infty(\mathbb R^3)}\le e^{- \frac{\delta}{2}s_0}
\end{equation}
Define the exit time $s^*$ to be the maximal time such that the following bootstrap bounds hold on $s\in [s_0,s^*]$:
\begin{equation}
\label{eq: Bootstrap Bound 1}
\Vert e^{s\mathcal N}PX(s)\Vert_{\mathbb H}\le e^{-\frac{\delta}{2}(1+\frac{3c}{4})s},
\end{equation}
for $j=0, 1$ and $\kappa<\frac{1}{4(p+1)}$, 
\begin{equation}
\label{eq: Bootstrap Bound 2}
\bigg\Vert \frac{\rho^{j-\kappa}\nabla ^j\Psi^D(s)}{u_n^D}\bigg\Vert_{L^\infty(|y|\ge1)}\le 1,
\end{equation}
for all $0\le j \le 4$,
\begin{equation}
\label{eq: Bootstrap Bound 3}
I_j(s):=\int_{|y|\ge 1} \rho^{2j-2s_c}\xi(e^{-s}\rho)^{2n_p+1}\Big(|\nabla^j\Psi^D(s)|^2+\mathbbm 1_{j\ge1}|\nabla^{j-1}\Omega^D(s)|^2\Big)\,dy\le1
\end{equation}
where
$$
 \xi(r)=\eta(r)^{-\frac{1}{n_p}}=\begin{cases} 
1 & r\le1 \\
r & r\ge 2\end{cases}
$$
and for $\frac{\delta}{1+c}<\delta_0<\delta$,
\begin{equation}
\label{eq: Bootstrap Bound 4}
\Vert X(s)\Vert_{\mathbb H}\le e^{-\frac{\delta_0}{2}s}.
\end{equation}
Then the bootstrap bounds \eqref{eq: Bootstrap Bound 2}, \eqref{eq: Bootstrap Bound 3} and \eqref{eq: Bootstrap Bound 4} can be strictly improved in $s\in[s_0,s^*]$. Equivalently, if $s^*<\infty$, then equality holds for \eqref{eq: Bootstrap Bound 1} at $s=s^*$. Furthermore, the following non-linear bound holds:
\begin{equation}
\label{eq: Nonlinear Bound 1}
\forall s\in [s_0,s^*],\quad \Vert G(s)\Vert_{\mathbb H}\le e^{-\frac{\delta}{2}(1+c)s}.
\end{equation}
\end{proposition}

Let us assume Proposition \ref{prop: Bootstrap} and conclude the proof of Theorem \ref{theo: Result 2}.\\

\begin{proof}[proof of (Proposition \ref{prop: Bootstrap} $\Rightarrow$ Theorem \ref{theo: Result 2})]
\label{proof: Result 2} Assume Proposition \ref{prop: Bootstrap} holds. Let $s_0$ be as in Proposition \ref{prop: Bootstrap}. Note that the bootstrap bounds \eqref{eq: Bootstrap Bound 3} and \eqref{eq: Bootstrap Bound 4} imply
$$
\int_{\mathbb R^3} |\Psi^D|^2\,dy \le \int_{|y|\le1} |\Psi|^2\,dy+ \int_{|y|\ge1} \rho^{-2s_c+2n_p+1} |\Psi^D|^2\,dy<\infty
$$
and
$$
\int_{\mathbb R^3} |\Delta^2\Psi|^2\,dy<\infty.
$$
Then
$$
\Vert \tilde\Psi\Vert_{H^4(\mathbb R^3)}^2\le \Vert u_n^D\Vert_{L^2(\mathbb R^3)}^2+\Vert \Psi^D\Vert_{L^2(\mathbb R^3)}^2+\Vert u_n\Vert_{\dot H^4(\mathbb R^3)}^2+\Vert\Psi\Vert_{\dot H^4(\mathbb R^3)}^2<\infty.
$$\\
Similarly for $\Omega^D$. Thus, we infer
$$
\Vert u_n+\Psi\Vert_{H^4(\mathbb R^3)}+\Vert\Lambda u_n-\Omega\Vert_{H^3(\mathbb R^3)}\le C(s)
$$
for $s\in [s_0,s^*]$ so it follows that
$$
\Vert \Phi\Vert_{\dot H^{s_c}(\mathbb R^3)}+\Vert \partial_t\Phi\Vert_{\dot H^{s_c-1}(\mathbb R^3)}\le C(t)
$$
so the bootstrap time is strictly smaller than the life time provided by the standard Cauchy theory (see \cite{GX}).\\\\
We now conclude from the Brouwer fixed point argument. Note that for all initial data satisfying \eqref{eq: Choice of Initial Value 1} and \eqref{eq: Choice of Initial Value 2} in the space
$$
H=\bigg\{(\Psi^D,\Omega^D)\in (H^4\times H^3)(\mathbb R^3)\bigg|\, \sum_{j=0}^4\left\Vert \langle\rho\rangle^{j+\alpha+n_p+1}\Big(\left|\nabla^j\Psi^D\right|+\left|\mathbbm 1_{j\ge1}\nabla^{j-1}\Omega^D\right|\Big)\right\Vert_{L^\infty}<\infty\bigg\}.
$$
the non-linear bound \eqref{eq: Nonlinear Bound 1} and \eqref{eq: Bootstrap Bound 1} have been shown to hold on $[s_0,s^*]$. Then by Corollary \ref{cor: Exponential Decay 2}, $s^*\ge s_0+\Gamma$ for $\Gamma$ large. Moreover, as explained in the proof of Corollary \ref{cor: Exponential Decay 2}, given $(I-P)X(s_0)$, after a choice of projection of initial data on the subspace of unstable nodes $PX(s_0)$, the solution can be immediately propagated to any time $t<T$. This choice is dictated by Corollary \ref{cor: Exponential Decay 2}. Furthermore, this choice of $PX(s_0)$ is unique and is Lipschitz dependent on $(I-P)X(s_0)$ from Lemma \ref{lem: Lipschitz}.
\end{proof}

The rest of this section is devoted to the proof of the boostrap Proposition \ref{prop: Bootstrap}.


\subsection{Weighted Sobolev bounds}

Recall that we have set $d=3$, $k=3$. Then, we write $\mathbb H=\mathbb H_3$.

\begin{lemma}[Sobolev embedding]
\label{lem: L^infty to L^2 Non-radial}
Let $(\Psi^D,\Omega^D)$ be such that the right hand side of the bound \eqref{eq: L^infty to L^2 Non-radial} is finite. Then, for $j=0,1$,
\begin{equation}
\label{eq: L^infty to L^2 Non-radial}
\bigg\Vert \frac{\rho^{j-\kappa}\nabla ^j\Psi^D(s)}{u_n^D}\bigg\Vert_{L^\infty(|y|\ge1)}\lesssim \Vert\nabla^j\Psi^D\Vert_{L^\infty(|y|=2)}+\bigg(\sum_{l=0}^4I_l(s)\bigg)^{\frac{1}{2}}.
\end{equation}
\end{lemma}
\begin{proof}
\textbf{Step 1} (General bound):
 We recall the notations for the spherical harmonics from Section \ref{Sec: Dissipativity}. In particular, we write the spherical harmonic functions as $Y^{(l,m)}$ with eigenvalues $-\lambda_m=-m(m+1)$. We claim that given $i\in\mathbb N$ and $\beta\in\mathbb R$ and for all $f\in C_{c,\rad}^\infty(\mathbb R^3\setminus\{0\})$,
\begin{equation}
\label{eq: Bound for Spherical Harmonics}
\int_{\mathbb R^3} r^\beta|\nabla^i f(r)Y^{(l,m)}(\theta,\varphi)|^2\,dx = (1+o_{m\rightarrow\infty}(1))\underbrace{\sum_{j=0}^i \begin{pmatrix} i\\ j\end{pmatrix} \lambda_m^{i-j}\int_0^\infty r^{2+\beta+2(j-i)}|f^{(j)}|^2\,dr}_{:=S_{i,m}[f]}.
\end{equation}
We proceed by induction on $i$. Claim for $i=1,\,2$ is proved in Lemma 2.1 from \cite{CCF}. If claim holds for $i=2k-1,\,2k$, then by replacing $f$ in \eqref{eq: Bound for Spherical Harmonics} by $(\mathcal L-r^{-2}\lambda_m)f$ where we recall that $\mathcal L$ is the radial part of the Laplacian, we infer
$$
\begin{aligned}
&\int_{\mathbb R^3} r^\beta|\nabla^{i+2} f(r)Y^{(l,m)}(\theta,\varphi)|^2\,dx\\
=&(1+o_{m\rightarrow\infty}(1)) \sum_{j=0}^i \begin{pmatrix} i\\ j\end{pmatrix} \lambda_m^{i-j}\int_0^\infty r^{2+\beta+2(j-i)}|\partial_r^j(\partial_r^2+2r^{-1}\partial_r-r^{-2}\lambda_m)f|^2\,dr\\
=&\sum_{j=0}^i \begin{pmatrix} i\\ j\end{pmatrix} \lambda_m^{i-j}\int_0^\infty r^{2+\beta+2(j-i)}|(\partial_r^{j+2}-\lambda_mr^{-2}\partial_r^j)f|^2\,dr+o_{m\rightarrow\infty}(S_{i+2,m}[f])\\
=& \sum_{j=0}^i \begin{pmatrix} i\\ j\end{pmatrix} \lambda_m^{i-j}\int_0^\infty r^{2+\beta+2(j-i)}(|f^{(j+2)}|^2+2\lambda_m r^{-2}|f^{(j+1)}|^2+\lambda_m^2r^{-4}|f^{(j)}|^2)\,dr+o_{m\rightarrow\infty}(S_{i+2,m}[f])
\end{aligned}
$$
where in the last equality we have used integration by parts:
$$
\begin{aligned}
-&\lambda_m^{i-j+1}\int_0^\infty r^{\beta+2(j-i)}f^{(j+2)} f^{(j)}\,dr\\
=&\lambda_m^{i-j+1}\int_0^\infty r^{\beta+2(j-i)}|f^{(j+1)}|^2\,dr+C_{i,j,\beta}\lambda_m^{i-j+1}\int_0^\infty r^{-2+\beta+2(j-i)}|f^{(j)}|^2\,dr\\
=&\lambda_m^{i-j+1}\int_0^\infty r^{\beta+2(j-i)}|f^{(j+1)}|^2\,dr+o_{m\rightarrow\infty}(S_{i+2,m}[f]).
\end{aligned}
$$
Then, we infer
$$
\begin{aligned}
&\int_{\mathbb R^3} r^\beta|\nabla^{i+2} f(r)Y^{(l,m)}(\theta,\varphi)|^2\,dx\\
=&(1+o_{m\to\infty}(1)) \sum_{j=0}^{i+2} \left[\begin{pmatrix} i\\ j\end{pmatrix} +2\begin{pmatrix} i\\ j-1\end{pmatrix}+\begin{pmatrix} i\\ j-2\end{pmatrix}  \right]\lambda_m^{i+2-j}\int_0^\infty r^{2+\beta+2(j-i-2)}|f^{(j)}|^2\,dr
\end{aligned}
$$
Hence, the result follows for $i+2$. This concludes the proof of our claim \eqref{eq: Bound for Spherical Harmonics}.\\\\
\textbf{Step 2} (Interior Bound): From the claim, we have that for $M$ large, for all $f\in C_{c,\rad}^\infty(\mathbb R^3\setminus \{0\})$ and $m\ge M$,
$$
\sum_{j=0}^i\lambda_m^{i-j}\int_0^\infty \rho^{2j+2\alpha-1}|f^{(j)}|^2\,d\rho\lesssim_i \int_{\mathbb R^3}\rho^{2i+2\alpha-3}|\nabla^i f(\rho)Y^{(l,m)}|^2\,dx.
$$
Also, by induction on $i$, we have that for all $m<M$, 
\begin{equation}
\label{eq: Spherical Harmonics Decomposition}
\sum_{j=0}^i\int_0^\infty \rho^{2j+2\alpha-1}|f^{(j)}|^2\,d\rho\lesssim_i C_m\sum_{j=0}^i\int_{\mathbb R^3}\rho^{2j+2\alpha-3}|\nabla^j f(\rho)Y^{(l,m)}|^2\,dx.
\end{equation}
Thus, \eqref{eq: Spherical Harmonics Decomposition} holds for all $m\in \mathbb N$ with some universal constant independent of $m$. We now apply this to a function vanishing and $0$ and $\infty$.  Let $\chi_s\in C^\infty_{\rad}(\mathbb R^3)$ and $\varphi\in C^\infty(\mathbb R)$ be such that 
$$
\varphi(\rho)=\begin{cases} 0 & \rho\le 1\\ 1 & \rho\ge 2\end{cases},\quad
\chi_s(y)=\begin{cases}
\varphi(|y|) & |y| \le e^s\\
 1-\varphi(e^{-s}|y|)& |y|\ge e^s
\end{cases},
$$
Write 
$$
\Psi^D(y)=\sum_{l,m} \Psi_{l,m}^D(\rho) Y^{(l,m)}(\theta,\varphi),
$$
and apply \eqref{eq: Spherical Harmonics Decomposition} to $f(\rho)=\chi_s\Psi_{l,m}^D(\rho)$, we infer,
$$
\begin{aligned}
&\sum_{j=0}^{4}\lambda_m^{4-j}\int_2^{e^s} r^{2j+2\alpha-1}|\partial_\rho^j \Psi_{l,m}^D|^2\,d\rho\\
\le& \sum_{j=0}^{4}\lambda_m^{4-j}\int_0^\infty r^{2j+2\alpha-1}|\partial_\rho^j(\chi_s \Psi_{l,m}^D)|^2\,d\rho\\
\lesssim & \sum_{j=0}^4\int_{\mathbb R^3} \rho^{2j-2s_c}|\nabla^j(\chi_s\Psi_{l,m}^D(\rho)Y^{(l,m)}(\theta,\varphi))|^2dy\\
\lesssim & \sum_{j=0}^4\int_{|y|\ge1}\rho^{2j-2s_c}\xi(e^{-s}\rho)^{2n_p+1}|\nabla^j(\Psi_{l,m}^D(\rho)Y^{(l,m)}(\theta,\varphi))|^2\,dy
\end{aligned}
$$
where in the last inequality we have used that for all $e^s\le \rho\le e^{2s}$,
$$
|\partial_\rho^j\chi_s(\rho)|\lesssim_j e^{-js}\lesssim \rho^{-j}.
$$
Since the universal constant here does not depend on $m$, we sum over $l$ and $m$ to infer
$$
\sum_{l,m}\sum_{j=0}^{4}\lambda_m^{4-j}\int_2^{e^s} \rho^{2j+2\alpha-1}|\partial_\rho^j \Psi_{l,m}^D|^2\,d\rho \lesssim \sum_{j=0}^4 I_j(s).
$$
Note the universal $L^\infty$-bound for spherical harmonics which one can find in \cite{TZ} states that
$$
\Vert Y^{(l,m)}(\theta,\varphi)\Vert_{L^\infty(\mathbb S^2)}\lesssim \lambda_m^{\frac{1}{4}}.
$$
Thus, we infer for $2\le |y|\le e^s$,
$$
\begin{aligned}
&\bigg|\frac{\rho^{-\kappa}\Psi^D(y)}{u_n^D}\bigg|\lesssim \Vert \Psi^D\Vert_{L^\infty(|y|=2)}+\sum_{l,m}\Vert Y^{(l,m)}\Vert_{L^\infty(\mathbb S^2)}\int_2^{e^s}|\partial_\rho(\rho^{\alpha-\kappa}\Psi_{l,m}^D)|\,d\rho\\
\lesssim& \Vert \Psi^D\Vert_{L^\infty(|y|=2)}+\sum_{l,m}\lambda_m^{\frac{1}{4}}\bigg(\int_2^{e^s} \rho^{-1-2\kappa}\,d\rho\bigg)^{\frac{1}{2}}\bigg(\int_2^{e^s}\rho^{2\alpha-1}(|\Psi_{l,m}^D|^2+\rho^2|\partial_\rho\Psi_{l,m}^D|^2)\,d\rho\bigg)^{\frac{1}{2}}\\
\le& \Vert \Psi^D\Vert_{L^\infty(|y|=2)}+\bigg(\sum_{l,m} \lambda_m^{-\frac{3}{2}}\bigg)^{\frac{1}{2}}\bigg(\sum_{l,m}\lambda_m^2\int_2^{e^s}\rho^{2\alpha-1}(|\Psi_{l,m}^D|^2+\rho^2|\partial_\rho\Psi_{l,m}^D|^2)\,d\rho\bigg)^{\frac{1}{2}}\\
\lesssim &\Vert \Psi^D\Vert_{L^\infty(|y|=2)}+ \bigg(\sum_{l=0}^4I_l(s)\bigg)^{\frac{1}{2}}.
\end{aligned}
$$
Next, we bound the derivatives of $\Psi^D$. Explicit calculation of the derivatives of $Y^{(l,m)}$ yields
$$
\Vert \partial_\theta Y^{(l,m)}\Vert_{L^\infty(\mathbb S^2)}+\Vert \partial_\varphi Y^{(l,m)}\Vert_{L^\infty(\mathbb S^2)}\lesssim \lambda_m^{\frac{3}{4}}.
$$
Then, for $2\le |y|\le e^s$, by writing $(\tilde y_1,\tilde y_2,\tilde y_3)=(\rho,\theta,\varphi)$ and $(n_1,n_2,n_3)=(0,-1,-1)$, we infer
$$
\begin{aligned}
&\bigg|\frac{\rho^{1-\kappa}\nabla\Psi^D(y)}{u_n^D}\bigg|\lesssim \Vert \nabla\Psi^D\Vert_{L^\infty(|y|=2)}+\sum_{i=1}^3\int_2^{e^s}\sup_{\mathbb S^2}|\partial_\rho(\rho^{\alpha+1+n_i-\kappa}\partial_{\tilde y_i}(\Psi_{l,m}^D(\rho)Y^{(l,m)}))|\,d\rho\\
\lesssim & \Vert \nabla\Psi^D\Vert_{L^\infty(|y|=2)}+\underbrace{\sum_{l,m}\lambda_m^{\frac{1}{4}}\int_2^{e^s}|\partial_\rho(\rho^{\alpha+1-\kappa}\partial_\rho\Psi_{l,m}^D)|\,d\rho}_{=\partial_\rho \text{ term}}+\underbrace{\sum_{l,m}\lambda_m^{\frac{3}{4}}\int_2^{e^s}|\partial_\rho(\rho^{\alpha-\kappa}\Psi_{l,m}^D)|\,d\rho.}_{=\partial_\theta,\,\partial_\varphi\text{ terms}}
\end{aligned}
$$
Then, as before,
$$
\begin{aligned}
&(\partial_\rho \text{ term}) \lesssim\sum_{l,m}\lambda_m^{\frac{1}{4}}\bigg(\int_2^{e^s} \rho^{-1-2\kappa}\,d\rho\bigg)^{\frac{1}{2}}\bigg(\int_2^{e^s}\rho^{2\alpha+1}(|\partial_\rho\Psi_{l,m}^D|^2+\rho^2|\partial_\rho^2\Psi_{l,m}^D|^2)\,d\rho\bigg)^{\frac{1}{2}}\\
\le& \bigg(\sum_{l,m} \lambda_m^{-\frac{3}{2}}\bigg)^{\frac{1}{2}}\bigg(\sum_{l,m}\lambda_m^2\int_2^{e^s}\rho^{2\alpha+1}(|\partial_\rho\Psi_{l,m}^D|^2+\rho^2|\partial_\rho^2\Psi_{l,m}^D|^2)\,d\rho\bigg)^{\frac{1}{2}}\lesssim \bigg(\sum_{l=0}^4I_l(s)\bigg)^{\frac{1}{2}}
\end{aligned}
$$
and similarly,
$$
(\partial_\theta,\,\partial_\varphi\text{ terms})\lesssim\bigg(\sum_{l,m} \lambda_m^{-\frac{3}{2}}\bigg)^{\frac{1}{2}}\bigg(\sum_{l,m}\lambda_m^3\int_2^{e^s}\rho^{2\alpha-1}(|\Psi_{l,m}^D|^2+\rho^2|\partial_\rho\Psi_{l,m}^D|^2)\,d\rho\bigg)^{\frac{1}{2}}\lesssim\bigg(\sum_{l=0}^4I_l(s)\bigg)^{\frac{1}{2}}.
$$
Thus, we infer for all $2\le |y|\le e^s$ that
$$
\bigg|\frac{\rho^{1-\kappa}\nabla\Psi^D(y)}{u_n^D}\bigg|\lesssim
\Vert \nabla\Psi^D\Vert_{L^\infty(|y|=2)}+\bigg(\sum_{l=0}^4I_l(s)\bigg)^{\frac{1}{2}}.
$$
\textbf{Step 3} (Exterior Bound): We now propagate the $L^\infty$-bound to the region outside of the self-similar scale. From the claim in \textbf{Step 1}, we infer the bound
$$
\sum_{j=0}^i\lambda_m^{i-j}\int_0^\infty \rho^{2j+2\alpha+2n_p}|f^{(j)}|^2\,d\rho\lesssim_i\sum_{j=0}^i\int_{\mathbb R^3}\rho^{2j+2\alpha+2n_p-2}|\nabla^i f(\rho)Y^{(l,m)}|^2\,dy
$$
with some universal constant independent of $m$. Using the same $\eta$ and decomposition of $\Psi^D$ as in \textbf{Step 2} and apply the above bound with $f(\rho)=\tilde \chi_s\Psi_{l,m}^D(\rho)$ for a cut-off $\tilde \chi_s(y)=\varphi(2e^{-s}|y|)$, we infer
$$
\begin{aligned}
&\sum_{j=0}^{4}\lambda_m^{4-j}\int_{e^s}^\infty r^{2j+2\alpha-1}\xi(e^{-s}\rho)^{2n_p+1}|\partial_\rho^j \Psi_{l,m}^D|^2\,d\rho\\
\lesssim & \sum_{j=0}^4\int_{|y|\ge1}\rho^{2j-2s_c}\xi(e^{-s}\rho)^{2n_p+1}|\nabla^j(\Psi_{l,m}^D(\rho)Y^{(l,m)}(\theta,\varphi))|^2\,dy
\end{aligned}
$$
Thus, as in \textbf{Step 2}, we infer
$$
\sum_{l,m}\sum_{j=0}^{4}\lambda_m^{4-j}\int_{e^s}^\infty r^{2j+2\alpha-1}\xi(e^{-s}\rho)^{2n_p+1}|\partial_\rho^j \Psi_{l,m}^D|^2\,d\rho\lesssim\sum_{j=0}^4I_j(s).
$$
Thus, we infer for $|y|\ge e^s$,
$$
\bigg|\frac{\rho^{-\kappa}\Psi^D(y)}{u_n^D}\bigg|\lesssim \bigg\Vert \frac{\rho^{-\kappa}\Psi^D}{u_n^D}\bigg\Vert_{L^\infty(|y|=e^s)}+\sum_{l,m}\lambda_m^{\frac{1}{4}}\int_{e^s}^\infty e^{-n_ps}|\partial_\rho(\rho^{\alpha-\kappa}\Psi_{l,m}^D)|\,d\rho.
$$
Since
$$
\begin{aligned}
&\sum_{l,m}\lambda_m^{\frac{1}{4}}\int_{e^s}^\infty e^{-n_ps}|\partial_\rho(\rho^{\alpha-\kappa}\Psi_{l,m}^D)|\,d\rho\\
\lesssim&\sum_{l,m}\lambda_m^{\frac{1}{4}}\bigg(\int_{e^s}^\infty e^s\rho^{-2-2\kappa}\,d\rho\bigg)^{\frac{1}{2}}\bigg(\int_{e^s}^\infty\rho^{2\alpha-1}\xi(e^{-s}\rho)^{2n_p+1}(|\Psi_{l,m}^D|^2+\rho^2|\partial_\rho\Psi_{l,m}^D|^2)\,d\rho\bigg)^{\frac{1}{2}}\\
\le&\bigg(\sum_{l,m} \lambda_m^{-\frac{3}{2}}\bigg)^{\frac{1}{2}}\bigg(\sum_{l,m}\lambda_m^2\int_{e^s}^\infty\rho^{2\alpha-1}\xi(e^{-s}\rho)^{2n_p+1}(|\Psi_{l,m}^D|^2+\rho^2|\partial_\rho\Psi_{l,m}^D|^2)\,d\rho\bigg)^{\frac{1}{2}}\\
\lesssim &\bigg(\sum_{l=0}^4I_l(s)\bigg)^{\frac{1}{2}},
\end{aligned}
$$
combining with the interior bound, we infer \eqref{eq: L^infty to L^2 Non-radial} for $\Psi^D$. As in \textbf{Step 2}, we can bound the derivatives of $\Psi^D$ in the region $|y|\ge e^s$. This concludes the proof of \eqref{eq: L^infty to L^2 Non-radial}.
\end{proof}


\subsection{Proof of Proposition \ref{prop: Bootstrap}}


We are in position to prove Proposition \ref{prop: Bootstrap}.\\

\begin{proof}[proof of Proposition \ref{prop: Bootstrap}]
\textbf{Step 1} (Energy estimates): We claim the energy estimate
\begin{equation}
\label{eq: Energy Estimate}
\frac{dI_j}{ds}\lesssim e^{-\varepsilon s}
\end{equation}
holds for some $\varepsilon>0$ for all $0\le j\le 4$ so in particular, by the choice of initial value \eqref{eq: Choice of Initial Value 2},
$$
I_j(s)\le I_j(s_0)+ Ce^{-\varepsilon s_0} 
$$
is arbitrarily small for $s_0$ sufficiently large.\\\\
\textbf{Case 1} ($1\le j\le 4$): Suppose claim holds for $< j$ cases. Denote by $I_j^\Psi$, $I_j^\Omega$ the weighted $L^2$-norm of $\Psi^D$ and $\Omega^D$ in $I_j$. For the $\Psi^D$ component, we infer
$$
\begin{aligned}
\frac{dI_j^\Psi}{ds}&=\int_{|y|\ge1}\rho^{2j-2s_c}\bigg[-\rho\frac{\partial}{\partial \rho}\xi(e^{-s}\rho)^{2n_p+1}|\nabla^j\Psi^D|^2+2\xi(e^{-s}\rho)^{2n_p+1}\nabla^j\Psi^D\cdot\partial_s\nabla^j\Psi^D\bigg]\,dy\\
&\le 2\int_{|y|\ge1} \rho^{2j-2s_c}\xi(e^{-s}\rho)^{2n_p+1}\Big[(j+\Lambda+\partial_s)\nabla^j\Psi^D\Big]\cdot\nabla^j\Psi^D\,dy
\end{aligned}
$$
where we integrate by parts for the last inequality and note that the boundary terms are non-positive. By the commutation relations
$$
[\nabla^k,\Lambda]=k\nabla^k,
$$
and \eqref{eq: Wave Equation for Perturbation}, we infer
$$
\begin{aligned}
\frac{dI_j^\Psi}{ds} &\le 2\int_{|y|\ge1}\rho^{2j-2s_c}\xi(e^{-s}\rho)^{2n_p+1}\nabla^j(\Lambda+\partial_s)\Psi^D\cdot\nabla^j\Psi^D\,dy\\
&=-2\int_{|y|\ge1} \rho^{2j-2s_c}\xi(e^{-s}\rho)^{2n_p+1}\nabla^j\Omega^D\cdot\nabla^j\Psi^D\,dy
\end{aligned}
$$
Similarly, for $\Omega^D$ component, it follows from the above commutation relation and \eqref{eq: Wave Equation for Perturbation} that
\begin{equation}
\label{eq: Omega Energy}
\begin{aligned}
\frac{dI_j^\Omega}{ds} &\le 2\int_{|y|\ge1}\rho^{2j-2s_c}\xi(e^{-s}\rho)^{2n_p+1}\Big[(j+\Lambda+\partial_s)\nabla^{j-1}\Omega^D\Big]\cdot\nabla^{j-1}\Omega^D\,dy\\
&=2\int_{|y|\ge1} \rho^{2j-2s_c}\xi(e^{-s}\rho)^{2n_p+1}\nabla^{j-1}(-\Delta  \Psi^D -\tilde \Psi^p+\mathcal E)\cdot\nabla^{j-1}\Omega^D\,dy.
\end{aligned}
\end{equation}
where we recall the definition \eqref{eq: E} of $\mathcal E$. Integrate by parts the first term we infer
\begin{equation}
\label{eq: Cross Term}
\begin{aligned}
&2\int_{|y|\ge1} \rho^{2j-2s_c}\xi(e^{-s}\rho)^{2n_p+1}(-\nabla^{j+1} \Psi^D)\cdot\nabla^{j-1}\Omega^D\,dy\\
\le&\,2\int_{|y|\ge1} \rho^{2j-2s_c}\xi(e^{-s}\rho)^{2n_p+1}\nabla^j\Psi^D\cdot\nabla^j\Omega^D\,dy\\
+&\,2\int_{|y|\ge1} \nabla\Big[\rho^{2j-2s_c}\xi(e^{-s}\rho)^{2n_p+1}\Big]\cdot\nabla^j\Psi^D\,\nabla^{j-1}\Omega^D\,dy.
\end{aligned}
\end{equation}
From the bootstrap bound \eqref{eq: Bootstrap Bound 4} and \eqref{eq: Bootstrap Bound 3} , we infer for $2\varepsilon<\frac{\delta_0}{2k+1-2s_c}=\frac{\delta_0}{7-2s_c}$, the bound for the last term above
$$
\begin{aligned}
&\int_{|y|\ge1} \rho^{2j-2s_c-1}\xi(e^{-s}\rho)^{2n_p+1}|\nabla^j \Psi^D| \,|\nabla^{j-1}\Omega^D|\,dy\\
\le &\, e^{-\varepsilon s}\int_{|y|\ge e^{\varepsilon s}} \rho^{2j+2s_c}\xi(e^{-s}\rho)^{2n_p+1}|\nabla^j \Psi^D |\,|\nabla^{j-1}\Omega^D|\,dy\\
+&\,e^{\varepsilon(2k+1-2s_c)s}\int_{1\le|y|\le e^{\varepsilon s}} \rho^{-2(k+1-j)}|\nabla^j\Psi |\,|\nabla^{j-1}\Omega|\,dy\\
\le &\,e^{-\varepsilon s}(I_j^\Psi I_j^\Omega)^{\frac{1}{2}}+e^{\frac{\delta_0}{2} s}\int_{|y|\le e^{\varepsilon s}}\left(|\nabla^j \Psi |^2+|\nabla^{j-1}\Omega |^2\right)\langle\rho\rangle^{-2(4-j)}\,dy\\
\le &\,e^{-\varepsilon s}I_j+e^{\frac{\delta_0}{2}s}\Vert X\Vert_{\mathbb H}^2\le e^{-\varepsilon s}I_j+e^{-\frac{\delta_0}{2}s}\le e^{-\varepsilon s}
\end{aligned}
$$
for some $\varepsilon>0$. Note that we have used Hardy's inequality from Lemma \ref{lem: Compact Embedding}:
\begin{equation}
\label{eq: Hardy on Psi}
\int_{|y|\le e^{\varepsilon s}}|\nabla^j \Psi |^2\langle\rho\rangle^{-2(4-j)}\,dy\lesssim\Vert \Psi\Vert_{H_4}^2 \le\Vert X\Vert_{\mathbb H}^2
\end{equation}
and similarly for $\Omega$. Thus, we infer the bound for \eqref{eq: Cross Term}:
$$
\begin{aligned}
&2\int_{|y|\ge1} \rho^{2j-2s_c}\xi(e^{-s}\rho)^{2n_p+1}(-\nabla^{j+1} \tilde \Psi) \cdot\nabla^{j-1}\Omega^D\,dy\\
\le & \,2\int_{|y|\ge1} \rho^{2j-2s_c}\xi(e^{-s}\rho)^{2n_p+1}\nabla^j\Psi^D\cdot\nabla^j\Omega^D\,dy+C e^{-\varepsilon s}.
\end{aligned}
$$
Next, we prove the bound for the term with $\tilde \Psi^p=(u_n+\Psi)^p$ and $\mathcal E$. By the bootstrap bound \eqref{eq: Bootstrap Bound 2} together with the asymptotic behaviour of $u_n^D$, it holds for $l=0,1$ that
$$
\bigg\Vert \frac{\rho^{l-\kappa}\nabla ^l\tilde \Psi(s)}{u_n^D}\bigg\Vert_{L^\infty(|y|\ge1)}\lesssim1,
$$
we infer for all $\rho\ge 1$ and $j\le 4$,
$$
\begin{aligned}
\left|\nabla^{j-1}\left(|\tilde \Psi|^{p-1}\tilde \Psi\right)\right|&\lesssim \sum_{i=1}^{j-1} |\tilde \Psi|^{p-i}\sum_{|\alpha|=j-1, \alpha>0} |\nabla^{\alpha_1}\tilde \Psi|\cdots |\nabla^{\alpha_i}\tilde \Psi|\\
&\lesssim \sum_{l=1}^{j-1}|\nabla^l\tilde \Psi|\sum_{i=1}^{j-1}|\tilde \Psi|^{p-i}\sum_{\substack{|\alpha|=j-l-1,\\\Vert \alpha\Vert_\infty\le 1}}|\nabla^{\alpha_1}\tilde \Psi|\cdots |\nabla^{\alpha_{i-1}}\tilde \Psi|\\
&\lesssim \sum_{l=0}^{j-1}\rho^{-j+l+1+(-\alpha+\kappa)(p-1)}|\nabla^l\tilde \Psi|\le \sum_{l=0}^{j-1}\rho^{-j+l-\frac{3}{4}} |\nabla^l\tilde \Psi|
\end{aligned}
$$
where we have used that $\kappa<\frac{1}{4(p-1)}$ and that $p> 5$ to bound $|\tilde\Psi|^{p-i}$. \\\\
Next, we bound $\mathcal E$ where we recall the definition \eqref{eq: E} of $\mathcal E$. Observe that
$$
\partial_\rho^j \eta(e^{-s}\rho)= e^{-js}  \eta^{(j)}(e^{-s}\rho)\lesssim \rho^{-j} \eta(e^{-s}\rho).
$$
In view of the asymptotic behaviours of $u_n^D$ and its derivatives, we have that for all $\rho\ge1$ and $j\le 4$,
$$
|\nabla^{j-1}\mathcal E|\lesssim \Big|\nabla^{j-1}\Big( \eta(e^{-s}\rho) u_n^p-(\Delta\eta(e^{-s}\rho))u_n-2e^{-s}\eta'(e^{-s}\rho) u_n'\Big)\Big| \lesssim \rho^{-j-1}u_n^D.
$$
Adding the two bounds obtained above, we infer for $\rho\ge1$ that
\begin{equation}
\label{eq: Inner Nonlinear Bound}
\left|\nabla^{j-1}\left(|\tilde \Psi|^{p-1}\tilde \Psi-\mathcal E\right)\right|\lesssim \sum_{l=0}^{j-1}\rho^{-j+l-\frac{3}{4}}\Big (|\nabla^l \Psi^D|+|\nabla^l u_n^D|\Big).
\end{equation}
We improve the above bound in the region $1\le\rho\le e^s$. Here, $\eta(e^{-s}\rho)\equiv1$ so $\mathcal E=u_n^p$ and we infer for $j\le 4$,
$$
\begin{aligned}
&\left|\nabla^{j-1}\left(|\tilde \Psi|^{p-1}\tilde \Psi-\mathcal E\right)\right|\lesssim \left |\nabla^{j-1}\left(\Psi\int_0^1 |u_n+\tau\Psi|^{p-1}d\tau\right)\right|\\
\lesssim &\sup_{0\le \tau\le1}|u_n+\tau\Psi|^{p-4}\sum_{i=0}^{j-1}|\nabla^i \Psi| \sum_{\substack{|\alpha|=j-1-i,\\ \alpha_1\ge\alpha_2\ge\alpha_3}} \prod_{q=1}^3\sup_{0\le \tau\le1}\left|\nabla^{\alpha_q}(u_n+\tau\Psi)\right|
\end{aligned}
$$
Since $i+\alpha_1+\alpha_2+\alpha_3=j-1\le 3$ in the sum above, $\alpha_2,\alpha_3\le 1$ so the $L^\infty$-bound \eqref{eq: Bootstrap Bound 2} applies. Then, we have for all $1\le\rho\le e^s$ that
\begin{equation}
\label{eq: Exterior Nonlinear Bound}
\begin{aligned}
\left|\nabla^{j-1}\left(|\tilde \Psi|^{p-1}\tilde \Psi-\mathcal E\right)\right|&\lesssim \rho^{(-\alpha+\kappa)(p-2)}\sum_{i=0}^{j-1}|\nabla^i\Psi| \Big(|\nabla^{j-1-i}u_n|+|\nabla^{j-1-i}\Psi|\Big)\\
&\lesssim \sum_{i=0}^{j-1}\rho^{-j+i+1+(-\alpha+\kappa)(p-1)}|\nabla^i \Psi^D| \lesssim\sum_{i=0}^{j-1}\rho^{-j+i-\frac{3}{4}}|\nabla^i \Psi^D|.
\end{aligned}
\end{equation}
Thus, using the bounds \eqref{eq: Interior Nonlinear Bound} and \eqref{eq: Exterior Nonlinear Bound} above, we infer for the $\tilde \Psi^p$ and $\mathcal E$ terms in \eqref{eq: Omega Energy} that
$$
\begin{aligned}
&\int_{|y|\ge 1} \rho^{2j-2s_c}\xi(e^{-s}\rho)^{2n_p+1}|\nabla^{j-1}(\tilde \Psi^p-\mathcal E)|\,|\nabla^{j-1}\Omega^D|\,dy\\
\lesssim&\,\sum_{l=0}^{j-1}\int_{|y|\ge 1} \rho^{j+l-2s_c-\frac{1}{2}}\xi(e^{-s}\rho)^{2n_p+1}\Big(|\nabla^l\Psi^D|+\mathbbm 1_{|y|\ge e^s}|\nabla^l u_n^D|\Big)|\nabla^{j-1}\Omega^D|\,dy\\
\le &\, \sum_{l=0}^{j-1}\int_{|y|\ge e^s} \rho^{j+l-2s_c-\frac{1}{2}}\xi(e^{-s}\rho)^{2n_p+1}|\nabla^lu_n^D||\nabla^{j-1}\Omega^D|\,dy\\
+&\,e^{-\frac{\varepsilon}{2} s}\sum_{l=0}^{j-1}\int_{|y|\ge e^{\varepsilon s}} \rho^{j+l-2s_c}\xi(e^{-s}\rho)^{2n_p+1}|\nabla^l \Psi^D |\,|\nabla^{j-1}\Omega^D|\,dy\\
+&\,\sum_{l=0}^{j-1}\int_{1\le|y|\le e^{\varepsilon s}} \rho^{j+l-2s_c-\frac{1}{2}}|\nabla^l\Psi |\,|\nabla^{j-1}\Omega|\,dy.
\end{aligned}
$$
Thus, from the bootstrap bound \eqref{eq: Bootstrap Bound 4} and Hardy's inequality \eqref{eq: Hardy on Psi}, we infer for $2\varepsilon<\frac{\delta_0}{2k+\frac32-2s_c}=\frac{\delta_0}{\frac{15}{2}-2s_c}$, the bound
$$
\begin{aligned}
&\int_{|y|\ge 1} \rho^{2j-2s_c}\xi(e^{-s}\rho)^{2n_p+1}|\nabla^{j-1}(\tilde \Psi^p-\mathcal E)|\,|\nabla^{j-1}\Omega^D|\,dy\\
\le &\sum_{l=0}^{j-1}\bigg[\bigg(\int_{e^s}^\infty \rho^{-\frac52}\xi(e^{-s}\rho)\,d\rho\bigg)^\frac{1}{2}(I_j^\Omega)^{\frac12}+e^{-\frac{\varepsilon}{2} s}(I_l^\Psi I_j^\Omega)^{\frac{1}{2}}\\
+&\,e^{\frac{\delta_0}{2} s}\bigg(\int_{|y|\le e^{\varepsilon s}}|\nabla^l \Psi |^2\langle\rho\rangle^{-2(4-l)}\,dy\bigg)^{\frac{1}{2}}\bigg(\int_{|y|\le e^{\varepsilon s}} |\nabla^{j-1}\Omega|^2\langle\rho\rangle^{-2(4-j)}\,dy\bigg)^{\frac{1}{2}}\bigg]\\
\le &e^{-\frac34 s}I_j+\sum_{l=0}^j e^{-\frac{\varepsilon}{2} s} I_l+e^{\frac{\delta_0}{2}s}\Vert X\Vert_{\mathbb H}^2\lesssim e^{-\frac34 s}+ e^{-\frac{\varepsilon}{2} s}+e^{-\frac{\delta_0}{2}s}
\end{aligned}
$$
Take smaller $\varepsilon$ if necessary, we infer
$$
\frac{dI_j^\Omega}{ds}\le 2\int_{|y|\ge 1} \rho^{2j-2s_c}\xi(e^{-s}\rho)^{2n_p+1}\nabla^j\Psi^D\cdot\nabla^j\Omega^D\,dy+C e^{-\varepsilon s}
$$
Hence, by adding the bounds for $I_j^\Psi$ and $I_j^\Omega$, we obtain the overall bound
\begin{equation}
\label{eq: Case 1 Energy Estimate}
\frac{dI_j}{ds} \lesssim e^{-\varepsilon s}
\end{equation}
i.e. the claim \eqref{eq: Energy Estimate} holds.\\\\
\textbf{Case 2} ($j=0$): Note that $I_0=I_0^\Psi$. As in Case 1,
$$
\frac{dI_0}{ds} \le -2\int_{|y|\ge 1} \rho^{-2s_c}\xi(e^{-s}\rho)^{2n_p+1}\Omega^D\Psi^D\,dy.
$$
From the bootstrap bound \eqref{eq: Bootstrap Bound 4} and \eqref{eq: Bootstrap Bound 3} , we infer for $2\varepsilon<\frac{\delta_0}{2k+1-2s_c}=\frac{\delta_0}{7-2s_c}$, the bound the above:
$$
\begin{aligned}
&\int_{|y|\ge1} \rho^{-2s_c}\xi(e^{-s}\rho)^{2n_p+1}|\Psi^D\Omega^D|\,dy\\
\le &\, e^{-\varepsilon s}\int_{|y|\ge e^{\varepsilon s}} \rho^{1+2s_c}\xi(e^{-s}\rho)^{2n_p+1}| \Psi^D \Omega^D|\,dy+e^{\varepsilon(2k+1-2s_c)s}\int_{1\le|y|\le e^{\varepsilon s}} \rho^{-(2k+1)}|\Psi \Omega|\,dy\\
\le &\,e^{-\varepsilon s}(I_0^\Psi I_1^\Omega)^{\frac{1}{2}}+e^{\frac{\delta_0}{2} s}\int_{|y|\le e^{\varepsilon s}}\left(| \Psi |^2\langle\rho\rangle^{-2(k+1)}+|\Omega |^2\langle\rho\rangle^{-2k}\right)\,dy\le e^{-\varepsilon s}+e^{-\frac{\delta_0}{2}s}\le e^{-\varepsilon s}
\end{aligned}
$$
for some $\varepsilon>0$. Hence, the claim.\\\\
\textbf{Step 2} (Improvement of \eqref{eq: Bootstrap Bound 2} and \eqref{eq: Bootstrap Bound 3}): Given $d_0\ll 1$, we claim that these quantities can be bounded by $d_0$ in $s\in [s_0,s^*]$.\\\\
\textbf{Improved bound for the weighted Sobolev norm}: It follows from the energy estimate \eqref{eq: Energy Estimate} and the choice of initial value \eqref{eq: Choice of Initial Value 2} that given $d_0\ll1$, we have that for all $s\in [s_0,s^*]$ and $0\le j\le 4$,
\begin{equation}
\label{eq: Improved Bound 1}
I_j(s)\le I_j(s_0)+ Ce^{-\varepsilon s_0} \le d_0
\end{equation}
for $s_0$ sufficiently large.\\\\
\textbf{Improved pointwise bound}: Let $0\le j\le 1$. By Sobolev embedding and \eqref{eq: Bootstrap Bound 4}, we infer for large $s_0$ that
$$
\Vert\nabla^j\Psi^D\Vert_{L^\infty(|y|\le2)}\ll d_0.
$$
Then, by Lemma \ref{lem: L^infty to L^2 Non-radial}, we have that for $0\le j\le 1$, 
\begin{equation}
\label{eq: Apply L^infty Bound}
\bigg\Vert \frac{\rho^{j-\kappa}\nabla ^j\Psi^D}{u_n^D}\bigg\Vert_{L^\infty(|y|\ge1)}\lesssim \Vert\nabla^j\Psi^D\Vert_{L^\infty(|y|=2)}+\bigg(\sum_{l=0}^{4}I_l(s)\bigg)^{\frac{1}{2}}\le d_0.
\end{equation}
where the last inequality follows from \eqref{eq: Improved Bound 1}.\\\\
\textbf{Step 3} (Improved $\Vert \cdot\Vert_{\mathbb H}$ bound and non-linear bound): Recall that
\begin{equation}
\label{eq: Integral Form of Nonlinearity}
\begin{aligned}
G_\Omega=&-|\Psi+u_n|^{p-1}(\Psi+u_n)+u_n^p+pu_n^{p-1}\Psi\\
=&-p(p-1)\Psi^2\int_0^1(1-\tau)|u_n+\tau\Psi|^{p-3}(u_n+\tau\Psi)\,d\tau.
\end{aligned}
\end{equation}
We claim that by choosing $s_0$ sufficiently large and $c>0$ small,
\begin{equation}
\label{eq: Nonlinear Bound 2}
\forall s\in [s_0,s^*],\quad \Vert G(s)\Vert_{\mathbb H}\le\Vert X(s)\Vert_{\mathbb H}^{1+c}.
\end{equation}
Let $\rho\ge 1$. Then,
\begin{equation}
\label{eq: Derivative Nonlinearity}
|\nabla^k G_\Omega|\lesssim \sum_{i+j+l=k}|\nabla^i\Psi||\nabla^j\Psi| \left|\int_0^1(1-\tau)\nabla^l\left(|u_n+\tau\Psi|^{p-3}(u_n+\tau\Psi)\right)\,d\tau\right|.
\end{equation}
 For $m\le 3$ and $p>5$, we have the bound: 
\begin{equation}
\label{eq: Nonlinearity}
\left|\int_0^1(1-\tau)|u_n+\tau\Psi|^{p-m-3}(u_n+\tau\Psi)\,d\tau\right|\lesssim \sup_{0\le\tau\le1}|u_n+\tau\Psi|^{p-m-2} 
\end{equation}
This, together with the $L^\infty$-bound \eqref{eq: Bootstrap Bound 2} which implies $|\Psi|\lesssim \langle\rho\rangle^{-\alpha+\kappa}$ and the asymptotic behaviour of $u_n$, we infer
\begin{equation}
\label{eq: Nonlinearity Expansion}
\begin{aligned}
&\left|\int_0^1(1-\tau)\nabla^l\left(|u_n+\tau\Psi|^{p-3}(u_n+\tau\Psi)\right)\,d\tau\right|\\
\lesssim& \sum_{m=0}^l\int_0^1(1-\tau)|u_n+\tau\Psi|^{p-m-3}(u_n+\tau\Psi)\,d\tau\sum_{|\alpha|=l}\,\prod_{q=1}^m(|\nabla^{\alpha_q}u_n|+|\nabla^{\alpha_q}\Psi|)\\
\lesssim &\sum_{m=0}^l\rho^{(-\alpha+\kappa)(p-m-2)}\sum_{|\alpha|=l}\,\prod_{q=1}^m(|\nabla^{\alpha_q}u_n|+|\nabla^{\alpha_q}\Psi|).
\end{aligned}
\end{equation}
Note that \eqref{eq: Nonlinearity} applies since $m\le l\le k=3$. Also, at most one of $\alpha_1,$ $i$, $j$ is $>1$ i.e. we can apply the $L^\infty$-bound \eqref{eq: Bootstrap Bound 2} for at least two of $\nabla^{\alpha_1}\Psi$, $\nabla^i\Psi$, $\nabla^j\Psi$ factors. Thus, we infer
$$
\begin{aligned}
|\nabla^k G_\Omega|&\lesssim \sum_{i+j+l=k}|\nabla^i\Psi||\nabla^j\Psi| \sum_{m=0}^l\rho^{(-\alpha+\kappa)(p-m-2)}\sum_{|\alpha|=l}\,\prod_{q=1}^m(|\nabla^{\alpha_q}u_n|+|\nabla^{\alpha_q}\Psi|)\\
\lesssim&\sum_{i+j+l=k}\rho^{-\alpha-j+\kappa}|\nabla^i\Psi|\sum_{m=0}^l\rho^{(-\alpha+\kappa)(p-m-2)}\rho^{m(-\alpha+\kappa)-l}\\
\lesssim& \sum_{i=0}^k\rho^{(-\alpha+\kappa)(p-1)+i-k}|\nabla^i\Psi|\lesssim \sum_{i=0}^k \rho^{i-k-\frac{3}{2}}|\nabla^i\Psi|.
\end{aligned}
$$
where the final inequality follows from $\kappa<\frac{1}{2(p+1)}$. Then for $R\ge 1$, by setting $k=3$, we infer
\begin{equation}
\label{eq: Exterior Nonlinear Bound}
\begin{aligned}
&\int_{|y|\ge R} |\nabla^3 G_\Omega|^2\,dy\lesssim\sum_{i=0}^3\int_{|y|\ge R}\rho^{-2i-3}|\nabla^{3-i}\Psi|^2\,dy\lesssim R^{-1}\Vert\Psi\Vert_{H_3}^2\le R^{-1}\Vert X\Vert_{\mathbb H}^2
\end{aligned}
\end{equation}
where we have used the Hardy's inequality \eqref{eq: Hardy on Psi}. Now we consider the region $0\le\rho\le R$. Denote 
$$
H^3_R:=H^3(B_R(0)) 
$$
Then, there exists $M_1>0$ such that
$$
\Vert \phi\psi\Vert_{H^3_R}^2\le R^{M_1}\Vert \phi\Vert_{H^3_R}^2\Vert \psi\Vert_{H^3_R}^2\quad \forall \phi,\ \psi\in H^3_R
$$
since $3=k>\frac{d}{2}=\frac32$ so that $H^3(\mathbb R^3)$ is an algebra. From \eqref{eq: Nonlinearity} and the assumption $3=k<p-2$ we infer that,
$$
\sum_{m=0}^3\left\Vert\int_0^1(1-\tau)|u_n+\tau\Psi|^{p-m-3}(u_n+\tau\Psi)\,d\tau\right\Vert_{L^\infty(\mathbb R^3)}\lesssim 1.
$$
 Note also that the $L^\infty$-bound \eqref{eq: Bootstrap Bound 2} implies $|\nabla^j\Psi|\lesssim\langle\rho\rangle^{-j-\alpha+\kappa}$ for $0\le j\le 2$ and for all $s\in [s_0,s^*]$. Then it follows from \eqref{eq: Integral Form of Nonlinearity} that
\begin{equation}
\label{eq: Interior Nonlinear Bound}
\begin{aligned}
&\int_{|y|\le R} |\nabla^3 G_\Omega|^2\,dy\le \Vert G_\Omega\Vert_{H^3_R}^2\\
\lesssim &\ R^{2M_1}\Vert \Psi\Vert_{H_R^3}^4\left\Vert \int_0^1(1-\tau)|u_n+\tau\Psi|^{p-3}(u_n+\tau\Psi)\,d\tau\right\Vert_{H_R^3}^2\\
\lesssim & \ R^{2M_1}\Vert \Psi\Vert_{H_R^3}^4\sum_{|\alpha|\le 3}\left\Vert\prod_q(|\nabla^{\alpha_q}u_n|+|\nabla^{\alpha_q}\Psi|)\right\Vert_{L^2(B_R(0))}^2\\
\lesssim& \,R^{2M_1}\Vert\Psi\Vert_{H^3_R}^4 \sum_{m=0}^3(\Vert u_n\Vert_{H^3_R}^2+\Vert\Psi\Vert_{H^3_R}^2)^m\lesssim R^M\Vert X\Vert_{\mathbb H}^4
\end{aligned}
\end{equation}
for some $M>0$. Set $R=\Vert X\Vert_{\mathbb H}^{-\frac{2}{1+M}}$ and add \eqref{eq: Exterior Nonlinear Bound} with \eqref{eq: Interior Nonlinear Bound} so the claim \eqref{eq: Nonlinear Bound 2} follows by choosing $c<\frac{1}{1+M}$.\\\\
By the decay estimate in Corollary \ref{cor: Exponential Decay 1},
\begin{equation}
\label{eq: Stable Evolution}
\begin{aligned}
\Vert (I-P)X(s)\Vert_{\mathbb H}&\lesssim e^{-\frac{\delta}{2}(s-s_0)}\Vert X(s_0)\Vert_{\mathbb H}+\int_{s_0}^s e^{-\frac{\delta}{2}(s-\tau)}\Vert G(\tau)\Vert_{\mathbb H}\,d\tau\\
&\lesssim e^{-\frac{\delta}{2}s}\bigg[e^{\frac{\delta}{2}s_0}\Vert X\Vert_{\mathbb H}+\int_{s_0}^s e^{(\frac{\delta}{2}-\frac{\delta_0}{2}(1+c))\tau}\,d\tau\bigg]\lesssim e^{-\frac{\delta}{2}s}
\end{aligned}
\end{equation}
since $\frac{\delta}{1+c}<\delta_0$. This, together with \eqref{eq: Bootstrap Bound 1}, we infer
$$
\Vert X(s)\Vert_{\mathbb H}\lesssim e^{-\frac{\delta}{2}s}.
$$
This proves an improved bound for \eqref{eq: Bootstrap Bound 4}. Then, by \eqref{eq: Nonlinear Bound 2}, the non-linear bound \eqref{eq: Nonlinear Bound 1} follows.
\end{proof}

\appendix

\section{Bound on self-similar profiles}
In this section, we derive some $\rho\rightarrow \infty$ asymptotic properties of the smooth profiles $u_n$ constructed in Theorem \ref{theo: Result 1}.

\begin{lemma}
\label{lem: Bounds on u_n}
By induction on $k$. Let $u_n$ be the self-similar profiles constructed in \emph{Proposition \ref{prop: Countable Solutions}}. For all $k\in \mathbb N$, as $\rho\rightarrow\infty$,
\begin{equation}
\label{eq: Bounds on u_n}
\partial_\rho^k u_n=\mathcal O(\rho^{-\alpha-k}),\quad \partial_\rho^k(u_n^{p-1})=\mathcal O(\rho^{-2-k}).
\end{equation}
\end{lemma}
\begin{proof}
In view of \eqref{eq: Bounds on w}, taking $\varepsilon \ll1$ we infer
$$
u_n=\mathcal O(\rho^{-\alpha}), \quad u_n'=\mathcal O(\rho^{-\alpha-1})
$$
and $u_n\ge0$ for all $\rho$ sufficiently large. It follows immediately that 
$$
u_n^{p-1}=\mathcal O(\rho^{-2}),\quad (u_n^{p-1})'=(p-1)u_n^{p-1}u_n'=\mathcal O(\rho^{-3}).
$$
In view of \eqref{eq: Self-similar Profile}, we infer
$$
|u_n^{(k)}|\lesssim_{k,\rho} \partial^{k-2}\bigg[\frac{1}{\rho^2}\bigg(\rho u_n'+u_n+u_n^p)\bigg)\bigg]
$$
for all $\rho>\rho_0$ and $k\ge 2$. Suppose lemma holds for some $k\ge2$. Then by hypothesis, for all $\rho>\rho_0$,
$$
|u_n^{(k+1)}|\lesssim \sum_{j=0}^{k} \rho^{-j-2}u_n^{(k-j-1)}+\sum_{j=0}^{k-1} \rho^{-j-2}\sum_{i=0}^{k-j-1}u_n^{(i)}(u_n^{p-1})^{(k-j-i-1)}\lesssim \rho^{-\alpha-k-1}.
$$
Furthermore, by hypothesis and bound on $u_n^{(k+1)}$, we infer
$$
|(u_n^{p-1})^{(k+1)}|\lesssim \sum_{j=0}^{k+1}u_n^{p-k+j-2}\sum_{|\alpha|=k+1,\,\alpha>0} u_n^{(\alpha_1)}\cdots u_n^{(\alpha_j)}\lesssim \rho^{-3-k}
$$
and this concludes the proof by induction.
\end{proof}

\section{Maximality of $\tilde{\mathcal M}$} 
In this section, we consider the problem \eqref{eq: Decomposition of Psi}. Given $H$ such that $H(\rho)Y^{(l,m)}\in C_c^\infty(\mathbb R^3)$, we seek solution to
\begin{equation}
\label{eq: Radial part of Psi}
[\mathcal L-\rho^{-2}\lambda_m-(\Lambda+R+1)(\Lambda+R)+pu_n^{p-1}]\Psi=H.
\end{equation}

\begin{lemma}
\label{lem: Local Existence}
Let $H\in C^\infty([0,\infty))$. Then for $R$ sufficiently large, there exists a unique solution $\Psi\in C^1([0,\infty))$ to \eqref{eq: Radial part of Psi}. Furthermore, if $H(\rho)Y^{(l,m)}\in C_c^\infty(\mathbb R^3)$, then $\Psi(\rho)Y^{(l,m)}$ is smooth on $\mathbb R^3$.
\end{lemma}
\begin{proof}
\textbf{Step 1} (Solutions at $\rho=0$): Set $(\Psi_1,\Psi_2)=(\rho^{m+1}\Psi, \partial_\rho(\rho^{m+1}\Psi))$. Writing \eqref{eq: Radial part of Psi} in the form required in Proposition \ref{prop: Singular ODE}, 
\begin{equation}
\label{eq: Inner Maximality}
\begin{cases}
\rho\,\partial_\rho\Psi_1=\rho\Psi_2\\
\rho\,\partial_\rho\Psi_2=\frac{\rho}{1-\rho^2}\left[\xi-pu_n^{p-1}\right]\Psi_1+\frac{\rho}{1-\rho^2}\left[\frac{2m}{\rho}+\eta\rho\right]\Psi_2+\frac{\rho^{m+2}}{1-\rho^2}H
\end{cases}
\end{equation}
where
$$
\xi=(m-\alpha-R+1)(m-\alpha-R),\quad \eta=-2(m-\alpha-R).
$$
Hence,
$$
\rho\,\partial_\rho\begin{pmatrix}
\Psi_1\\
\Psi_2
\end{pmatrix}
= A(\rho)\begin{pmatrix}
\Psi_1\\
\Psi_2
\end{pmatrix}
+\frac{\rho^{m+2}}{1-\rho^2}\begin{pmatrix}
0\\
H
\end{pmatrix}
$$
where $A$ is smooth in $[0,1)$,
$$
A(0)=\begin{pmatrix}
0 & 0\\
0 & 2m\\
\end{pmatrix}
$$
with $\sigma(A(0))=\{0,2m\}$. Thus, by Proposition \ref{prop: Singular ODE} with $l=2m+1$, we infer for all $a,\,b\in \mathbb R$, there exists a unique smooth solution to the homogeneous problem for \eqref{eq: Inner Maximality} such that
$$
(\Psi_1,\Psi_1',\cdots,\Psi_1^{(2m)},\Psi_1^{(2m+1)})(0)=(a,0,\cdots,0,b).
$$
Since $H(\rho)Y^{(l,m)}$ is smooth radial, $H=\mathcal O_{\rho\rightarrow 0}(\rho^m)$ so from Proposition \ref{prop: Singular ODE} we can write the solution $\Psi_{a,b}$ to \eqref{eq: Inner Maximality} with the boundary condition
$$
(\Psi_{a,b},\Psi_{a,b}',\cdots,\Psi_{a,b}^{(2m)},\Psi_{a,b}^{(2m+1)})(0)=(a,0,\cdots,0,b)
$$
as
$$
\Psi_{a,b}=\Psi_0+a\psi_1+b\psi_2,\quad \begin{cases}
\psi_1(\rho)\propto 1+\mathcal O_{\rho\rightarrow 0}(\rho^{2m+2})\\
\psi_2(\rho)\propto \rho^{2m+1}+\mathcal O_{\rho\rightarrow 0}(\rho^{2m+2})\
\end{cases}
$$
where $\psi_1$, $\psi_2$ are the linearly independent solutions to the homogenous problem for \eqref{eq: Inner Maximality} in $[0,1)$ with appropriate initial values.\\\\
\textbf{Step 2} (Solutions at $\rho=1$): For $(\tilde \Psi_1,\tilde \Psi_2)=(\Psi,\partial_\rho\Psi)$, we write \eqref{eq: Radial part of Psi} as
$$
\begin{cases}
(\rho-1)\partial_\rho \tilde \Psi_1=(\rho-1)\tilde \Psi_2\\
(\rho-1)\partial_\rho \tilde \Psi_2=\frac{-(\alpha+R)(\alpha+R+1)+p u_n^{p-1}-\frac{\lambda_m}{\rho^2}}{1+\rho}\tilde \Psi_1+\frac{\frac{2}{\rho}-2(\alpha+R+1)\rho}{1+\rho}\tilde \Psi_2-\frac{H}{1+\rho}.
\end{cases}
$$
Hence,
$$
(\rho-1)\,\partial_\rho\begin{pmatrix}
\tilde\Psi_1\\
\tilde\Psi_2
\end{pmatrix}
= B(\rho)\begin{pmatrix}
\tilde\Psi_1\\
\tilde\Psi_2
\end{pmatrix}
+\frac{1}{\rho+1}\begin{pmatrix}
0\\
H
\end{pmatrix}
$$
where $B$ is smooth in $(0,\infty)$,
$$
B(1)=\frac{1}{2}\begin{pmatrix}
0 & 0\\
-(\alpha+R)(\alpha+R+1)-\lambda_m+pu_n^{p-1}(1) & 2s_c-2R-3\\
\end{pmatrix}
$$
with $\sigma(B(1))=\{s_c-R-\frac{3}{2},0\}$. Thus, by Proposition \ref{prop: Singular ODE}, for all $b\in\mathbb R$, there exists a unique smooth solution $\tilde \Psi_b\in C^\infty ((0,\infty))$ to \eqref{eq: Radial part of Psi} with 
$$
(\tilde\Psi_c(1), \tilde\Psi_c'(1))=\bigg(2c,-\bigg[\alpha+R+1+\frac{pu_n^{p-1}(1)-\lambda_m}{-s_c+R+\frac{3}{2}}\bigg]c+\frac{H(1)}{-s_c+R+\frac{3}{2}}\bigg).
$$
We can write
$$
\tilde \Psi_c = \tilde \Psi_0+c\,\tilde\psi,\quad (\tilde\psi(1),\tilde\psi'(1))=\bigg(2,-(\alpha+R+1)-\frac{pu_n^{p-1}(1)-\lambda_m}{-s_c+R+\frac{3}{2}}\bigg)
$$
where $\tilde \psi$ is the unique solution to the homogeneous problem for \eqref{eq: Radial part of Psi} in $(0,\infty)$ with the given initial values.\\\\
\textbf{Step 3} (Matching): Next, we claim that for $R$ sufficiently large and for all $m\ge0$, the homogeneous problem for \eqref{eq: Radial part of Psi} with $H=0$ has a unique $C^1$ solution $\Psi\equiv0$ on $[0,1]$. Suppose otherwise i.e. there is $R$ arbitrarily large and $m\ge0$ such that there exists $\Psi_{l,m}\not\equiv 0$ smooth in $[0,1]$ such that \eqref{eq: Radial part of Psi} holds with $H=0$ and $\Psi=\Psi_{l,m}(\rho)Y^{(l,m)}$ is smooth at the origin. Extend uniquely the homogeneous solution $\Psi_{l,m}$ to $[1,\infty)$. Then, using the fixed point argument as in the proof of Lemma \ref{lem: Global Existence} we infer
$$
\sum_{j=0}^{k+3}\sup_{\rho\ge1}\rho^{\alpha+R+j}|\partial_
\rho^j\Psi_{l,m}|<\infty
$$
and therefore, $(\Psi,-(\Lambda+R)\Psi)\in \mathcal D_R$ where we recall the definition \eqref{eq: D_R} of $\mathcal D_R$ and 
$$
\langle \mathcal MX,X\rangle=R\langle X,X\rangle.
$$
By dissipativity of $\tilde{\mathcal M}$ for $X\in \mathcal D_R$ proved in \textbf{Step 1} of the proof of Proposition \ref{prop: Maximal Dissipativity}, we infer for all $X\in \mathcal D_R$
$$
\langle \mathcal M X,X\rangle\le C \langle X,X\rangle
$$
for some $C$ independent of $R$ and this is a contradiction so we have our claim. This yields the uniqueness result.\\\\
Choose $R$ sufficiently large so the claim holds. Since $\{\rho^{-m-1}\psi_1,\rho^{-m-1}\psi_2\}$ is a basis of solutions to the homogeneous problem in $(0,1)$, there exists $A,$ $B\in\mathbb R$ such that
$$
\tilde\psi=\rho^{-m-1}(A\psi_1+B\psi_2)
$$
in $(0,1)$. If $A=0$, then $\tilde \psi\in C^\infty([0,1])$ contradicting the claim above. Since $\{\rho^{-m-1}\psi_1,\rho^{-m-1}\psi_2\}$ is a basis of solutions to the homogeneous problem in $(0,1)$, there exists $a,$ $b\in\mathbb R$ such that 
$$
\rho^{-m-1}\Psi_{a,b}=\tilde \Psi_0
$$
Then, 
$$
\Psi=\tilde \Psi_0-\frac{a}{A}\tilde\psi=\rho^{-m-1}\bigg(\Psi_{a,b}-a\psi_1-\frac{aB}{A}\psi_2\bigg)
$$
is smooth at $\rho=0$ by the first equality and is smooth at $\rho=1$ by the second equality. Thus, we have the existence and uniqueness of $C^1([0,\infty))$ solution. Furthermore, if $H(\rho)Y^{(l,m)}$ is smooth i.e. $H=\mathcal O_{\rho\rightarrow 0}(\rho^m)$ and $H^{(m+2k+1)}(0)=0$ for $k\in\mathbb N_{\ge 0}$, then it follows that $\Psi^{(m+2k+1)}(0)=0$ for $k\in\mathbb N_{\ge 0}$. Thus, $\Psi(\rho)Y^{(l,m)}$ is smooth.
\end{proof}

\begin{lemma}
\label{lem: Global Existence}
For $H$ such that $H(\rho)Y^{(l,m)}\in C_c^\infty (\mathbb R^3)$, let $\Psi$ be the unique $C^1$ solution to \eqref{eq: Radial part of Psi} found in \emph{Lemma \ref{lem: Local Existence}}. Then for $R$ sufficiently large, $\Psi(\rho)Y^{(l,m)}\in H^{k+1}(\mathbb R^3)$.
\end{lemma}
\begin{proof}
Using the fixed point argument, we prove the existence of $C^{k+1}$ solution $\Psi$ to \eqref{eq: Radial part of Psi} in $\{\rho\ge\rho_0\}$ for $\rho_0$ sufficiently large with sufficiently rapid decay as $\rho\rightarrow\infty$ so that $\Psi\in H_{\rad}^{k+1}(\{\rho\ge\rho_0\})$. Then by uniqueness of solution, we argue that this solution is indeed what we found in Lemma \ref{lem: Local Existence}.\\\\
Consider the homogeneous problem for \eqref{eq: Radial part of Psi} without the $pu_n^{p-1}$ potential term:
\begin{equation}
\label{eq: Linear Wave}
\underbrace{\Big\{(1-\rho^2)\partial_\rho^2+[2\rho^{-1}-2(\alpha+R+1)\rho]\partial_\rho -\lambda_m\rho^{-2}-(\alpha+R)(\alpha+R+1)\Big\}}_{:=\mathcal L_R}\varphi=0
\end{equation}
in $[1,\infty)$. Computation similar to Lemma \ref{lem: Fundamental Solutions} yields a pair of linearly independent solutions
\begin{equation}
\label{eq: Fundamental Solutions for Linear Wave}
\begin{aligned}
\varphi_1&=\rho^{-\alpha-R-1}\,_2F_1\bigg(\frac{\alpha+R+m+1}{2},\frac{\alpha+R-m}{2},\frac{3}{2},\rho^{-2}\bigg)\\
\varphi_2&=\rho^{-\alpha-R}\,_2F_1\bigg(\frac{\alpha+R+m}{2},\frac{\alpha+R-m-1}{2},\frac{1}{2},\rho^{-2}\bigg)
\end{aligned}
\end{equation}
with the Wronskian
$$
W:=\varphi_1'\varphi_2-\varphi_2'\varphi_1\propto\rho^{-2}|1-\rho^2|^{s_c-R-\frac{3}{2}}.
$$
Define the spaces
$$
\begin{aligned}
\bar X_{\rho_0}&=\bigg\{w\in C^{k+1}((\rho_0,\infty))\,\bigg|\,\Vert w\Vert_{\bar X_{\rho_0}}:=\sum_{j=0}^{k+1}\,\sup_{\rho\ge\rho_0}\rho^{\alpha+R+j}|\partial_\rho^j w|\bigg\},\\
\bar Y_{\rho_0}&=\bigg\{w\in C^{k+1}((\rho_0,\infty))\,\bigg|\,\Vert w\Vert_{\bar Y_{\rho_0}}:=\sum_{j=0}^{k+1}\,\sup_{\rho\ge\rho_0}\rho^{\alpha+R+j+2}|\partial_\rho^jw|\bigg\}.
\end{aligned}
$$
Claim that for $\rho_0>1$, the resolvent map $\mathcal T_R: \bar Y_{\rho_0}\rightarrow \bar X_{\rho_0}$ given by
$$
\mathcal T_R(f)=\varphi_1 \int_{\rho_0}^\rho \frac{f\varphi_2}{(1-r^2)W}\,dr-\varphi_2\int_{\rho_0}^\rho \frac{f\varphi_1}{(1-r^2)W}\,dr
$$
is well-defined and bounded with $\mathcal L_R \circ \mathcal T_R=\id_{\bar Y_{\rho_0}}$. Note that 
$$
\begin{aligned}
\partial_\rho^j\mathcal T_R(f)&=\varphi_1^{(j)} \int_{\rho_0}^\rho \frac{f\varphi_2}{(1-r^2)W}\,dr-\varphi_2^{(j)} \int_{\rho_0}^\rho \frac{f\varphi_1}{(1-r^2)W}\,dr\\
&+\sum_{i=0}^{j-2}\partial_\rho^i\bigg[\frac{f(\varphi_1^{(j-i-1)}\varphi_2-\varphi_2^{(j-i-1)}\varphi_1)}{(1-\rho^2)W}\bigg].
\end{aligned}
$$
In view of \eqref{eq: Fundamental Solutions for Linear Wave} and the asymptotic expansion of the fundamental solutions, we infer
$$
\partial_\rho^l\bigg[\frac{\varphi_1^{(j-i-1)}\varphi_2-\varphi_2^{(j-i-1)}\varphi_1}{(1-\rho^2)W}\bigg]=\mathcal O_{\rho\rightarrow\infty}(\rho^{i-j-l}).
$$
Then for all $\rho\ge \rho_0$ and $0\le j\le k+1$,
$$
\begin{aligned}
\rho^{\alpha+R+j}|\partial_\rho^j\mathcal T_R(f)|&\lesssim\bigg( \rho^{-1}\int_{\rho_0}^\rho\rho^{-2}\,d\rho\bigg)\sup_{r\ge\rho_0}r^{\alpha+R+2}|f|+\bigg(\int_{\rho_0}^\rho\rho^{-3}\,d\rho\bigg)\sup_{r\ge\rho_0}r^{\alpha+R+2}|f|\\
&+\sum_{i=0}^{j-2}\rho^{i-j}\bigg(\rho^{j-i-2}\sup_{r\ge\rho_0}r^{\alpha+R+i+2}|\partial_\rho^i f|\bigg)\lesssim \rho_0^{-2}\Vert f\Vert_{\bar Y_{\rho_0}}.
\end{aligned}
$$
Thus, $\mathcal T_R$ is a bounded map with operator norm $\Vert \mathcal T_R\Vert \lesssim \rho_0^{-2}$ as claimed. Now we solve the fixed point problem:
\begin{equation}
\label{eq: Fixed Point for Psi}
\Psi =\underbrace{c_1\varphi_1+c_2\varphi_2+ \mathcal T_R[H-u_n^{p-1}\Psi]}_{:=G_R(\Psi)}
\end{equation}
for $c_1$, $c_2$ such that the $\Psi(\rho_0)$, $\Psi'(\rho_0)$ agree with the corresponding values of the unique solution in Lemma \ref{lem: Local Existence}. Note that $\varphi_1$, $\varphi_2\in \bar X_{\rho_0}$, $H\in C_c^\infty([0,\infty))$. By Lemma \ref{lem: Bounds on u_n}, $\partial_\rho^j(u_n^{p-1})=\mathcal O(\rho^{-m-2})$ as $\rho\rightarrow \infty$ so we infer 
$$
\Vert u_n^{p-1}\Psi\Vert_{\bar Y_{\rho_0}}\lesssim \Vert\Psi\Vert_{\bar X_{\rho_0}}
$$
and hence, $H-u_n^{p-1}\Psi\in \bar Y_{\rho_0}$ so indeed $G_R:\bar X_{\rho_0}\rightarrow \bar X_{\rho_0}$. For $\rho_0$ sufficiently large, $G_R$ is a contraction map since for all $\Psi_1,\Psi_2\in\bar X_{\rho_0}$,
$$
\Vert G_R(\Psi_1)-G_R(\Psi_2)\Vert_{\bar X_{\rho_0}}\lesssim \Vert\mathcal T_R\Vert\, \Vert u_n^{p-1}(\Psi_1-\Psi_2)\Vert_{\bar Y_{\rho_0}}\lesssim \rho_0^{-2}\Vert \Psi_1-\Psi_2\Vert_{\bar X_{\rho_0}}.
$$
Thus, it follows from the Banach fixed point theorem that there exists a unique $\Psi\in \bar X_{\rho_0}$ such that \eqref{eq: Fixed Point for Psi} holds. Taking $R>s_c$, $\bar X_{\rho_0}$ continuously embeds in $H_{\rad}^{k+1}(\{\rho\ge\rho_0\})$ so $\Psi\in H_{\rad}^{k+1}(\{\rho\ge\rho_0\})$. Also, by uniqueness of solution to an ODE at ordinary point, this is indeed the solution we found in   Lemma \ref{lem: Local Existence}.
\end{proof}

\section{Behaviour of the Sobolev norm}
In this section, we prove the asymptotic behaviours \eqref{eq: Subcritical Sobolev}, \eqref{eq: Critical Sobolev} and \eqref{eq: Supercritical Sobolev} of the Sobolev norms of the blow up solutions. In this section we denote by $\tau$ the self-similar time in order to distinguish from the Sobolev exponent $s$.

\begin{proof}[proof of \eqref{eq: Subcritical Sobolev}, \eqref{eq: Critical Sobolev}, \eqref{eq: Supercritical Sobolev}]
Suppose that $(\Phi,\Phi_t)$ is a blow up solution as in the statement of Theorem \ref{theo: Result 2}. Then, the bootstrap bounds in Proposition \ref{prop: Bootstrap} are satisfied in the region $\tau\in[s_0,\infty)$ in the self-similar time. In particular, from \eqref{eq: Improved Bound 1}, we have that
\begin{equation}
\label{eq: Energy Bound}
\int_{|y|<e^\tau} \langle\rho\rangle^{2j-2s_c}\left(|\nabla^j\Psi|^2+\mathbbm{1}_{j\ge1}|\nabla^{j-1}\Omega|^2\right)\,dy< d_0,\quad 0\le j\le 4,
\end{equation}
and from \eqref{eq: Bootstrap Bound 4},
\begin{equation}
\label{eq: Accretivity Norm}
\int_{\mathbb R^3} \left(|\nabla^4\Psi|^2+|\nabla^3\Omega|^2\right)dy<e^{-\frac{\delta_0}{2} \tau}.
\end{equation}
Recall the definition of dampened profile $u_n^D$ and perturbation $\Psi^D$ from Section \ref{sec: Bootstrap}. From \eqref{eq: Improved Bound 1} with $j=0$, we infer
$$
\Vert \Phi\Vert_{L^2(|x|>1)}^2=\left\Vert \frac{1}{(T-t)^\alpha}\tilde\Psi\left(\frac{r}{T-t}\right)\right\Vert_{L^2(|x|>1)}^2\lesssim \int_{|y|\ge e^\tau} \rho^{-2s_c}\xi(e^{-\tau}\rho)^{2n_p+1}|\tilde\Psi|^2\,dy< d_0
$$
where we have used that $\xi(r)\gtrsim r$ for $r\ge1$ and that $s_c<n_p$. Similarly, set $j=2$ in \eqref{eq: Improved Bound 1},
$$
\Vert \Phi\Vert_{\dot H^2(|x|>1)}^2=\left\Vert \frac{1}{(T-t)^\alpha}\tilde\Psi\left(\frac{r}{T-t}\right)\right\Vert_{\dot H^2(|x|>1)}^2\lesssim \int_{|y|\ge e^\tau} \rho^{-2(s_c-2)}\xi(e^{-\tau}\rho)^{2n_p+1}|\Delta\tilde\Psi|^2\,dy< d_0.
$$
We interpolate the above two bounds and infer
\begin{equation}
\label{eq: Exterior Bound}
\Vert\Phi\Vert_{\dot H^s(|x|>1)}^2\lesssim d_0,\quad 0\le s\le 2
\end{equation}
\\
{\bf Step 1} ($H^{s_c}$ Bound): In view of the Gagliardo-Nierenberg inequality (see \cite{DDS}), we infer the $H^{s_c}$ bound on $\Psi^D$:
$$
\begin{aligned}
&\left\Vert \frac{1}{(T-t)^\alpha}\nabla_r^{s_c}\Psi^D\left(\frac{r}{T-t}\right)\right\Vert_{L^2(\mathbb R^3)}^2=\int_{\mathbb R^3}|\nabla^{s_c}\Psi^D|^2dy\\
\lesssim& \left(\int_{\mathbb R^3} \langle\rho\rangle^{2(1-s_c)}|\nabla\Psi^D|^2\,dy \right)^\theta \left(\int_{\mathbb R^3} \langle\rho\rangle^{2(2-s_c)}|\Delta\Psi^D|^2\,dy\right)^{1-\theta}
\end{aligned}
$$
where
$$
s_c=\theta+2(1-\theta), \quad \theta\in(0,1).
$$
Thus, from \eqref{eq: Improved Bound 1}, we infer
\begin{equation}
\label{eq: Critical Norm Psi}
\left\Vert \frac{1}{(T-t)^\alpha}\nabla_r^{s_c}\Psi^D\left(\frac{r}{T-t}\right)\right\Vert_{L^2(\mathbb R^3)}^2=\int_{\mathbb R^3}|\nabla^{s_c}\Psi^D|^2dy\lesssim d_0.
\end{equation}
Also, note that for $s\le s_c$,
\begin{equation}
\label{eq: Subcritical Norm Profile}
\begin{aligned}
&\left\Vert  \frac{1}{(T-t)^\alpha}\nabla_r^s u_n\left(\frac{r}{T-t}\right)\right\Vert_{L^2(|x|<1)}^2=e^{-2(s_c-s)\tau}\int_{|y|\le e^\tau}|\nabla^su_n|^2\,dy\\
\sim\, &c_{n,s} e^{-2(s_c-s)\tau}\int_1^{e^\tau}\rho^{2(s_c-s)-1}d\rho\sim\ c_{n,s} \begin{cases} 1 & s<s_c,\\ \tau & s=s_c. \end{cases}
\end{aligned}
\end{equation}
Above inequalities, together with \eqref{eq: Exterior Bound} with $s=s_c$ we infer,
$$
\Vert \Phi\Vert_{\dot H^{s_c}}^2=c_n(1+o_{t\rightarrow T}(1))|\log(T-t)|, 
$$
Similarly for $\Phi_t$. Hence, we infer \eqref{eq: Critical Sobolev}.\\\\
{\bf Step 2} (Subcritical Bound): Set $j=0$ in \eqref{eq: Energy Bound}, we have the $L^2$ bound on $\Psi$:
$$
\begin{aligned}
\left\Vert \frac{1}{(T-t)^\alpha}\Psi\left(\frac{r}{T-t}\right)\right\Vert_{L^2(|x|<1)}^2\le&\int_{|x|<1} e^{2s_c\tau}\langle e^\tau r\rangle^{-2s_c}\left|\frac{1}{(T-t)^\alpha}\Psi\left(\frac{r}{T-t}\right)\right|^2dx\\
=&\int_{|y|<e^\tau} \langle\rho\rangle^{-2s_c}|\Psi|^2\,dy< d_0.
\end{aligned}
$$
This, together with \eqref{eq: Exterior Bound} we infer
$$
\left\Vert \frac{1}{(T-t)^\alpha}\Psi^D\left(\frac{r}{T-t}\right)\right\Vert_{L^2(\mathbb R^3)}^2\lesssim d_0.
$$
Interpolate with the critical norm \eqref{eq: Critical Norm Psi} above, we have for $0\le s<s_c$,
$$
\left\Vert \frac{1}{(T-t)^\alpha}\nabla_r^s\Psi^D\left(\frac{r}{T-t}\right)\right\Vert_{L^2(\mathbb R^3)}^2\lesssim d_0.
$$
Adding with the norm of the dampened profile \eqref{eq: Subcritical Norm Profile}, we infer
$$
\limsup_{t\rightarrow T} \Vert\Phi\Vert_{\dot H^s}^2<\infty.
$$
Similarly for $\Phi_t$. Hence, we infer \eqref{eq: Subcritical Sobolev}.\\\\
{\bf Step 3} (Supercritical Bound): Since
$$
\int_{|y|\ge e^\tau} |\nabla^4(u_n-u_n^D)|^2dy\lesssim e^{-2(4-s_c)\tau},
$$
it follows from \eqref{eq: Accretivity Norm} that
$$
\int_{\mathbb R^3}|\nabla^4\Psi^D|^2dy\lesssim e^{-\frac{\delta_0}{2}\tau}+e^{-2(4-s_c)\tau}.
$$
We interpolate this with \eqref{eq: Critical Norm Psi} and infer for $s_c< s\le 2$
$$
\begin{aligned}
\Vert \Psi\Vert_{\dot H^s}^2\le&\int_{|y|\ge e^\tau} |\nabla^s(u_n-u_n^D)|^2dy+\int_{\mathbb R^3}|\nabla^s\Psi^D|^2dy\\
\lesssim&_s \,e^{-2(s-s_c)\tau}+ e^{-c_s \tau}\rightarrow 0.
\end{aligned}
$$
Similarly for $\Omega$. Hence, we infer \eqref{eq: Supercritical Sobolev}.
\end{proof}

\section{Lipschitz dependence of initial data}

Recall from Section \ref{sec: Bootstrap} the definition of the projection operator $P$ onto $V$ the subspace of unstable directions under semigroup action of maximally dissipative operator $\mathcal M-\mathcal P$. In the Proof \ref{proof: Result 2} and Corollary \ref{cor: Exponential Decay 2}, it is proved that for any small intitial perturbation in the stable direction:
$$
\Vert (I-P)X(s_0)\Vert_{\mathbb H}\le e^{-\frac{\delta}{2}s_0},
$$
there exists a choice of $PX(s_0)$ so that the solution is global in self-similar time with
$$
\Vert PX(s)\Vert_{\mathbb H}\le e^{-\frac{\delta}{2}(1+\frac{c}{2})s}\quad s\ge s_0.
$$
In this section, we prove that the choice of $PX(s_0)$ is unique and is Lipschitz dependent on $(I-P)X(s_0)$. In particular, we show that for any two global solutions $X$ and $\bar X$, if the initial difference in the unstable direction is too big compared to the initial differences in the stable direction, the unstable linear dynamics wins and expels the differences of unstable parameters away from $0$. Hence one of the two solutions cannot blow up according to our scenario, yielding a contradiction. In particular, we claim the following:

\begin{lemma}
\label{lem: Lipschitz}
Let us assume $X$ and $\bar X$ are two global solutions as in Proposition \ref{prop: Bootstrap} i.e. there holds the initial condition \eqref{eq: Choice of Initial Value 1}, and the bootstrap bounds \eqref{eq: Bootstrap Bound 2}, \eqref{eq: Bootstrap Bound 4} for $s\ge s_0$. Denote by 
$$
 X_s=(I-P)X,\quad X_u=PX
$$
the stable and unstable part of the perturbation and similarly $\bar X_u$, $\bar X_s$. Then, for $s_0\gg1$ sufficiently large,
\begin{equation}
\label{eq: Lipschitz}
\Vert \triangle X_u(s_0)\Vert_{\mathbb H}\le c_{s_0} \Vert \triangle X_s(s_0)\Vert_{\mathbb H}
\end{equation}
where $\triangle X_u=X_u-\bar X_u$,  $\triangle X_s=X_s-\bar X_s$.
\end{lemma}
\begin{proof}
{\bf Step 1} (Difference of nonlinear term): Recall \eqref{eq: Linearized System} and define $\triangle G=G-\bar G$. Then,
$$
\begin{aligned}
\triangle G_\Omega =&-|\Psi+u_n|^{p-1}(\Psi+u_n)+|\bar\Psi+u_n|^{p-1}(\bar\Psi+u_n)+pu_n^{p-1} \triangle \Psi\\
=&\ p\triangle \Psi \left(u_n^{p-1}-\int_0^1 |u_n+\bar \Psi+\tau \triangle\Psi|^{p-1}\,d\tau \right)
\end{aligned}
$$
We claim the following nonlinear bound: there exists $c>0$ such that
$$
\Vert \triangle G(s)\Vert_{\mathbb H}\lesssim e^{-\frac{c\delta}{2}s}\Vert \triangle X(s)\Vert_{\mathbb H}.
$$
This is an analogue of \eqref{eq: Nonlinear Bound 2} for the difference $\triangle  X$.\\\\
Let $\rho\ge1$. Note that for $m\le k=3< p-1$,
\begin{equation}
\label{eq: Difference Nonlinearity}
\int_0^1 |u_n+\bar \Psi+\tau \triangle\Psi|^{p-m-1}\,d\tau\lesssim \sup_{\tau\in [0,1]}|u_n+\bar\Psi+\tau \triangle \Psi|^{p-m-1}
\end{equation}
Thus, using \eqref{eq: Bootstrap Bound 2} and following the similar steps as in \eqref{eq: Nonlinearity Expansion} we infer
$$
\begin{aligned}
|\nabla^k \triangle G_\Omega|\lesssim& \sum_{j+l=k}|\nabla^j\triangle\Psi| \sum_{m=0}^l\rho^{(p-m-1)(-\alpha+\kappa)}\sum_{|\alpha|=l}\,\prod_{q=1}^m(|\nabla^{\alpha_q}(u_n+\bar\Psi)|+|\nabla^{\alpha_q}\triangle\Psi|)\\
\lesssim&\sum_{j+l=k}|\nabla^j\triangle\Psi|\sum_{m=0}^l\rho^{(p-m-1)(-\alpha+\kappa)}\rho^{m(-\alpha+\kappa)-l}\lesssim \sum_{i=0}^k \rho^{i-k-\frac{3}{2}}|\nabla^i\triangle\Psi|.
\end{aligned}
$$
where in the last inequality, we have used that $\kappa<\frac{1}{2(p+1)}$. Then for $R\ge1$, we infer
\begin{equation}
\label{eq: Difference Nonlinearity Bound}
\int_{|y|\ge R} |\nabla^3\triangle G_\Omega|^2\,dy\lesssim\sum_{i=0}^3\int_{|y|\ge R}\rho^{-2i-3}|\triangle\nabla^{3-i}\Psi|^2\,dy\lesssim R^{-1}\Vert  \triangle \Psi \Vert_{H_4}^2\le R^{-1} \Vert \triangle X\Vert_{\mathbb H}^2
\end{equation}
where we have used the Hardy's inequality. We now bound $\triangle G_\Omega$ in the region $\rho\le R$. We rewrite
$$
\triangle G_\Omega=-p(p-1)\triangle \Psi \int_0^1\int_0^1 (\bar\Psi+\tau\triangle \Psi)|u_n+\tau'(\bar \Psi+\tau \triangle\Psi)|^{p-3}(u_n+\tau'(\bar \Psi+\tau \triangle\Psi))\,d\tau d\tau'.
$$
Note that for $m\le 3<p-2$,
$$
\int_0^1 |u_n+\tau'(\bar \Psi+\tau \triangle\Psi)|^{p-m-2}\,d\tau'\lesssim \sup_{0\le\tau\le1}|u_n+\tau(\bar\Psi+\tau\triangle \Psi)|^{p-m-2}
$$
Thus, we infer from the assumption $k=3<p-1$ that
$$
\sum_{m=0}^3\left\Vert\int_0^1 |u_n+\tau'(\bar \Psi+\tau \triangle\Psi)|^{p-m-2}\,d\tau'\right\Vert_{L^\infty(\mathbb R^3)}\lesssim 1.
$$
Then following the similar steps as in \eqref{eq: Interior Nonlinear Bound} by exploiting the algebra structure of the Sobolev space $H_R^3$, we bound the nonlinear difference in the region $0\le \rho\le R$:
$$
\begin{aligned}
&\int_{|y|\le R} |\nabla^3 \triangle G_\Omega|^2\,dy\le \Vert \triangle G_\Omega\Vert_{H^3_R}^2\\
\lesssim& \,R^{2M_1}\Vert\triangle\Psi\Vert_{H^3_R}^2(\Vert\Psi\Vert_{H^3_R}^2+\Vert\bar\Psi\Vert_{H^3_R}^2) \sum_{m=0}^3(\Vert u_n\Vert_{H^3_R}^2+\Vert\Psi\Vert_{H^3_R}^2+\Vert\bar\Psi\Vert_{H^3_R}^2)^m\\
\lesssim &R^M\Vert\triangle X\Vert_{\mathbb H}^2(\Vert X\Vert_{\mathbb H}^2+\Vert\bar X\Vert_{\mathbb H}^2)\lesssim R^M e^{-\delta s}\Vert\triangle X\Vert_{\mathbb H}^2
\end{aligned}
$$
for some $M>0$. Note that the final inequality follows from \eqref{eq: Bootstrap Bound 4}. Set $R=e^{\frac{\delta s}{1+M}}$ and add \eqref{eq: Exterior Nonlinear Bound} with \eqref{eq: Interior Nonlinear Bound} so the claim \eqref{eq: Nonlinear Bound 2} follows by choosing $c<\frac{1}{1+M}$.\\\\
{\bf Step 2} (Bound on initial perturbation): Recall that in the decomposition
$$
\mathbb H=U\oplus V,
$$
we have for all $\lambda\in \sigma(\mathcal M-\mathcal P)|_V$, that $\Re(\lambda)\ge0$. Then, without loss of generality, restrict to an irreducible subspace so that for $\Re(\lambda)\ge0$, we write $A:=\mathcal M-\mathcal P$ as in \eqref{eq: Jordan Block}. Then, from Duhamel's formula, \eqref{eq: Linearized System} implies
$$
e^{(s_0-s)A}\triangle X_u(s)=\triangle X_u(s_0)+\int_{s_0}^\infty e^{(s_0-\tau)A} \triangle G_u(\tau)\,d\tau-\int_s^\infty e^{(s_0-\tau)A} \triangle G_u(\tau)\,d\tau
$$
where $G_u=P G(s)$ and $\triangle G_u= G_u-\bar G_u$. Also, from \eqref{eq: Difference Nonlinearity Bound}, we bound
$$
\begin{aligned}
&\left\Vert e^{(s_0-s)A}\triangle X_u(s)+\int_s^\infty e^{(s_0-\tau)A} \triangle G_u(\tau)\,d\tau\right\Vert_{\mathbb H}\\
\lesssim& (s-s_0)^{m_\lambda-1}e^{-\Re(\lambda)(s-s_0)} \Vert \triangle X_u(s)\Vert_{\mathbb H}+ \int_s^\infty (\tau-s_0)^{m_\lambda-1}e^{-\Re(\lambda)(\tau-s_0)}\Vert \triangle G_u\Vert_{\mathbb H}\,d\tau\rightarrow 0
\end{aligned}
$$
since we have exponential decay of $X$, $\bar X$ from \eqref{eq: Bootstrap Bound 4}  and of $G$, $\bar G$ from \eqref{eq: Nonlinear Bound 1}. Thus, for all $s\ge s_0$,
\begin{equation}
\label{eq: Initial Bound}
\Vert \triangle X_u(s)\Vert_{\mathbb H}=\left\Vert \int_s^\infty e^{(s-\tau)A} \triangle G_u(\tau)\,d\tau\right\Vert_{\mathbb H}\le \int_{s_0}^\infty \Vert \triangle G_u(\tau)\Vert_{\mathbb H}\,d\tau \le \int_{s_0}^\infty e^{-\frac{c\delta}{2}\tau}\Vert \triangle X(\tau)\Vert_{\mathbb H}\,d\tau.
\end{equation}
Now, consider the evolution in the stable subspace $U$ where $A$ is dissipative so Corollary \ref{cor: Exponential Decay 1} applies. Again, from Duhamel's formula,
$$
\triangle X_s(s)=e^{(s-s_0)A}\triangle X_s(s_0)+\int_{s_0}^s e^{(s-\tau)A} \triangle G_s(\tau)\,d\tau,
$$
so we bound for all $s\ge s_0$:
$$
\Vert \triangle X_s(s)\Vert_{\mathbb H}\le\Vert \triangle X_s(s_0)\Vert_{\mathbb H} +\int_{s_0}^s \Vert \triangle G_u(\tau)\Vert_{\mathbb H}\,d\tau \le \Vert \triangle X_s(s_0)\Vert_{\mathbb H}+\int_{s_0}^s e^{-\frac{c\delta}{2}\tau}\Vert \triangle X(\tau)\Vert_{\mathbb H}\,d\tau.
$$
Takinge supermum over $s$,
$$
\begin{aligned}
\Vert \triangle X_s\Vert_{\mathbb H, L_s^\infty} \le& \Vert \triangle X_s(s_0)\Vert_{\mathbb H}+(\Vert \triangle X_s\Vert_{\mathbb H, L_s^\infty}+\Vert \triangle X_u\Vert_{\mathbb H, L_s^\infty})\int_{s_0}^\infty e^{-\frac{c\delta}{2}\tau}d\tau\\
\lesssim &\Vert \triangle X_s(s_0)\Vert_{\mathbb H}+\Vert \triangle X_u\Vert_{\mathbb H, L_s^\infty}.
\end{aligned}
$$
where in the last inequality, we absorb the $\triangle X_s$ on the RHS by taking a large $s_0$. Thus, from \eqref{eq: Initial Bound},
$$
\begin{aligned}
\Vert \triangle X_u\Vert_{\mathbb H,L_s^\infty}&\le \int_{s_0}^\infty e^{-\frac{c\delta}{2}\tau}(\Vert \triangle X_s\Vert_{\mathbb H,L_s^\infty}+\Vert \triangle X_u\Vert_{\mathbb H, L_s^\infty})\,d\tau\\
&\lesssim e^{-\frac{c\delta}{2}s_0}(\Vert \triangle X_s(s_0)\Vert_{\mathbb H}+\Vert \triangle X_u\Vert_{\mathbb H, L_s^\infty})\lesssim \Vert \triangle X_s(s_0)\Vert_{\mathbb H}.
\end{aligned}
$$
Again absorb the $\triangle X_u$ term by taking a large $s_0$. Thus, we infer \eqref{eq: Lipschitz}.
\end{proof}


\bibliographystyle{acm}

\end{document}